\documentclass[a4paper,10pt]{amsart}
\usepackage[utf8]{inputenc}
\usepackage{hyperref}
\usepackage{amssymb,amsmath,dsfont}
\usepackage[all]{xy}
\usepackage{upgreek}

\newtheorem{theorem}{Theorem}[section]

\newtheorem{definition}[theorem]{Definition}
\newtheorem{convention}[theorem]{Convention}

\newtheorem{example}[theorem]{Example}
\newtheorem{proposition}[theorem]{Proposition}
\newtheorem{lemma}[theorem]{Lemma}
\newtheorem{remark}[theorem]{Remark}

\newtheorem{corollary}[theorem]{Corollary}

\newcommand{\BC}{\pt/\mathbb{C}^*}

\newcommand{\dd}{\mathbf{d}}
\newcommand{\DD}{\mathbb{D}}

\newcommand{\ee}{\mathbf{e}}
\newcommand{\mm}{\mathbf{m}}

\newcommand{\ff}{\mathbf{f}}
\newcommand{\Boxtimes}{\mbox{\larger[3]{$\boxtimes$}}}

\DeclareMathOperator{\HN}{HN}
\DeclareMathOperator{\Z}{Z}
\DeclareMathOperator{\cyc}{cyc}

\newcommand{\LL}{\mathbb{L}}
\DeclareMathOperator{\DIM}{dim}

\newcommand{\QQ}{\mathbb{Q}}

\DeclareMathOperator{\PpP}{P}
\DeclareMathOperator{\Simp}{S}
\DeclareMathOperator{\Top}{top}
\DeclareMathOperator{\sfr}{-sfr}

\newcommand{\WW}{\mathcal{T}r(W)}
\newcommand{\WWW}{\mathfrak{Tr}(W)}

\newcommand{\ZZ}{\mathbb{Z}}
\newcommand{\Mst}{\mathfrak{M}}

\newcommand{\Msp}{\mathcal{M}}

\newcommand{\ICS}{\mathcal{IC}}

\newcommand{\phim}[1]{\phi^{\mon}_{#1}}
\newcommand{\shave}{\overline{\Simp}_{\Ss}}

\renewcommand{\theta}{\uptheta}
\renewcommand{\alpha}{\upalpha}
\renewcommand{\beta}{\upbeta}
\renewcommand{\gamma}{\upgamma}
\renewcommand{\delta}{\updelta}
\renewcommand{\zeta}{\upzeta}
\renewcommand{\pi}{\uppi\hspace{0.05em}}
\renewcommand{\xi}{\upxi}
\renewcommand{\chi}{\upchi}
\renewcommand{\sigma}{\upsigma}
\renewcommand{\Lambda}{\Uplambda}
\renewcommand{\Gamma}{\Upgamma}
\renewcommand{\phi}{\upphi}
\renewcommand{\psi}{\uppsi}
\renewcommand{\nu}{\upnu}
\renewcommand{\tau}{\uptau}
\renewcommand{\mu}{\upmu}
\renewcommand{\eta}{\upeta}

\DeclareMathOperator{\IIm}{\mathcal{I}m}
\DeclareMathOperator{\RRe}{\mathcal{R}e}
\DeclareMathOperator{\sst}{-ss}
\DeclareMathOperator{\nilp}{nilp}
\DeclareMathOperator{\KK}{K_0}

\DeclareMathOperator{\Dim}{Dim}

\DeclareMathOperator{\surj}{surj}

\DeclareMathOperator{\st}{-st}

\DeclareMathOperator{\tw}{tw}
\DeclareMathOperator{\reg}{reg}
\DeclareMathOperator{\Tor}{T}
\DeclareMathOperator{\Ror}{A}
\DeclareMathOperator{\Hodge}{Hodge}
\DeclareMathOperator{\For}{F}
\DeclareMathOperator{\crit}{crit}

\DeclareMathOperator{\CTens}{\hat{T}}

\DeclareMathOperator{\Hom}{Hom}

\DeclareMathOperator{\mon}{mon}
\DeclareMathOperator{\Mu}{\Upxi}
\DeclareMathOperator{\End}{End}
\DeclareMathOperator{\Ext}{Ext}
\DeclareMathOperator{\MMHM}{MMHM}

\DeclareMathOperator{\fd}{\mathbf{f.d}}

\DeclareMathOperator{\Gr}{Gr}

\DeclareMathOperator{\tmod}{mod-}
\DeclareMathOperator{\Tmod}{Mod-}
\newcommand{\rmod}[1]{\tmod\!#1}

\newcommand{\Rmod}[1]{\Tmod\!#1}

\DeclareMathOperator{\supp}{supp}

\DeclareMathOperator{\Aut}{Aut}

\DeclareMathOperator{\codim}{codim}
\DeclareMathOperator{\Perv}{Perv}
\DeclareMathOperator{\MHM}{MHM}
\DeclareMathOperator{\IC}{IC}
\DeclareMathOperator{\DT}{DT}
\DeclareMathOperator{\Sym}{Sym}

\DeclareMathOperator{\Spec}{Spec}
\DeclareMathOperator{\Gl}{GL}

\DeclareMathOperator{\id}{id}
\DeclareMathOperator{\Jac}{Jac}
\DeclareMathOperator{\HJac}{\widehat{Jac}}
\DeclareMathOperator{\Perf}{Perf}
\DeclareMathOperator{\HGamma}{\widehat{\Gamma}}

\DeclareMathOperator{\Tr}{Tr}

\DeclareMathOperator{\pt}{pt}

\DeclareMathOperator{\EXP}{EXP}

\DeclareMathOperator{\princ}{princ}

\DeclareMathOperator{\Ss}{\overline{s}}

\newcommand{\pleth}[1]{\EXP\left(#1\right)}
\DeclareMathOperator{\vir}{vir}

\DeclareMathOperator{\Ho}{\mathcal{H}}
\DeclareMathOperator{\HO}{H}

\DeclareMathOperator{\Dulf}{\mathcal{D}^{\geq, lf}}
\DeclareMathOperator{\Dllf}{\mathcal{D}^{\leq, lf}}
\newcommand{\Db}{\mathcal{D}^{b}}
\newcommand{\Dub}{\mathcal{D}}

\begin{document}
\title[Positivity for quantum cluster algebras]{Positivity for quantum cluster algebras}
\author{Ben Davison}

\maketitle
\vspace{-0.2in}
\begin{center} 
{\it In memory of Kentaro Nagao}
\end{center}
\begin{abstract}
Building on work by Kontsevich, Soibelman, Nagao and Efimov, we prove the positivity of quantum cluster coefficients for all skew--symmetric quantum cluster algebras, via a proof of a conjecture first suggested by Kontsevich on the purity of mixed Hodge structures arising in the theory of cluster mutation of spherical collections in 3--Calabi--Yau categories. The result implies positivity, as well as the stronger Lefschetz property conjectured by Efimov, and also the classical positivity conjecture of Fomin and Zelevinsky, recently proved by Lee and Schiffler.  Closely related to these results is a categorified ``no exotics'' type theorem for cohomological Donaldson--Thomas invariants, which we discuss and prove in the appendix.
\end{abstract}
\tableofcontents

\section{Introduction}
\subsection{Background}
This paper concerns the positivity conjecture for quantum cluster algebras, which were introduced in \cite{BZ05}.  These algebras are certain combinatorially defined noncommutative algebras over $\mathbb{Z}[q^{\pm 1/2}]$, generated by a distinguished set of n-element subsets, for some fixed $n$, called the \textit{clusters}.  These clusters are related to each other by a recursive procedure, called cluster mutation.  We refer the reader to the excellent references \cite{Ke375,Keller10} for a comprehensive guide and introduction to the background on cluster algebras, and we recall the necessary definitions in Section \ref{genCluSec}.  While quantum cluster mutations have a straightforward definition, the behaviour of clusters after iterated cluster mutation is rather complicated.  

The quantum cluster positivity conjecture states that every element in every cluster is a Laurent polynomial in the monomials of any other cluster, where the coefficients of these Laurent polynomials are themselves Laurent polynomials in $q^{1/2}$ with positive integer coefficients.  This statement, without the positivity, was proved at the outset of the subject by Berenstein and Zelevinsky in \cite{BZ05}.  Specialising the quantum cluster positivity conjecture at $q^{1/2}=1$ we obtain the \textit{classical} positivity conjecture of Fomin and Zelevinsky \cite{FZ02}, recently proved by Lee and Schiffler \cite{LS15} in the skew-symmetric case, and then in the case of geometric type by Gross, Hacking, Keel and Kontsevich \cite{GHKK14}, via quite different methods, which are not so distant from the mathematics of the current paper.

Some cases of the quantum version of the positivity conjecture have already been proved.  In \cite[Cor.3.3.10]{KQ14} the conjecture is proved by Kimura and Qin in the case in which the cluster algebra has a seed with an acyclic quiver, building on ideas of Nakajima \cite{Na11} in the classical case.  By different methods, which are much closer to the ones employed in this paper, Efimov \cite{Efi11} recovers positivity in the case that either the cluster $S'$ containing the monomial that we wish to express in terms of some other cluster $S$ corresponds to an acyclic quiver, or in the case that $S$ itself corresponds to an acyclic quiver.  In this instance Efimov proves the stronger \textit{Lefschetz property} (see Definition \ref{LefDef} below), which we prove in the general case.  Finally in \cite{DMSS13}, following \cite{Efi11}, along with Maulik, Sch\"urmann and Szendr\H{o}i  we were able to prove the quantum cluster positivity conjecture, along with the Lefschetz property, for all quantum cluster algebras admitting a seed, the quiver of which admits a quasihomogeneous nondegenerate potential.  This result is sufficient to prove the conjecture in many, but not all, cases arising `in nature'.

The papers \cite{Efi11} and \cite{DMSS13} closely follow the outline in the paper \cite{Na13}.  The discussion \cite[Sec.7]{Na13} is recommended for the reader wishing to gain some heuristic insight into the functioning of the proof below (although effort has been made to make the present paper reasonably self-contained).  There are some differences between our approach and that of \cite{Na13}, which can be explained with reference to the following principle, quoted from \cite{Na13}:
\begin{center}
``Starting from a simple categorical statement, provide an identity in the motivic Hall algebra. Pushing it out by the integration map, we get a power series identity for the generating functions of Donaldson--Thomas invariants.''
\end{center}
Via elementary recursive arguments, the resulting power series expansions recover mutated quantum cluster monomials.  

We start with the same categorical statement  --- the existence and uniqueness of Harder--Narasimhan filtrations --- but we use it instead to provide an isomorphism in the category of mixed Hodge modules over the space of dimension vectors, as opposed to an equality in a Grothendieck ring.  Our integration map (and the resulting identities) are essentially the same as the map used by Nagao, which was introduced by Kontsevich and Soibelman in their work on motivic Donaldson--Thomas invariants and Hall algebras \cite{KS}.  The benefit of working with these mixed Hodge modules, which provide a categorification of the motivic Hall algebra, is that we can make use of powerful results from Saito's theory of mixed Hodge modules \cite{Sai89,Sai90,Saito89}.  In particular, we make essential use of the concept of purity of mixed Hodge modules, and Saito's version of the decomposition theorem of Beilinson, Bernstein, Deligne and Gabber \cite{BBD}.

The relation of purity to the quantum cluster positivity conjecture was first made explicit in a conjecture suggested by Kontsevich, and explained by Efimov in \cite{Efi11}.  Before this, the deep connection between cluster mutation and the motivic Donaldson--Thomas theory of 3--Calabi--Yau categories was established by Kontsevich and Soibelman in \cite{KS}.  It is in this framework that the above quote of Nagao makes sense.  Via further work of Kontsevich and Soibelman \cite{COHA} the vanishing cycle cohomology of the particular moduli space underlying the element in the motivic Hall algebra that produces quantum cluster coefficients under the integration map carries a monodromic mixed Hodge structure, as defined in \cite[Sec 7]{COHA}.  Purity of this monodromic mixed Hodge structure implies the quantum cluster positivity conjecture; this implication was proved by Efimov in \cite{Efi11}.  Kontsevich and Efimov conjectured, and we prove below as Theorem \ref{KConj}, that this purity statement holds generically\footnote{In fact we are able to do without the genericity assumption.}.  We conclude the paper with the resulting derivation of quantum cluster positivity, which we state as Theorem \ref{mainThm} below.

\subsection{Standing conventions}
Throughout the paper we set $\mathbb{N}=\{r\in\mathbb{Z}|\hbox{ }r\geq 0\}$.

All varieties, schemes and stacks are assumed to be complex.  All quotients of schemes by algebraic groups are taken in the stack theoretic sense, unless we specify otherwise.  All functors are derived, unless we specify otherwise.

Given an object $M$ in an Abelian or triangulated category $\mathcal{D}$, we denote by $[M]$ the corresponding element in the Grothendieck group, which we denote by $\KK(\mathcal{D})$.

Given an algebra $A$, we denote by $\rmod{A}$ the category of finite-dimensional right $A$-modules, and $\Rmod{A}$ denotes the category of all right $A$-modules.  Where we consider the derived category of $A$-modules, a superscript $\fd$ will indicate that we restrict to the full subcategory containing those complexes of modules with finite dimensional total cohomology.  

All quivers in the paper are finite.  If $\Ss=(s_1,\ldots,s_t)$ is a sequence of vertices of a quiver, we denote by $\Ss'$ the truncated sequence $(s_1,\ldots,s_{t-1})$, and for $t'\leq t$ we denote by $\Ss_{\leq t'}$ the truncated sequence $(s_1,\ldots,s_{t'})$.

For $1\leq s\leq n$, we will use the symbol $1_s$ to denote the $s$th element in the natural set of generators for $\mathbb{Z}^n$.  We consider vectors as column vectors, so that the bilinear form associated to a square matrix $C$ is defined by $(\dd',\dd''):=\dd'^TC\dd''$.

If $\Mst$ is a moduli space of modules for a finitely generated algebra $A$, and $\Mst^{\nilp}\subset \Mst$ is the reduced subspace whose geometric points correspond to nilpotent modules, and $\mathcal{F}$ is a monodromic mixed Hodge module on $\Mst$, we write $\mathcal{F}_{\nilp}$ to denote $(\Mst^{\nilp}\rightarrow \Mst)_*(\Mst^{\nilp}\rightarrow \Mst)^*\mathcal{F}$. 

If $\mathcal{F}\in\mathcal{D}$ is an element of a triangulated category with a given t structure, we denote the total cohomology
\[
\Ho(\mathcal{F}):=\bigoplus_{i}\Ho^i(\mathcal{F})[-i].
\]

At numerous points we take ordered tensor products $\bigotimes_{\gamma\in S}\mathcal{F}_{\gamma}$, where $S$ is an infinite ordered set.  These are to be understood in the following sense.  Firstly, it will always be the case that each $\mathcal{F}_{\gamma}$ is canonically written as a direct sum $\mathbf{1}\oplus\mathcal{F}'_{\gamma}$, where $\mathbf{1}$ is a monoidal unit.  Then we define
\[
\bigotimes_{\gamma\in S}\mathcal{F}_{\gamma}=\bigoplus_{\textrm{finite}\;T\subset S}\left(\bigotimes_{\gamma\in T}\mathcal{F}'_{\gamma}\right).
\]
The direct sum includes the term $\mathbf{1}$ corresponding to the empty subset $\emptyset\subset S$.
\begin{remark}
\label{differenceRem}
For the reader that is familiar with preceding work on the link between Donaldson--Thomas theory and quantum cluster mutation in  \cite{Efi11} and \cite{DMSS13}, we flag and explain a technical difference in the approach of the present paper.  Firstly, recall (or see Section \ref{mutPot}) that the mutation of an algebraic quiver with potential need not be algebraic, in the sense of Definition \ref{algDef}.  In \cite{Efi11} and \cite{DMSS13} this issue was handled as follows.  We picked a formal potential $W$ for $Q$ and considered the category $\mathcal{A}=\Rmod{\HJac(Q,W)}$ of right modules over the associated completed Jacobi algebra, defined in Section \ref{mutPot}.  We compared this category with the tilted heart $\mathcal{A}'=(\Rmod{\HJac(Q,W)})^{(\mathcal{F}'_{\Ss},\mathcal{T}'_{\Ss}[-1])}$, where $(\mathcal{T}'_{\Ss},\mathcal{F}'_{\Ss})$ is a torsion structure on $\mathcal{A}$, built recursively from the data of a sequence of vertices $\Ss$, using the version of Nagao's procedure in \cite[Sec.3]{Na13} that starts with a \textbf{left} tilt.  I.e. in the case that $\Ss$ is empty, $\mathcal{F}'_{\Ss}=\Rmod{\HJac(Q,W)}$.

For certain choices of $W$, there is an equivalence of categories $\mathcal{A}'\cong \Rmod{\HJac(Q',W')}$, where now $W'$ is algebraic, and so in the target category $\mathcal{A}'$, Donaldson--Thomas theory is more straightforward to set up. In particular, the moduli stack in the motivic Hall algebra that we apply the integration map to and then conjugate by in order to reproduce the operation of quantum cluster mutation (the stack of finite-dimensional objects in $\mathcal{T}'_{\Ss}$) carries a monodromic mixed Hodge structure on its vanishing cycle cohomology, via the constructions of \cite{COHA}, since it may be considered as a substack of the stack of objects in $\mathcal{A}'$.  

Note, however, that the intermediate torsion categories $\mathcal{T}'_{\Ss_{\leq t'}}$ occurring in the recursive definition of $\mathcal{T}'_{\Ss}$ are all subcategories of $\mathcal{A}$, and not $\mathcal{A}'$.  Since we have no guarantee that $\mathcal{A}=\Rmod{\HJac(Q,W)}$ is the category of right modules of a (completed) Jacobi algebra arising from an algebraic quiver with potential, it is thus not clear how to use the Hodge-theoretic constructions of \cite{COHA} in inductive arguments.  

We remedy this by \textit{picking} $W$ to be an algebraic potential for $Q$, and tilting \textbf{right}, i.e. the category we tilt towards is $\mathcal{A}''=(\Rmod{\HJac(Q,W)})^{(\mathcal{F}''_{\Ss}[1],\mathcal{T}''_{\Ss})}$, where $(\mathcal{T}''_{\Ss},\mathcal{F}''_{\Ss})$ is a torsion structure built using the version of Nagao's recipe that starts with a \textit{right} tilt.  I.e. if $\Ss$ is empty, we have $\mathcal{T}''_{\Ss}=\Rmod{\HJac(Q,W)}$.  With this approach, quantum cluster mutaton is given by conjugating by the integral of the stack of finite-dimensional objects in $\mathcal{F}''_{\Ss}$.  Furthermore, the categories $\mathcal{F}''_{\Ss_{\leq t'}}$ that we encounter in inductive arguments are subcategories of $\mathcal{A}$, which now has been chosen to be the category of representations of a Jacobi algebra with an algebraic potential, allowing us to use Hodge theory inductively.  

As a result of this subtle change in setup, some of the statements below are technically different from their counterparts in the literature, and for this reason, as well as the hope that the paper can be relatively self-contained, some proofs are repeated.  Finally, due to this change, strictly speaking, the version of Kontsevich's conjecture that we prove as Theorem \ref{KConj} is different to the version stated in \cite{Efi11}.
\end{remark}
\subsection{Acknowledgements}
I would like to thank Bernhard Keller for patiently explaining cluster mutation to me, and Bal\'azs Szendr\H{o}i for introducing me to the subject in the first place.  During the writing of this paper, I was a postdoctoral researcher at EPFL, supported by the Advanced Grant ``Arithmetic and Physics of Higgs moduli spaces'' No. 320593 of the European Research Council.  During the revision of this paper I was supported by the University of Glasgow.

\section{Quivers and cluster algebras}\label{genCluSec}
\subsection{Quantum cluster algebras}
\label{cluSec}
Let $Q$ be a finite quiver, i.e. a pair of finite sets $Q_1$ (the arrows) and $Q_0$ (the vertices), along with two maps $s,t\colon Q_1\rightarrow Q_0$, taking an arrow to its source and target, respectively.  We assume that $Q$ has no loops or 2-cycles.  We define two bilinear forms on $\mathbb{Z}^{Q_0}$:
\begin{align}
(\dd',\dd'')_Q=&\sum_{i\in Q_0}\dd'_i\dd''_i-\sum_{a\in Q_1} \dd'_{t(a)}\dd''_{s(a)},\\
\langle \dd',\dd''\rangle_Q=&(\dd',\dd'')_Q-(\dd'',\dd')_Q.
\end{align}
We will omit the subscript $Q$ where there is no chance of confusion.  We fix a labelling of the vertices $Q_0$ by numbers $\{1,\ldots,n\}$, and fix a number $m\leq n$.  This defines for us an \textit{ice quiver}, i.e. a quiver with the extra data of a subset of vertices $S\subset Q_0$, the so-called \textit{frozen vertices}.  In our case we set $S=\{m+1,\ldots,n\}$.  The full subquiver $Q_{\princ}$ containing the vertices $\{1,\ldots,m\}$ is called the \textit{principal part} of $Q$ --- these are the vertices that we are allowed to perform mutations at.  

The \textit{mutation} $\mu_i(Q)$ of $Q$ at a vertex $i\leq m$ is performed in 3 steps.
\begin{enumerate}
\item
For all paths of length 2 passing through $i$, i.e. pairs of arrows $b,c\in Q_0$ with $t(b)=i=s(c)$, we introduce an arrow $[cb]$ with $s([cb])=s(b)$ and $t([cb])=t(c)$.
\item
For all arrows $b\in Q_1$ incident to $i$, we replace $b$ with an arrow $\overline{b}$ with the opposite orientation.
\item
If for two vertices $j,j'$ there are $r$ arrows going from $j$ to $j'$, and $r'$ arrows going from $j'$ to $j$, with $r\geq r'$, we delete all of the arrows going from $j'$ to $j$, and $r'$ of the arrows going from $j$ to $j'$.
\end{enumerate}
The set of frozen vertices is unchanged by mutation.  This defines an automorphism of the set of isomorphism classes of ice quivers without loops and 2-cycles.  The operation is well-defined because of the deletion step, and is an automorphism because mutation at $i$ is an involution on the set of such isomorphism classes.

Let $L$ be a rank $n$ free $\mathbb{Z}$-module, and let $\Lambda\colon L\times L\rightarrow L$ be a skew-symmetric bilinear form.  The \textit{quantum torus} $\Tor_{\Lambda}$ is freely generated, as a $\mathbb{Z}[q^{\pm 1/2}]$-module, by elements $X^\ee$, for $\ee\in L$, with multiplication on these elements defined by
\begin{equation}
\label{LambdaMult}
X^\ee\cdot X^\ff=q^{\Lambda(\ee,\ff)/2}X^{\ee+\ff}
\end{equation}
and extended $\mathbb{Z}[q^{\pm 1/2}]$-linearly to the whole of $\Tor_{\Lambda}$.  We denote by $\For_{\Lambda}$ the skew field of fractions of $\Tor_{\Lambda}$.  A \textit{toric frame} for $\For_{\Lambda}$ is a map
\[
M\colon \mathbb{Z}^n\rightarrow \For_{\Lambda}
\]
defined by $M(c)=\tau(X^{\nu(c)})$ with $\tau\in\Aut_{\mathbb{Q}(q^{\pm 1/2})}(\For_{\Lambda})$ and $\nu\colon \mathbb{Z}^n\rightarrow L$ an isomorphism of lattices.  We fix an identification $L=\mathbb{Z}^n$, and we fix an initial toric frame $M$ by setting $\tau=\id$.  The pair $(Q,M)$ is called the \textit{initial seed}.

Since $Q$ contains no 2-cycles, the isomorphism class of the ice quiver $Q$ is encoded in the $n\times m$ matrix $\tilde{B}$, defined by setting $\tilde{B}_{ij}=a_{ji}-a_{ij}$, where $a_{ij}$ is the number of arrows $a\in Q_1$ with $s(a)=i$ and $t(a)=j$.  We identify $\Lambda$ with the $n\times n$ matrix associated with $\Lambda$ via the identification $L=\mathbb{Z}^n$.  The matrix $\tilde{B}$ may alternatively be described as the matrix given by expressing the bilinear form ${}-\langle\bullet,\bullet\rangle_Q$ in the standard basis, and then deleting the last $(n-m)$ columns.  We say that $\tilde{B}$ is \textit{compatible} with $\Lambda$ if 
\begin{equation}
\label{compatibility}
\tilde{B}^T\Lambda=\tilde{I},
\end{equation}
where the first $m$ columns of $\tilde{I}$ are the identity matrix, and the remaining entries are zeroes.  We say the ice quiver $Q$ is compatible with $\Lambda$ if $\tilde{B}$ is.

The elements $M(1_1),\ldots,M(1_m)\in \For_{\Lambda}$ are called the \textit{cluster variables}, while the elements $M(1_{m+1}),\ldots,M(1_n)$ are called the \textit{coefficients}.  So the cluster variables of the initial seed are $X^{1_1},\ldots,X^{1_m}$, while the coefficients are $X^{1_{m+1}},\ldots,X^{1_n}$.  Note that the coefficients are unchanged by mutation.

We define the ring $\Ror_Q$ to be the free $\mathbb{Z}[q^{\pm 1/2}]$-module generated by elements $Y^\ee$, for $\ee\in\mathbb{N}^m$, with multiplication defined by
\begin{equation}
\label{Ymult}
Y^\ee\cdot Y^\ff=q^{\langle \ff,\ee\rangle_Q/2}Y^{\ee+\ff}.
\end{equation}
If $\tilde{B}$ and $\Lambda$ are compatible, then the map
\begin{align}
\iota\colon&\Ror_Q\rightarrow\Tor_{\Lambda}\label{iotadef}\\
&Y^\ee\rightarrow X^{{}-\tilde{B}\cdot \ee}\nonumber
\end{align}
is a homomorphism of algebras.  Let $\PpP=\mathbb{Z}((q^{1/2}))$.  We define $\hat{\Ror}_Q$ to be the completion of $\Ror_Q\otimes_{\mathbb{Z}[q^{\pm 1/2}]} P$ with respect to the two-sided ideal generated by $\{Y^{1_s}|s\leq m\}$.  Consider the $\Tor_{\Lambda}$-module $K=\prod_{\ee\in\mathbb{Z}^n}X^{\ee}P$, with $\Tor_{\Lambda}$-action defined via (\ref{LambdaMult}).  The map $\iota$ extends naturally to a map $\hat{\iota}\colon\hat{\Ror}_Q\rightarrow K$, which is injective by (\ref{compatibility}), and we define
\[
\hat{\Tor}_{\Lambda}:=\bigcup_{\ee\in \mathbb{Z}^n}X^{\ee}\cdot \textrm{Image}(\hat{\iota})
\]
with multiplication as in (\ref{LambdaMult}).  There is a natural inclusion of $\mathbb{Z}[q^{\pm 1/2}]$-algebras $\Tor_{\Lambda}\subset\hat{\Tor}_{\Lambda}$.

If $\tilde{B}$ is the matrix associated to the ice quiver $Q$, and $s$ is a vertex in the principal part of $Q$, we define $\mu_s(\tilde{B})$ to be the matrix associated to the mutated ice quiver $\mu_s(Q)$.  We define the mutation of toric frames via the rule
\begin{align}\label{mutFram}
\mu_s(M)(1_i)=\begin{cases} M(1_i)& \textrm{for }i\neq s,\\ M\left(\sum_{b_{rs}>0}b_{rs}1_r-1_s\right)+ M\left({}-\sum_{b_{rs}<0}b_{rs}1_r-1_s\right) &\textrm{for }i=s.\end{cases}
\end{align}
Mutation of seeds is defined by $\mu_s((Q,M))=(\mu_s(Q),\mu_s(M))$.  The classical notion of cluster mutation is recovered by specialising at $q^{1/2}=1$.  

We consider the initial seed $(Q,M)$ defined above.  If $\Ss=(s_1,\ldots,s_t)$ is a sequence of mutations, we define $\mu_{\Ss}(Q):=\mu_{s_t}(\ldots\mu_{s_1}(Q)\ldots )$ and define $\mu_{\Ss}((Q,M))$ similarly.  If a quiver $Q$ is understood, we write $\mu_{\Ss}(M)$ to denote the toric frame of $\mu_{\Ss}((Q,M))$, i.e. the toric frame defined recursively from the initial seed $(Q,M)$, where at each stage in the recursive procedure we use the mutated quiver to define the sum in (\ref{mutFram}).  The set 
\[
\big\{\{\mu_{\Ss}(M)(1_i)|i\leq m\}\;|\;\Ss\textrm{ a sequence of vertices of }Q_{\princ}\big\}
\]
is called the set of clusters, while the set 
\[
\{\mu_{\Ss}(M)(\dd)\;|\;\Ss\textrm{ a sequence of vertices of }Q_{\princ},\,\,\dd\in \mathbb{N}^{n}\}
\]
is called the set of cluster monomials.  The \textit{quantum cluster algebra} $\mathcal{A}_{\Lambda,Q}$ is the sub $\mathbb{Z}[q^{\pm 1/2}]$-algebra of $\For_{\Lambda}$ generated by the set
\[
\{\mu_{\Ss}(M)(\dd)\;|\;\Ss\textrm{ a sequence of vertices of }Q_{\princ},\,\, \dd\in \mathbb{N}^{m}\times\mathbb{Z}^{n-m}\}.
\]
Setting $q^{1/2}=1$, we recover the ordinary \textit{commutative} cluster algebra $\mathcal{A}_{Q}$ of \cite{FZ02}.  

\begin{remark}
\label{quantization}
We say that the cluster algebra $\mathcal{A}_Q$ associated to an ice quiver $Q$ can be \textit{quantized} if we can find a quiver $Q'$, containing $Q$ as a full subquiver, such that if we set the principal part of $Q'$ to be the same as the principal part of $Q$ (i.e. we only add frozen vertices), we can find a skew-symmetric $n'\times n'$ matrix $\Lambda$ compatible with $Q'$, where $|Q'_0|=n'$.  By \cite[Lem.4.4]{DMSS13}, this can always be done.
\end{remark}
By a result of Berenstein and Zelevinsky \cite[Cor.5.2]{BZ05}, the inclusion $\mathcal{A}_{\Lambda,Q}\subset\For_{\Lambda}$ factors through the inclusion $\Tor_{\Lambda}\subset \For_{\Lambda}$.  Equivalently, for a given mutated toric frame $M'=\mu_{\Ss}(M)$, and an arbitrary cluster monomial $Y$, we can write
\begin{equation}
\label{LaurentExp}
Y=\sum_{\dd\in \mathbb{Z}^n}a_\dd(q^{1/2})M'(\dd)
\end{equation}
where the $a_\dd(q^{1/2})\in\mathbb{Z}[q^{\pm 1/2}]$.
\begin{definition}
\label{LefDef}
We say that a Laurent polynomial $a(q^{1/2})$ is of \textit{Lefschetz type} if it can be written as a sum of polynomials of the form $(q^{d/2}-q^{-d/2})/(q^{1/2}-q^{-1/2})$ for positive integers $d$.
\end{definition}
In particular, a polynomial of Lefschetz type has positive integral coefficients.
\begin{remark}
Say $b(q^{1/2})=\sum_{i\in \mathbb{Z}}\dim(V^i)q^{i/2}$ is the characteristic polynomial of a $\mathbb{Z}$-graded finite-dimensional vector space.  Then $b(q^{1/2})$ is of Lefschetz type if and only if there is a degree two operator $l\colon V^{\bullet}\rightarrow V^{\bullet+2}$ such that $l^k\colon V^{-k}\rightarrow V^k$ is an isomorphism for all $k$.  For example, by the hard Lefschetz theorem, this occurs if $V=\HO(X,\mathbb{Q})[\dim(X)]$ is the (shifted) cohomology of a smooth projective variety.
\end{remark}
The purpose of this paper is to prove the following theorem.  The proof will be completed in Section \ref{mainProof}.
\begin{theorem}[Quantum cluster positivity]
\label{mainThm}
Let $\mathcal{A}_{\Lambda,Q}$ be a quantum cluster algebra defined by a compatible pair $(Q,\Lambda)$.  For a mutated toric frame $M'$, and a cluster monomial $Y$, the Laurent polynomials $a_\dd(q^{1/2})$ appearing in the expression (\ref{LaurentExp}) are of Lefschetz type, and in particular, they have positive coefficients.  Furthermore, they can be written in the form $a_{\dd}(q^{1/2})=b_{\dd}(q)q^{-\deg(b_{\dd}(q))/2}$ for $b_{\dd}(q)\in\mathbb{N}[q]$, i.e. each polynomial $a_{\dd}(q^{1/2})$ contains only even or odd powers of $q^{1/2}$.
\end{theorem}
\subsection{Mutation of quivers with potential}
\label{mutPot}
Given a quiver as in Section \ref{cluSec}, we define $\mathbb{C}Q$ to be the free path algebra of $Q$ over $\mathbb{C}$.  This algebra contains a two-sided ideal $\mathbb{C}Q_{\geq 1}$ generated by the elements $a\in Q_1$, and we define $\widehat{\mathbb{C}Q}$ to be the topological algebra obtained by completing $\mathbb{C}Q$ with respect to this ideal.

Given a topological algebra $A$, we define $A_{\cyc}:=A/\overline{[A,A]}$.  Let $W\in \widehat{\mathbb{C}Q}_{\cyc}$ be a \textit{potential}, i.e. a formal linear combination of cyclic words in $Q$, where we consider two cyclic words to be the same if we can cyclically permute one to the other.  
\begin{definition}
\label{algDef}
We say that a potential is \textit{algebraic} if it is in the image of the injection
\[
\mathbb{C}Q/[\mathbb{C}Q,\mathbb{C}Q]\hookrightarrow \widehat{\mathbb{C}Q}_{\cyc},
\]
i.e. it is a finite linear combination of cyclic words in $Q$.  We say that the quiver with potential (QP for short) $(Q,W)$ is algebraic if $W$ is.
\end{definition}
Sometimes we will refer to a potential as a \textit{formal potential} if we want to make it clear that it is not necessarily algebraic.

Given a cyclic word $c\in \mathbb{C}Q/[\mathbb{C}Q,\mathbb{C}Q]$, and an arrow $a\in Q_1$, we define
\[
\partial c/\partial a=\sum_{\tilde{c}=bag}gb,
\]
where $\tilde{c}$ is a fixed lift of $c$ to $\mathbb{C}Q$.  Extending by linearity and then continuity, we obtain an operation $\partial/\partial a \colon \widehat{\mathbb{C}Q}_{\cyc}\rightarrow \widehat{\mathbb{C}Q}$, restricting to an operation $\partial/\partial a\colon \mathbb{C}Q/[\mathbb{C}Q,\mathbb{C}Q]\rightarrow \mathbb{C}Q$.  Given a QP $(Q,W)$, we define the \textit{Jacobi algebra}
\begin{align*}
&\HJac(Q,W):=\widehat{\mathbb{C}Q}/\overline{\langle \partial W/\partial a|a\in Q_1\rangle},
\end{align*}
and if $W$ is algebraic, we define the \textit{algebraic} Jacobi algebra
\begin{align*}
&\Jac(Q,W):=\mathbb{C}Q/\langle \partial W/\partial a|a\in Q_1\rangle.
\end{align*}
If $W$ is an algebraic potential, by pulling back along the map $\Jac(Q,W)\rightarrow \HJac(Q,W)$ we obtain a functor $\rmod{\HJac(Q,W)}\rightarrow \rmod{\Jac(Q,W)}$.  

\begin{proposition}
\label{EngelProp}
Let $W$ be an algebraic potential for a quiver $Q$.  The functor $\rmod{\HJac(Q,W)}\rightarrow \rmod{\Jac(Q,W)}$ is an equivalence after restricting the target to $(\rmod{\Jac(Q,W)})_{\nilp}$, the full subcategory of finite dimensional $\Jac(Q,W)$-modules for which all sufficiently long paths act by the zero map.  
\end{proposition}
\begin{proof}
Let $M$ be a finite-dimensional $\Jac(Q,W)$-module for which every element $z\in \mathbb{C}Q_{\geq 1}$ acts nilpotently.  Then by Engel's theorem, for the Lie algebra \[
\textrm{Image}\left(\mathbb{C}Q_{\geq 1}\rightarrow \End_{\mathbb{C}}(M)\right), 
\]
there is a basis of $M$ on which every element of $\Jac(Q,W)$ acts via strictly upper triangular matrices, and so the action of $\mathbb{C}Q$ factors through the map $\mathbb{C}Q\rightarrow \mathbb{C}Q/\mathbb{C}Q_{\geq \dim(M)}$, and in particular $M$ carries a continuous $\widehat{\mathbb{C}Q}$ action inducing the given $\Jac(Q,W)$-action.  Conversely, assume that a $\Jac(Q,W)$-action on a vector space extends to a continuous $\widehat{\mathbb{C}Q}$-action.  Then write the action of $z\in \Jac(Q,W)_{\geq 1}$ as an upper triangular matrix, with $D$ denoting the maximum modulus of the entries on the diagonal.  If $D\neq 0$, we may consider the action of $z+D^{-1}z^2+D^{-2}z^3+\ldots$ to arrive at a contradiction, and so we deduce that the action of $z$ is via a strictly upper triangular matrix.  By Engel's theorem again, all paths of length greater than $\dim(M)$ act on $M$ via the zero map.
\end{proof}

If $(Q,W)$ is a QP, and $s\in Q_0$, we denote by $\Simp(Q)_s$ the nilpotent $\HJac(Q,W)$-module with dimension vector $1_s$, which by Proposition \ref{EngelProp} we may also consider as a nilpotent $\Jac(Q,W)$-module if $(Q,W)$ is an algebraic QP.

Given a QP $(Q,W)$, and a principal vertex $s\leq m$, we recall the definition of the mutated QP $\mu_s((Q,W))$ from \cite{DWZ08,DWZ10}, which are comprehensive references for the material in this subsection.  We assume that $Q$ has no loops or 2-cycles.  The \textit{premutation} $\mu'_s((Q,W))=(\mu'_s(Q),\mu'_s(W))$ is defined on the $Q$ component in the same way as mutation, except we leave out the deletion step (so $\mu'_s(Q)$ may contain 2-cycles).  We obtain $W_s$ from $W$ by replacing every instance of $cb$ in $W$, where $cb$ is a path of length two passing through $s$ (as in step one of the definition of mutation for $Q$), with $[cb]$.  We then define
\[
\mu'_s(W)=W_s+\sum_{\substack{c,b\in Q_1\\s(c)=t(b)=s}}[cb]\overline{b}\overline{c}.
\]
Given a quiver $Q'$ with vertex set equal to our fixed set $\{1,\ldots,n\}$, we define
\[
R:=\bigoplus_{s\in Q'_0} \mathbb{C}\cong \mathbb{C}^{\oplus n},
\]
and for $s\in Q'_0$ we define $e_s\in R$ to be the idempotent corresponding to the vertex $s$.  We fix an $R$-bimodule $E_{Q'}$ with $\dim(e_j E_{Q'} e_i)$ equal to the number of arrows from $i$ to $j$ in $Q'$.  Fixing an identification between the arrows of $Q'$ from $i$ to $j$, and a basis for the vector space $e_j E_{Q'} e_i$, defines an isomorphism $\widehat{\mathbb{C}Q'} \cong \CTens_R(E_{Q'})$, where $\CTens_R(E_{Q'})$ is the completed free unital tensor algebra generated by the $R$-bimodule $E_{Q'}$.  Let $W$ be a formal potential for $Q'$.  By the splitting theorem \cite[Thm.4.6]{DWZ08} there is an isomorphism of completed unital $R$-algebras
\[
\psi\colon\CTens(E_{Q'})\cong \CTens(E_{Q'_{\textrm{triv}}}\oplus_{R\textrm{-bimod}}E_{Q'_{\textrm{red}}})
\]
such that $\psi(W)=W_{\textrm{triv}}+W_{\textrm{red}}$, with $W_{\textrm{triv}}\in \CTens(E_{Q'_{\textrm{triv}}})_{\cyc}$ and $W_{\textrm{red}}\in \CTens(E_{Q'_{\textrm{red}}})_{\cyc}$, such that $\HJac(Q'_{\textrm{triv}},W_{\textrm{triv}})\cong R$ in the category of completed $R$-algebras, and such that $W_{\textrm{red}}$ can be expressed as a formal linear combination of cyclic words of length at least three.  Here, $Q'_{\mathrm{triv}}$ is the quiver with $\dim(e_jE_{Q'_{\textrm{triv}}}e_i)$ arrows from $i$ to $j$, and $Q'_{\mathrm{red}}$ is defined similarly.  We define $(Q',W)_{\textrm{red}}:=(Q'_{\textrm{red}},W_{\textrm{red}})$.  Finally, we define 
\[
\mu_s((Q,W))=(\mu'_s(Q),\mu'_s(W))_{\textrm{red}}.
\]
The mutated QP is well-defined up to isomorphisms induced by isomorphisms of completed $R$-algebras.  

We say that $W$ is nondegenerate with respect to mutation at $s$ if the underlying quiver of $\mu_s((Q,W))$ is equal to $\mu_s(Q)$, which occurs if and only if the underlying quiver contains no 2-cycles.  Given a sequence $\Ss=(s_1,\ldots,s_t)$ of vertices of $Q_0$, we say that $W$ is nondegenerate with respect to $\Ss$ if for all $t'\leq t$, the underlying quiver of $\mu_{\:\Ss_{\leq t'}}(Q,W)$ contains no 2-cycles so that, in particular, each $\mu_{\:\Ss_{\leq t'}}(Q,W)$ is well-defined, recursively.  We say that $W$ is nondegenerate if it is nondegenerate with respect to all sequences of principal vertices.  Since we work over $\mathbb{C}$, by \cite[Cor.7.4]{DWZ08}, there always exists an algebraic nondegenerate potential for $Q$.
\section{Some Donaldson--Thomas theory}
\subsection{Monodromic mixed Hodge modules}\label{MMHMSec}
Let $X$ be a complex variety.   We define as in \cite{Saito89,Sai90} the derived category $\Dub(\MHM(X))$ of mixed Hodge modules on $X$, or we refer the reader to \cite{Sai89} for an overview.  We refer the reader to \cite[Sec.4, Sec.7]{COHA} for a discussion of the related concept of monodromic mixed Hodge structures, which we expand upon to suit our purposes here.

The category of \textit{monodromic mixed Hodge modules} on $X$, denoted $\MMHM(X)$, is defined as the Serre quotient of two subcategories of $\MHM(X\times\mathbb{A}^1)$.  Firstly we define $\mathcal{B}_X$ to be the full subcategory of $\MHM(X\times\mathbb{A}^1)$ containing those $\mathcal{F}$ such that for every $x\in X$, the total cohomology of the pullback $(\{x\}\times\mathbb{G}_m\hookrightarrow X\times\mathbb{A}^1)^*\mathcal{F}$ is an admissible variation of mixed Hodge structure on $\mathbb{G}_m$.  Via Saito's description of $\MHM(X\times\mathbb{A}^1)$, we may alternatively describe $\mathcal{B}_X$ as the subcategory of $\MHM(X\times\mathbb{A}^1)$ obtained by iterated extension of mixed Hodge modules $\IC_Y(\mathcal{L})[\dim(Y)]$, where $\mathcal{L}$ is a pure variation of Hodge structure on a dense open subvariety $Y'$ of the regular locus $Y_{\reg}$ of a closed irreducible subvariety $Y\subset X\times\mathbb{A}^1$, where $\mathbb{G}_m$ acts by scaling $\mathbb{A}^1$ and this action restricts to an action on $Y'$.  Secondly, we define $\mathcal{C}_X$ to be the full subcategory of $\MHM(X\times\mathbb{A}^1)$ containing those $\mathcal{F}$ obtained as $\pi^*\mathcal{G}[1]$, for $\mathcal{G}\in\MHM(X)$, and $\pi\colon X\times\mathbb{A}^1\rightarrow X$ the natural projection.  A mixed Hodge module $\mathcal{F}$ is in $\mathcal{C}_X$ if and only if the total cohomology of $(\{x\}\times\mathbb{A}^1\rightarrow X\times\mathbb{A}^1)^*\mathcal{F}$ is an admissible variation of mixed Hodge structure for all $x\in X$.  Again, via Saito's results, we may alternatively describe $\mathcal{C}_X$ as the smallest full subcategory, closed under extensions, containing $\IC_{Y\times \mathbb{A}^1}(\mathcal{L})[\dim(Y)+1]$ for $Y\subset X$ an irreducible closed subvariety, and $\mathcal{L}$ a pure variation of Hodge structure on $Y'\times \mathbb{A}^1$, a dense open subvariety of $Y_{\reg}\times\mathbb{A}^1$, since by \cite{StZu85}, any such variation of Hodge structure is trivial along the fibres of the projection $X\times\mathbb{A}^1\rightarrow X$.  The category $\mathcal{C}_X$ is a Serre subcategory of $\mathcal{B}_X$.  

Following \cite[Sec.7]{COHA} we define 
\[
\MMHM(X):=\mathcal{B}_X/\mathcal{C}_X.
\]

The natural functor $\Dub(\mathcal{B}_X)/\Dub_{\mathcal{C}_X}(\mathcal{B}_X)\rightarrow \Dub(\mathcal{B}_X/\mathcal{C}_X)$ is an equivalence of triangulated categories, where $\Dub_{\mathcal{C}_X}(\mathcal{B}_X)\subset \Dub(\mathcal{B}_X)$ is the full subcategory containing those objects whose cohomology objects are in $\mathcal{C}_X$.  The subcategory $\Dub_{\mathcal{C}_X}(\mathcal{B}_X)$ is stable under the Verdier duality functor $\DD_{X\times \mathbb{A}^1}$ defined on $\Dub(\MHM(X\times\mathbb{A}^1))$, and so the category $\Dub(\MMHM(X))=\Dub(\mathcal{B}_X/\mathcal{C}_X)$ inherits a Verdier duality functor, which we denote $\DD_{X}^{\mon}$.  We define the four functors $f^*,f^!,f_*,f_!$ for categories of monodromic mixed Hodge modules via the same observation.  We embed $\MHM(X)$ inside $\MMHM(X)$ via direct image along the zero section $(X\times\{0\}\xrightarrow{z} X\times\mathbb{A}^1)$.

Since the associated graded object $\Gr_W(\mathcal{F})$ of an object in $\mathcal{C}_X$ is also in $\mathcal{C}_X$, an object in $\MMHM(X)$ has a well-defined weight filtration; if
\[
\xymatrix{
\mathcal{F}&\mathcal{G}\ar@{->>}[d]^{\varpi}\\
\mathcal{F}'\ar@{^(->}[u]\ar[r]^{\cong}&\mathcal{G}''
}
\]
represents an isomorphism $\mathcal{F}\rightarrow \mathcal{G}$ in $\mathcal{B}_X/\mathcal{C}_X$ (i.e. $\mathcal{F}/\mathcal{F}'$ and $\ker(\varpi)$ are elements of $\mathcal{C}_X$), then after applying the functor $W_n$ to the diagram it represents the isomorphism $W_n\mathcal{F}\cong W_n\mathcal{G}$ in the quotient category by exactness of the functor $W_n$ \cite[5.1.14]{Sa88}.

\begin{definition}
We say that an object of $\MMHM(X)$ is pure of weight $n$ if $W_{n-1}\mathcal{F}=0$ and $W_{n}\mathcal{F}=\mathcal{F}$.  We say that an object $\mathcal{F}\in\Dub(\MMHM(X))$ is pure of weight $n$ if $\Ho^l(\mathcal{F})$ is pure of weight $l+n$ for all $l$, or we will just say that $\mathcal{F}$ is \textit{pure} if it is pure of weight zero.
\end{definition}
We define $\Ho(\mathcal{F}):=\bigoplus_{i\in\mathbb{Z}}\Ho^i(\mathcal{F})[-i]$, and so the object $\mathcal{F}\in\Dub(\MMHM(X))$ is pure if and only if $\Ho(\mathcal{F})$ is.
\begin{definition}
We define $\Dulf(\MMHM(X))\subset \Dub(\MMHM(X))$ to be the full subcategory containing those objects $\mathcal{F}$ satisfying the following condition: for each connected component $Y\in\pi_0(X)$, there exists a $N_Y\in\mathbb{Z}$ such that $\Gr_W^g(\mathcal{H}(\mathcal{F})|_Y)=0$ for all $g\leq N_Y$.  Here $\mathcal{H}(\mathcal{F})$ is the total cohomology of $\mathcal{F}$, considered as an object of $\Dub(\MMHM(X))$ via the cohomological grading (i.e. as a complex with zero differential).  We require also that for all $g>N_Y$, $\Gr_W^{g}(\Ho(\mathcal{F})|_Y)\in\Db(\MMHM(Y))$.  We define $\Dllf(\MMHM(X))$ similarly.  
\end{definition}


The categories $\Dulf(\MMHM(X))$ and $\Dllf(\MMHM(X))^{\textrm{op}}$ are equivalent via Verdier duality.  For two varieties $X$ and $Y$ we define an external tensor product $\Dulf(\MMHM(X))\times\Dulf(\MMHM(Y))\rightarrow \Dulf(\MMHM(X\times Y))$ by setting
\begin{equation}
\label{TP}
\mathcal{F}\boxtimes\mathcal{G}:=(X\times Y\times \mathbb{A}^1\times\mathbb{A}^1\xrightarrow{\id_{X\times Y}\times +} X\times Y\times \mathbb{A}^1)_*\pi_{1,3}^*\mathcal{F}\otimes\pi_{2,4}^*\mathcal{G}
\end{equation}
where $\pi_{i,j}$ is the projection onto the $i$th and $j$th factors, and the tensor product on the right hand side of (\ref{TP}) is the usual tensor product of complexes of mixed Hodge modules.  If $Y$ is a point, and $\mathcal{F}\in\Dulf(\MMHM(X))$ and $\mathcal{G}\in\Dulf(\MMHM(Y))$, we will denote by $\mathcal{F}\otimes\mathcal{G}\in\Dulf(\MMHM(X))$ their external tensor product.
\begin{proposition}\label{tensPres}
The weight filtrations on $\MMHM(X)$ and $\MMHM(Y)$ are compatible with the external tensor product, which is biexact.
\end{proposition}
The part of the proposition regarding weight filtrations is an easy consequence of \cite[Prop.4]{COHA}.  We will mainly use two special cases: the external product of mixed Hodge modules with trivial monodromy in the sense of Definition \ref{trivMonDef}, where the proposition is \cite[(3.8.2)]{Sai90}, and the case where $X$ and $Y$ are a point, where the proposition is a special case of \cite[Prop.4]{COHA}.  
\begin{proof}
For the biexactness statement, first note that for $\mathcal{F}\in\mathcal{B}_X$ and $\mathcal{G}\in\mathcal{B}_Y$, as in \cite[Lem.1]{COHA}, there is an isomorphism
\[
(\id_{X\times Y}\times +)_!\left(\pi_{1,3}^*\mathcal{F}\otimes\pi_{2,4}^*\mathcal{G}\right)\rightarrow(\id_{X\times Y}\times +)_*\left(\pi_{1,3}^*\mathcal{F}\otimes\pi_{2,4}^*\mathcal{G}\right)
\]
considered as a morphism in $\Db(\MMHM(X\times Y))$.  The map $(\id_{X\times Y}\times +)$ is affine, and so $(\id_{X\times Y}\times +)_*$ is right exact, while $(\id_{X\times Y}\times +)_!$ is left exact.  In addition, $(\id_{X\times Y}\times +)_*$ increases weights, while $(\id_{X\times Y}\times +)_!$ decreases them, which gives the statement regarding the weight filtrations, as in \cite[Prop.4]{COHA}.
\end{proof}
If $(X,\,\tau\colon X\times X\rightarrow X,\,\eta\colon\Spec(\mathbb{C})\rightarrow X)$ is a monoid in the category of locally finite type schemes, with $\tau$ of finite type, and so in particular the induced map $\pi_0(X)\times\pi_0(X)\rightarrow \pi_0(X)$ has finite fibres, we define tensor products $\boxtimes_{\tau}$ for $\Dulf(\MMHM(X))$ and for $\Dllf(\MMHM(X))$ by setting 
\begin{equation}
\mathcal{F}\boxtimes_{\tau}\mathcal{G}:=\tau_*\left(\mathcal{F}\boxtimes\mathcal{G}\right).
\end{equation}
Saito proved \cite[(4.5.3),(4.5.4)]{Sai90} that for fixed $w\in\mathbb{Z}$, the category of pure weight $w$ mixed Hodge modules on a variety $X$ is semisimple, and if $\mathcal{F}\in\Dub(\MHM(X))$ is pure of weight $w$ for some $w$, and $f$ is a proper map of varieties, then there is a noncanonical isomorphism
\[
f_*\mathcal{F}\cong \Ho(f_*\mathcal{F})
\]
into pure weight $w$ summands.  This is the version, in the framework of Saito's theory, of the famous decomposition theorem of Beilinson, Bernstein and Deligne \cite{BBD}.  
\begin{proposition}
\label{weightPreserve}
If $\tau$ is proper, the tensor product $\boxtimes_{\tau}$ takes pairs of pure objects to pure objects.  If $\tau$ is moreover finite, then $\boxtimes_{\tau}$ is biexact.
\end{proposition}
\begin{proof}
The map $\tau\times +$ defining $\boxtimes_{\tau}$ can be factorised as $(\tau\times\id_{\mathbb{A}^1})(\id_{X\times X}\times +)$.  By our assumption on $\tau$, the map $(\tau\times \id_{\mathbb{A}^1})$ is proper, and so its associated direct image functor preserves purity by \cite[(4.5.2), (4.5.4)]{Sai90}, while $(\id_{X\times X}\times +)$ is the map defining the external tensor product on the category $\Dub(\MMHM(X))$, and the associated direct image functor preserves the weight filtration by Proposition \ref{tensPres}.  The biexactness follows similarly from Proposition \ref{tensPres}, and the fact that the direct image functor for finite morphisms is exact.
\end{proof}

Let $X$ be a smooth variety, and let $f$ be a regular function on $X$.  We denote by $X_0$ the preimage of zero under $f$, and set $X_{\leq 0}=f^{-1}(\mathbb{R}_{\leq 0})$.  We define the underived functor
\[
\Gamma_{X_{\leq 0}}\mathcal{F}(U)=\ker\left(\mathcal{F}(U)\rightarrow \mathcal{F}(U\setminus (U\cap X_{\leq 0}))\right), 
\]
and set $\phi_f\mathcal{F}=(R\Gamma_{X_{\leq 0}}\mathcal{F})[1]|_{X_0}$.  By the construction of Saito's category $\Dub(\MHM(X))$, the functor $\phi_f\colon \Dub(X)\rightarrow \Dub(X)$, defined at the level of derived categories of Abelian sheaves with constructible cohomology, lifts to a functor $\phi_f\colon \Dub(\MHM(X))\rightarrow \Dub(\MHM(X))$.  As in \cite[Sec.7]{COHA}, we define the functor
\begin{align*}
\phim{f}\colon&\Dub(\MHM(X))\rightarrow \Dub(\MMHM(X))\\
&\mathcal{F}\mapsto (X\times\mathbb{G}_m\rightarrow X\times \mathbb{A}^1)_!\phi_{f/u}(X\times\mathbb{G}_m\rightarrow X)^*\mathcal{F}
\end{align*}
where $u$ is a coordinate on $\mathbb{G}_m$.  
\begin{proposition}\label{exactness}
The functor $\phim{f}$ takes objects of $\MHM(X)$ to objects of $\MMHM(X)$; in other words it is exact with respect to the natural t structures on $\Dub(\MHM(X))$ and $\Dub(\MMHM(X))$.  
\end{proposition}
\begin{proof}
Exactness follows from the corresponding statement at the level of perverse sheaves.  The functors $(X\times\mathbb{G}_m\rightarrow X)^*[1]$ and $\phi_{f/u}[-1]$ are exact \cite{BBD}, and $(X\times\mathbb{G}_m\rightarrow X\times \mathbb{A}^1)_!$ is left exact, since $(X\times\mathbb{G}_m\rightarrow X\times \mathbb{A}^1)$ is affine, and right exact, since $(X\times\mathbb{G}_m\rightarrow X\times \mathbb{A}^1)$ is quasi-finite.
\end{proof}

For $X$ a not necessarily smooth complex quasiprojective variety, we may define $\phim{f}=i^*\phim{\overline{f}}i_*$ where $i$ is a closed embedding into a smooth variety $\overline{X}$, and $\overline{f}$ is a function on $\overline{X}$ extending $f$.  There is a natural isomorphism $\DD_{X}^{\mon}\phi^{\mon}_f\cong \phi^{\mon}_f\DD_X$ by the main theorem of \cite{Sai89duality}.  If $p\colon X\rightarrow Y$ is a proper map of varieties, and $f$ is a regular function on $Y$, then there is a natural isomorphism $\phim{f}p_*\cong p_*\phim{fp}$ by \cite[Thm.2.14]{Sai90}.  

Let $X$ and $Y$ be a pair of complex varieties, and let $f$ and $g$ be regular functions on them.  Then by the Thom--Sebastiani theorem due to Saito \cite{Saito10}, proved at the level of underlying perverse sheaves by Massey \cite{Ma01}, there is a natural equivalence of functors 
\begin{equation}
\label{TSiso}
\left(\phim{f\boxplus g}(\mathcal{F}\boxtimes \mathcal{G})\right)|_{f,g=0}\cong \phim{f}\mathcal{F}\boxtimes \phim{g}\mathcal{G}\colon \MHM(X)\times\MHM(Y)\rightarrow \MMHM(X\times Y).
\end{equation}
For further discussion of this isomorphism, as well as the compatibility of these two versions of the Thom--Sebastiani theorem, we refer the reader to Sch\"{u}rmann's appendix to \cite{Br12}.

\begin{definition}\label{trivMonDef}
We say an object $\mathcal{F}\in\Dub(\MMHM(X))$ has \textit{trivial monodromy} if it is in the essential image of the map 
\[
(X\times\{0\}\xrightarrow{z} X\times\mathbb{A}^1)_*\colon \Dub(\MHM(X))\rightarrow\Dub(\MMHM(X)).
\]
\end{definition}

Let $\mathbb{L}=\HO_c(\mathbb{A}^1,\mathbb{Q})$, i.e. $\mathbb{L}$ is the pure one-dimensional Hodge structure concentrated in cohomological degree 2.  The category $\Dub(\MMHM(\pt))$ contains a square root of $\LL$; we set $\LL^{1/2}:=\HO_c(\mathbb{A}^1,\phim{x^2}\mathbb{Q}_{\mathbb{A}_1})$, and we have $\LL^{1/2}\otimes \LL^{1/2}\cong \LL$ via the Thom--Sebastiani theorem.  The monodromic mixed Hodge module $\LL^{1/2}$ is pure and has perverse degree 1; it is given explicitly by $j_!\mathcal{L}$, where $\mathcal{L}$ is the rank one local system on $\mathbb{C}^*$ with monodromy given by multiplication by $-1$, and $j\colon\mathbb{C}^*\rightarrow\mathbb{A}^1$ is the inclusion.  
\begin{remark}
\label{monRem}
Note that $\LL^{1/2}$ does not have trivial monodromy, and in fact there is no square root of $\LL$ considered as an object of $\Dub(\MHM(\pt))$.  It follows that if we have a direct sum decomposition in $\Dub(\MMHM(\pt))$
\[
\mathcal{H}\cong\bigoplus_{g\in\mathbb{Z}}\left(\mathbb{L}^{g/2}\right)^{\oplus c_g},
\]
with $c_g\in\mathbb{N}$ for all $g$, then $\mathcal{H}$ has trivial monodromy if and only if $c_g=0$ for all odd $g$.
\end{remark}
In what follows, if $X$ is a connected irreducible algebraic variety, we set 
\[
\ICS_X(\mathbb{Q})=\IC_X(\mathbb{Q}_{X_{\mathrm{reg}}})\otimes\LL^{{}-\dim(X)/2}, 
\]
i.e. if $\dim(X)$ is even, we shift the usual intersection complex mixed Hodge module so that its underlying complex of perverse sheaves is in the natural heart of $\Dub(\Perv(X))$, and considered as an object in $\Dub(\MHM(X))$, the object $\ICS_X(\mathbb{Q})$ is pure.  In the odd case we are doing the same thing, but only after passing to the larger category of monodromic mixed Hodge modules on $X$.  If $X$ is a disjoint union of irreducible algebraic varieties, we set
\[
\ICS_X(\mathbb{Q})=\bigoplus_{Y\in \pi_0(X)}\ICS_Y(\mathbb{Q}).
\]
We set 
\[
\HO(X,\mathbb{Q})_{\vir}:=\Ho((X\rightarrow \pt)_*\ICS_X(\mathbb{Q})).
\]
If $X$ is a proper variety, then $\HO(X,\QQ)_{\vir}$ is pure.  We extend this notation by setting 
\[
\HO(\BC,\QQ)_{\vir}:=\HO(\BC,\QQ)\otimes\LL^{1/2}.
\]
We say a monodromic mixed Hodge module $\mathcal{F}\in\MMHM(\pt)$ is of \textit{Tate type} if it is obtained by taking iterated extensions of the monodromic mixed Hodge modules $\LL^{g/2}[g]$ for $g\in\mathbb{Z}$.  We say an object $\mathcal{F}\in\Dub(\MMHM(\pt))$ is of Tate type if each $\Ho^p(\mathcal{F})$ is.  If $X$ is a disjoint union of points, then there is a natural equivalence of categories
\begin{equation}
\label{pointDecomp}
\Dub(\MMHM(X))\cong\prod_{x\in X}\Dub(\MMHM(\pt)),
\end{equation}
and we say an object $\mathcal{F}\in\Dub(\MMHM(X))$ is of Tate type if each of its factors under the equivalence (\ref{pointDecomp}) is.

\subsection{Moduli spaces of quiver representations and stability conditions}
Let $Q$ be an ice quiver.  Recall that we always identify the vertices of $Q$ with the numbers $\{1,\ldots,n\}$, and set the vertices $\{m+1,\ldots,n\}$ to be the frozen vertices, in the sense explained at the start of Section \ref{cluSec}.  Let $\dd\in\mathbb{N}^{m}$ be a dimension vector, supported on the principal part of $Q$.  We define
\[
X(Q)_\dd:=\prod_{a\in Q_1}\Hom(\mathbb{C}^{\dd_{t(a)}},\mathbb{C}^{\dd_{s(a)}}),
\]
and 
\[
X(Q)=\coprod_{\dd\in\mathbb{N}^{m}}X(Q)_{\dd}, 
\]
and define 
\[
G_\dd:=\prod_{i\leq m}\Gl_{\dd_i}.  
\]
The group $G_\dd$ acts on $X(Q)_\dd$ by change of basis of each of the spaces $\mathbb{C}^{\dd_i}$.  We let $\Mst(Q)_\dd$ denote the stack of $\dd$-dimensional right modules of $\mathbb{C}Q$.  Then
\[
\Mst(Q)_\dd\cong X(Q)_\dd/G_\dd,
\]
where we take the stack-theoretic quotient.  We set
\[
\Mst(Q):=\coprod_{\dd\in\mathbb{N}^{m}}\Mst(Q)_\dd.
\]
Note that we will only ever consider moduli stacks of modules supported on the principal parts of quivers.

A (Bridgeland) \textit{stability condition} for the quiver $Q$ is an element $\zeta\in\mathbb{H}_+^{n}$, where 
\[
\mathbb{H}_+:=\{r\exp(i\theta)|\theta\in[0,\pi),r\in \mathbb{R}_{>0}\}.  
\]

Given a stability condition $\zeta$, we define $\Z^{\zeta}\colon\KK(\rmod{\mathbb{C}Q})\rightarrow \mathbb{C}$ by 
\[
\Z^{\zeta}([M])=\zeta\cdot\DIM(M),
\]
and we define the \textit{slope} $\Mu^{\zeta}(M)\in [0, 2\pi)$ of an object $M\in\Dub^{\fd}(\Rmod{\mathbb{C}Q})$ satisfying $\Z^{\zeta}([M])\neq 0$ to be the argument of $\Z^{\zeta}([M])$.  If $0\neq M\in\rmod{\mathbb{C}Q}$ then $\Mu^{\zeta}(M)\in [0,\pi)$.  Similarly, for nonzero $\dd\in\mathbb{N}^n$, we define $\Mu^{\zeta}(\dd)\in [0,\pi)$ to be the argument of $\Z^{\zeta}(\dd)$.  A $\mathbb{C}Q$-module $M$ is $\zeta$\textit{-stable} if for all proper submodules $M'\subset M$ we have the inequality $\Mu^{\zeta}(M')<\Mu^{\zeta}(M)$.  If it is only true that $\Mu^{\zeta}(M')\leq \Mu^{\zeta}(M)$, for all proper submodules $M'\subset M$, we say that $M$ is $\zeta$\textit{-semistable}.

Fix a dimension vector $\dd\in\mathbb{N}^m$.  If $\zeta\in \mathbb{H}_+^n$, we can first of all replace $\zeta$ with a stability condition in $(i+\mathbb{Q})^n$ such that the sets of $\zeta$-stable and $\zeta$-semistable $\dd$-dimensional modules are unchanged and such that $\RRe(\zeta\cdot \dd)=0$, by \cite[Lem.4.21]{DMSS13}.  Then we can pick a $N\in\mathbb{N}$ such that $N\RRe(\zeta_s)\in\mathbb{Z}$ for all $s\leq m$.  We linearise the $G_\dd$-action on $X(Q)_\dd$ via the character 
\[
\chi\colon (g_s)_{s\leq m}\mapsto\prod_{s\leq m}\det(g_s)^{N\RRe(\zeta_s)}
\]
and define $X(Q)^{\zeta\sst}_\dd$ to be the scheme of semistable points with respect to this linearisation.  By \cite{King94}, using the geometric invariant theory constructions of \cite{MFK94}, the GIT quotient $X(Q)_{\dd}/\!\!/_{\chi} G_\dd$ provides a coarse moduli space of $\zeta$-semistable $\dd$-dimensional right $\mathbb{C}Q$-modules.  For $K$ a field extension of $\mathbb{C}$, the $K$-points of this moduli space are in bijection with isomorphism classes of direct sums of $\zeta$-stable $KQ$-modules \cite[Prop.3.2]{King94} of the same slope.  We denote this GIT quotient by $\Msp(Q)^{\zeta\sst}_\dd$, and define
\[
\Msp(Q)^{\zeta\sst}:=\coprod_{\dd\in\mathbb{N}^{m}}\Msp(Q)_{\dd}^{\zeta\sst}.
\]
We denote by 
\begin{equation}
\label{pdef}
p^{\zeta}_{\dd}\colon \Mst(Q)^{\zeta\sst}_\dd\rightarrow\Msp(Q)^{\zeta\sst}_\dd
\end{equation}
the map from the stack-theoretic quotient to the coarse moduli space.

We abbreviate $\Msp(Q)_{\dd}:=\Msp(Q)_{\dd}^{\zeta_{\deg}}$, where $\zeta_{\deg}$ is the degenerate stability condition $(i,\ldots,i)$.  In words: if a stability condition is missing from a coarse moduli space, it is defined to be the coarse moduli space of semisimple modules or, equivalently, the affinization.  We denote by 
\begin{equation}
q_{\dd}^{\zeta}\colon\Msp(Q)^{\zeta\sst}_{\dd}\rightarrow \Msp(Q)_{\dd}
\end{equation}
the map to the affinization.
\begin{convention}
Wherever a space, map, or monodromic mixed Hodge module is defined with respect to a subscript $\dd$ denoting a dimension vector, and that dimension vector is replaced by a slope $\gamma\in [0,\pi)$, the direct sum over all $\dd\in\mathbb{N}^m$ satisfying $\dd=0$ or $\Mu^{\zeta}(\dd)=\gamma$ is intended.  If the subscript is missing altogether, then the direct sum over all $\dd\in\mathbb{N}^m$ is intended.
\end{convention}
Given a stability condition $\zeta\in\mathbb{H}_+^n$, a finite-dimensional $\mathbb{C}Q$-module $M$ admits a unique filtration (the Harder--Narasimhan filtration)
\[
0=M_0\subset M_1\subset\ldots\subset M_h=M
\]
such that each subquotient $M_g/M_{g-1}$ is $\zeta$-semistable, and the slopes of the subquotients $M_1,M_2/M_1,\ldots,M_h/M_{h-1}$ are strictly decreasing.  

For $\dd\in\mathbb{N}^{m}$, let $\HN_{\dd}$ denote the set of Harder--Narasimhan types for $\dd$, i.e. the set of tuples of nonzero dimension vectors $(\dd^1,\ldots,\dd^h)$ (for varying $h$) such that $\sum_{1\leq g\leq h} \dd^g=\dd$ and $\Mu^{\zeta}(\dd^g)>\Mu^{\zeta}(\dd^{g+1})$ for all $g<h$.  We define
\[
\HN=\coprod_{\dd\in\mathbb{N}^{m}}\HN_\dd.
\]
For $\overline{\dd}\in \HN_{\dd}$ we let $X(Q)_{\overline{\dd}}^{\zeta}\subset X(Q)_{\dd}$ denote the locally closed subvariety of points for which the Harder--Narasimhan filtration of the associated module, with respect to the stability condition $\zeta$, is in $\overline{\dd}$, and denote by $\Mst(Q)^{\zeta}_{\overline{\dd}}$ the corresponding stack.  The space $X(Q)$ admits a decomposition into locally closed subvarieties
\[
X(Q)=\coprod_{\overline{\dd}\in\HN}X(Q)^{\zeta}_{\overline{\dd}},
\]
and the stack $\Mst(Q)$ admits a decomposition into locally closed substacks
\[
\Mst(Q)= \coprod_{\overline{\dd}\in\HN}\Mst(Q)^{\zeta}_{\overline{\dd}}
\]
by \cite[Prop.3.4]{Reineke_HN}.  If $S\subset [0,\pi)$ is an interval, we denote by $X(Q)^{\zeta}_S\subset X(Q)$ the locally closed subvariety whose points correspond to modules for which the Harder--Narasimhan type $\overline{\dd}=(\dd^1,\ldots,\dd^h)$ satisfies $\Mu^{\zeta}(\dd^g)\in S$ for all $g\leq h$.  We define 
\[
X(Q)^{\zeta}_{S,\dd}:=X(Q)^{\zeta}_S\cap X(Q)_\dd
\]
and define $\Mst(Q)^{\zeta}_{S,\dd}$ likewise.  We denote by
\begin{equation}
\label{pSdef}
p^{\zeta}_{S,\dd}\colon \Mst(Q)_{S,\dd}^{\zeta}\rightarrow\Msp(Q)_{\dd}
\end{equation}
the map to the coarse moduli space (equivalently, the affinization, as the target is given the degenerate stability condition).  As usual, if the dimension vector $\dd$ is missing from (\ref{pSdef}), the disjoint union over all $\dd\in\mathbb{N}^m=\mathbb{N}^{(Q_{\princ})_0}$ is intended, although the moduli stack $\Mst(Q)^{\zeta}_{S,\dd}$ will be empty if $\dd\neq 0$ and $\Mu^{\zeta}(\dd)\notin S$.

Let $W\in\mathbb{C}Q/[\mathbb{C}Q,\mathbb{C}Q]$ be an algebraic potential.  Then taking the trace defines a function $\Tr(W)$ on $X(Q)$, defined by
\[
\Tr\left(\sum_{c \textrm{ a cycle in }Q}a_cc\right)\colon (\rho)\mapsto \sum_c a_c\Tr(\rho(c)).
\]
The restriction of this function to $X(Q)_\dd$ is $G_\dd$-invariant, and so $\Tr(W)$ induces a function 
\[
\mathfrak{Tr}(W)\colon \Mst(Q)\rightarrow \mathbb{C}.  
\]
We denote by 
\[
\WW\colon\Msp(Q)\rightarrow \mathbb{C}
\]
the induced function on the coarse moduli space.  As substacks of $\Mst(Q)_{\dd}$, there is an equality between $\crit(\WWW_{\dd})$ and the stack of $\dd$-dimensional right $\Jac(Q,W)$-modules.
\begin{convention}
\label{QWconvention}
If $\bullet(Q)^{\ldots}_{\ldots}$ is one of the spaces defined above, for which there is a natural map $\bullet(Q)^{\ldots}_{\ldots}\rightarrow\Mst(Q)$, we define
\[
\bullet(Q,W)^{\ldots}_{\ldots}:=\bullet(Q)^{\ldots}_{\ldots}\times_{\Mst(Q)} \crit(\WWW).
\]
Similarly, we define 
\[
\bullet(Q)^{\ldots,\nilp}_{\ldots}:=\bullet(Q)^{\ldots}_{\ldots}\times_{\Msp(Q)} \Msp(Q)^{\nilp}
\]
where $\Msp(Q)^{\nilp}\cong\mathbb{N}^m$ is the reduced vanishing locus of the infinite set of functions $\{\mathcal{T}r(c)|\hbox{ }c\textrm{ a nontrivial cycle in }Q\}$.  If $\bullet(Q)^{\ldots}_{\ldots}$ is a stack and not a scheme, we denote by $\Dim_{\ldots}^{\ldots}$ the map from $\bullet(Q)^{\ldots}_{\ldots}$ to $\mathbb{N}^{m}$ taking a connected component to its dimension vector.  If $\bullet(Q)^{\ldots}_{\ldots}$ is a scheme we denote this map $\dim_{\ldots}^{\ldots}$.  When the domain of the maps $\Dim_{\ldots}^{\ldots}$ or $\dim_{\ldots}^{\ldots}$ are clear, we will omit the superscripts and subscripts, to ease the notation.  If $\mathcal{F}\in\Dub(\MMHM(\bullet(Q)^{\ldots}_{\ldots}))$, then we define
\[
\mathcal{F}_{\nilp}:=\mathcal{F}|_{\bullet(Q)^{\ldots,\nilp}_{\ldots}}.
\]
\end{convention}

\subsection{Categorification of the completed quantum space $\hat{\Ror}_Q$}
The space $\Msp(Q)$ is a monoid in the category of schemes, with monoid map 
\[
\oplus\colon\Msp(Q)\times\Msp(Q)\rightarrow \Msp(Q)
\]
acting, at the level of points, by taking a pair of semisimple modules to their direct sum.  This map is finite by \cite[Lem.2.1]{MeRe14} (for the sake of completeness we reprove this as Lemma \ref{finLem} below).  It follows that the monoidal product $\boxtimes_{\oplus}$ on the category $\Dulf\left(\MMHM(\Msp(Q))\right)$, defined in Section \ref{MMHMSec}, is bi-exact, and preserves pure objects, by Proposition \ref{weightPreserve}.  

The map of schemes 
\[
\dim\colon\Msp(Q)\rightarrow\mathbb{N}^{m}
\]
taking a $\mathbb{C}Q_{\princ}$-module to its dimension vector is a map of monoids, where the monoidal structure on $\mathbb{N}^m$ is provided by addition.  It follows that the functor
\begin{align*}
&\dim_*\colon \Dulf\left(\MMHM(\Msp(Q))\right)\rightarrow\Dulf(\MMHM(\mathbb{N}^{m}))
\end{align*}
is a monoidal functor, where the domain category carries the monoidal product $\boxtimes_{\oplus}$, and the target carries the monoidal product $\boxtimes_+$.  The map 
\[
\Dim\colon \Mst(Q)\rightarrow\mathbb{N}^m
\]
is the composition $\dim\circ p$, with $p$ defined as in (\ref{pSdef}).

We next explain how to twist the above monoidal structures to form a categorification of the completed algebra $\hat{\Ror}_{Q}$.  Given an object $\mathcal{F}\in\Dulf\left(\MMHM(\Msp(Q))\right)$, we denote by $\mathcal{F}_\dd$ the summand supported on $\Msp(Q)_{\dd}$.  We define the \textit{twisted monoidal product}
\begin{equation}
\label{monDef}
\mathcal{F}\boxtimes_{\oplus}^{\tw}\mathcal{G}:=\bigoplus_{\dd',\dd''\in\mathbb{N}^{m}}\mathcal{F}_{\dd'}\boxtimes_{\oplus}\mathcal{G}_{\dd''}\otimes\mathbb{L}^{\langle \dd'',\dd'\rangle_Q /2}.
\end{equation}
We define the twisted monoidal product on $\Dulf(\MMHM(\mathbb{N}^m))$ via equation (\ref{monDef}) again, and so the functor $\dim_*$ is again a monoidal functor, if we give the domain and the target categories the \textit{twisted} monoidal product.  
\begin{remark}
\label{vswap}
By skew-symmetry of $\langle\bullet,\bullet\rangle_Q$ and properness of $\oplus$, the twisted monoidal product commutes with Verdier duality, after swapping the arguments; i.e. there is a natural isomorphism
\[
\DD^{\mon}_{\Msp(Q)}\mathcal{F}\boxtimes_{\oplus}^{\tw}\DD^{\mon}_{\Msp(Q)}\mathcal{G}\cong\DD^{\mon}_{\Msp(Q)}\left(\mathcal{G}\boxtimes_{\oplus}^{\tw}\mathcal{F}\right).
\]
\end{remark}
We define the \textit{integration map}
\begin{align}\label{IM}
\chi_{Q}\colon& \KK\left(\Dulf(\MMHM(\mathbb{N}^{m}))\right)\rightarrow\hat{\Ror}_Q\\
&[\mathcal{F}]\mapsto\sum_{\dd\in\mathbb{N}^{m}}\chi_{q}([\mathcal{F}_\dd],-q^{1/2})Y^{\dd},\label{dulfjust}
\end{align}
where for $\mathcal{G}\in\Dulf(\MMHM(\pt))$,
\[
\chi_q([\mathcal{G}],q^{1/2}):=\sum_{i\in\mathbb{Z}}\sum_{r\in \mathbb{Z}}(-1)^i\dim\left(\Gr_W^{r}(\Ho^i(\mathcal{G}))\right)q^{r/2}.
\]
We restrict to $\Dulf(\MMHM(\mathbb{N}^m))$ so that the coefficients on the right hand side of (\ref{dulfjust}) belong to $\mathbb{Z}((q^{1/2}))$.  
\begin{remark}\label{Lsign}
Note that $\chi_q(\mathbb{L}^{1/2},q^{1/2})=-q^{1/2}$ since $\mathbb{L}^{1/2}$ is pure of weight one.  The sign that appears in (\ref{dulfjust}) means that $\LL^{1/2}$ plays the role of $q^{1/2}$ in the theory of quantum cluster mutation.  The `classical limit', recovering the theory of commutative cluster algebras from quantum cluster algebras, is given by $q^{1/2}=1$.  This is \textit{not} the traditional classical limit of motivic Donaldson--Thomas theory, see e.g. \cite[Sec.7.1]{KS}, which is obtained by setting $q^{1/2}=-1$, or equivalently, taking the Euler characteristic.  This choice of signs means that quantum cluster positivity follows directly from purity:
\end{remark}
\begin{proposition}
Let $\mathcal{F}\in\Dulf(\MMHM(\mathbb{N}^m))$ be pure.  Then 
\[
\chi_Q([\mathcal{F}])\in \mathbb{N}((q^{1/2}))[[Y^{1_s}|s\leq m]],
\]
i.e. $\chi_Q([\mathcal{F}])$ has only positive coefficients.
\end{proposition}

The following proposition, which is a consequence of Proposition \ref{weightPreserve}, explains the sense in which $\Dulf(\MMHM(\mathbb{N}^{m}))$ is a categorification of $\hat{\Ror}_Q$.

\begin{proposition}
The map $\chi_{Q}$ is a homomorphism of rings, where $\KK\left(\Dulf(\MMHM(\mathbb{N}^{m}))\right)$ is given the noncommutative product induced by the twisted monoidal product $\boxtimes_{\oplus}^{\tw}$.
\end{proposition}

\subsection{Critical cohomology}
Let $X$ be a smooth complex variety, and let $\overline{f}$ be a regular function on $X$.  Furthermore, let $X$ carry a $G$-action, for $G$ an algebraic group, such that $\overline{f}$ is invariant with respect to the $G$-action, inducing a function $f$ on the stack-theoretic quotient $X/G$.  Let $\overline{p}\colon X\rightarrow Y$ be a map of varieties, which is also $G$-invariant, inducing a map $p\colon X/G\rightarrow Y$.  We recall the definition of $\Ho(p_*\phim{f}\mathbb{Q}_{X/G})$ and $\Ho(p_!\phim{f}\mathbb{Q}_{X/G})$, which is a relative version of the definition of the equivariant cohomology of the vanishing cycle complex from e.g. \cite[Sec.7]{COHA}, see \cite[Sec.2.2]{DaMe15b} for a fuller discussion.  

For simplicity we will assume that $X$ is equidimensional --- if $\Mst$ is obtained by taking the disjoint union of a number of smooth connected global quotient stacks $X_i/G_i$, we define 
\[
\Ho(p_*\phim{f}\mathbb{Q}_{\Mst})=\bigoplus_{X_i/G_i\in \pi_0(\Mst)}\Ho\left(p|_{X_i/G_i,*}\left(\phim{f|_{X_i/G_i}}\mathbb{Q}_{X_i/G_i}\right)\right)
\]
and we extend all related definitions in the same manner.  

Let $V_1\subset V_2\subset\ldots$ be an ascending sequence of finite-dimensional $G$-representations, which we identify with the total spaces of their underlying vector spaces, considered as $G$-equivariant varieties.  Assume that we have a sequence $U_1\subset U_2\subset\ldots$ of $G$-equivariant varieties satisfying the following conditions:
\begin{enumerate}
\item
$U_N\subset X\times V_N$ for all $N$, and $U_N$ is acted on scheme-theoretically freely by $G$.
\item
$\codim_{X\times V_N}\left((X\times V_N)\setminus U_N\right)\mapsto\infty$ as $N\mapsto \infty$.
\item
The map $\pi_N\colon U_N\rightarrow U_N/G$ is a principal $G$-bundle in the category of schemes.
\end{enumerate}
We denote by $f_N\colon U_N/G\rightarrow \mathbb{C}$ the induced function, and $p_N\colon U_N/G\rightarrow Y$ the induced map.  We define
\begin{align*}
\Ho(p_*\phim{f}\mathbb{Q}_{X/G}):=&\lim_{N\mapsto \infty}\Ho(p_{N,*}\phim{f_N}\mathbb{Q}_{U_N/G})\\
\Ho(p_!\phim{f}\mathbb{Q}_{X/G}):=&\lim_{N\mapsto \infty}\left(\Ho(p_{N,!}\phim{f_N}\mathbb{Q}_{U_N/G})\otimes \mathbb{L}^{-\dim(V_i)}\right)\\
\Ho(p_*\phim{f}\ICS_{X/G}(\mathbb{Q})):=&\Ho(p_*\phim{f}\mathbb{Q}_{X/G})\otimes\mathbb{L}^{(\dim(G)-\dim(X))/2}\\
\Ho(p_!\phim{f}\ICS_{X/G}(\mathbb{Q})):=&\Ho(p_!\phim{f}\mathbb{Q}_{X/G})\otimes\mathbb{L}^{(\dim(G)-\dim(X))/2}.
\end{align*}
In the first two equations, the limit exists because in each fixed cohomological degree, the cohomology stabilises for sufficiently large $N$ by our codimension assumption on $U_N$ --- see \cite[Sec.2]{DaMe15b}.  Since in the first equation, the right hand side vanishes in fixed sufficiently low degree, and by \cite[Prop.2.26]{Sai90}, the direct image increases weight, it follows that $\Ho(p_*\phim{f}\mathbb{Q}_{X/G})\in\Dulf(\MMHM(Y))$, and also that $\Ho(p_!\phim{f}\mathbb{Q}_{X/G})\in\Dllf(\MMHM(X))$ by commutativity of $\phi^{\mon}_{f_N}$ with the Verdier duality functor \cite{Sai89duality}.

Let $Z\subset X$ be a $G$-equivariant subvariety.  We define $Z_N:=(Z\times V_N)\cap U_N$.  Then we define 
\begin{align*}
\Ho\left(p_*\left(\phim{f}\mathbb{Q}_{X/G}\right)_{Z/G}\right):=&\lim_{N\mapsto \infty}\Ho\left(p_{N,*}\left(\phim{f_N}\mathbb{Q}_{U_N/G}\right)_{Z_N/G}\right)\\
\Ho\left(p_!\left(\phim{f}\mathbb{Q}_{X/G}\right)_{Z/G}\right):=&\lim_{N\mapsto \infty}\left(\Ho\left(p_{N,!}\left(\phim{f_N}\mathbb{Q}_{U_N/G}\right)_{Z_N/G}\right)\otimes \mathbb{L}^{-\dim(V_i)}\right)\\
\Ho\left(p_*\left(\phim{f}\ICS_{X/G}(\mathbb{Q})\right)_{Z/G}\right):=&\Ho\left(p_*\left(\phim{f}\mathbb{Q}_{X/G}\right)_{Z/G}\right)\otimes\mathbb{L}^{(\dim(G)-\dim(X))/2}\\
\Ho\left(p_!\left(\phim{f}\ICS_{X/G}(\mathbb{Q})\right)_{Z/G}\right):=&\Ho\left(p_!\left(\phim{f}\mathbb{Q}_{X/G}\right)_{Z/G}\right)\otimes\mathbb{L}^{(\dim(G)-\dim(X))/2}.
\end{align*}

\begin{definition}\label{properMapsDef}
We say that the map $p\colon X/G\rightarrow Y$ can be \textit{approximated by proper maps} if we can pick a system of $U_N$, continuing the notation from above, such that each of the maps $p_N\colon U_N/G\rightarrow Y$ is a  proper map.
\end{definition}
For $p\colon X/G\rightarrow Y$ a map to a variety, and $f$ a regular function on $X/G$, we only define the total cohomology of the pushforward of vanishing cycles to $Y$.  There is no a priori obvious comparison between $\Ho\left(\tau_*\Ho(p_*\phim{f}\mathbb{Q}_{X/G})\right)$ and 
\[
\HO_G(X,\phim{\overline{f}}\mathbb{Q}):=\Ho((\tau p)_*\phim{f}\mathbb{Q}_{X/G}), 
\]
where the equivariant cohomology of the vanishing cycle complex is defined as in \cite{COHA}, $\tau\colon Y\rightarrow\pt$ is the structure morphism, and $\overline{f}$ is the induced function on $X$.  By contrast, in the case of morphisms that are approximated by proper maps, we have the following degeneration result, which is one of the many uses of this notion.
\begin{lemma}
Assume that the map $p\colon X/G\rightarrow Y$ from the smooth stack $X/G$ to the variety $Y$ is approximated by proper maps, $f$ is a regular function on $X/G$ which can be written as $f'p$ for some function $f'$ on $Y$, and $\tau\colon Y\rightarrow \pt$ is the structure morphism.  Then there is a (noncanonical) isomorphism
\[
\Ho((\tau p)_*\phim{f}\mathbb{Q}_{X/G})\cong\Ho\left(\tau_*\Ho(p_*\phim{f}\mathbb{Q}_{X/G})\right).
\]
\end{lemma}
\begin{proof}
Fix a cohomological degree $l$, and fix a sufficiently large $N\in\mathbb{N}$.  Then there is a chain of isomorphisms
\begin{align*}
\Ho^l((\tau p)_{N,*}\phim{f_N}\mathbb{Q}_{U_N/G})= &\Ho^l((\tau p_N)_*\phim{f_N}\mathbb{Q}_{U_N/G})
\\ \cong&\Ho^l(\tau_* \phim{f'} p_{N,*}\mathbb{Q}_{U_N/G})&\textrm{properness of }p_N
\\ \cong&\Ho^l(\tau_* \phim{f'} \Ho(p_{N,*}\mathbb{Q}_{U_N/G}))&\textrm{decomposition theorem}
\\ \cong&\Ho^l(\tau_* \Ho(\phim{f'} p_{N,*}\mathbb{Q}_{U_N/G})) &\textrm{exactness of }\phim{f'}
\\\cong&\Ho^l(\tau_* \Ho( p_{N,*}\phim{f_N}\mathbb{Q}_{U_N/G})) &\textrm{properness of }p_{N},
\end{align*}
and then the lemma follows, since once $\Ho^{l'}( p_{N,*}\phim{f_N}\mathbb{Q}_{U_N/G})$ stabilises for all $l'\leq l+\dim(Y)$, the final term in the chain of isomorphisms stabilises to $\Ho^l(\tau_* \Ho( p_{*}\phim{f}\mathbb{Q}_{X/G}))$.
\end{proof}
We now explain how certain maps from moduli stacks of $\mathbb{C}Q$-modules to the corresponding coarse moduli spaces are approximated by proper maps.
\begin{definition}[Framed quiver]
Let $Q$ be an ice quiver, and let $\ff\in\mathbb{N}^{n}$ be a dimension vector.  We define $Q_{\ff}$ by 
\begin{itemize}
\item
$(Q_{\ff})_0=Q_0\coprod \{\infty\}$
\item
$(Q_{\ff})_1=Q_1\coprod\{\beta_{i,g}|i\in Q_0,1\leq g\leq \ff_i\}$
\end{itemize}
where $t(\beta_{i,g})=\infty$ and $s(\beta_{i,g})=i$.  We set $\infty$ to be a principal vertex of $Q_{\ff}$.
\end{definition}
In what follows, for dimension vectors $\ff\in\mathbb{N}^n$, we write $\ff\mapsto \infty$ to mean that each component of $\ff$ is taken to be arbitrarily large.

Fix a dimension vector $\dd\in\mathbb{N}^{m}$.  Let 
\begin{equation}
\label{VDef}
V_{\ff,\dd}=\prod_{s\leq m}\Hom(\mathbb{C}^{\ff_s},\mathbb{C}^{\dd_s}).  
\end{equation}
Then $V_{\ff,\dd}$ carries a $G_{\dd}$-action via the $\Gl_{\dd_s}$-actions on the vector spaces $\mathbb{C}^{\dd_s}$.  Define $V_{\ff,\dd}^{\surj}\subset V_{\ff,\dd}$ by $V_{\ff,\dd}^{\surj}=\prod_{s\leq m}\Hom^{\surj}(\mathbb{C}^{\ff_s},\mathbb{C}^{\dd_s})$.  Composition of linear maps provides a surjection of topological spaces
\[
\Hom(\mathbb{C}^{\ff_s},\mathbb{C}^{\dd_s-1})\times \Hom(\mathbb{C}^{\dd_s-1},\mathbb{C}^{\dd_s})\rightarrow  \left(\Hom(\mathbb{C}^{\ff_s},\mathbb{C}^{\dd_s})\setminus \Hom^{\surj}(\mathbb{C}^{\ff_s},\mathbb{C}^{\dd_s})\right)
\]
for which the domain has dimension $\ff_s\dd_s-\ff_s+\dd_s^2-\dd_s$ and so we deduce
\begin{align}
\label{codim}
\codim_{V_{\ff,\dd}}\left(V_{\ff,\dd}\setminus V^{\surj}_{\ff,\dd}\right)\geq \lvert\ff\rvert-\dd\cdot\dd+\lvert\dd\rvert\mapsto \infty&&\textrm{as }\ff\mapsto\infty.
\end{align}
We write $(1,\dd)$ for the dimension vector for $Q_\ff$ that is $1$ at the framing vertex $\infty$, and $\dd$ when restricted to the quiver $Q$.

\begin{definition}
\label{ztDef}
Let $\zeta\in\mathbb{H}_+^n$ be a stability condition.  Fix $\theta\in [0,\pi)$ a slope.  We extend $\zeta$ to a stability condition $\zeta^{(\theta)}$ for $Q_\ff$ by setting the argument of $\zeta^{(\theta)}_{\infty}$ to be equal to $\theta+\epsilon$ for $0<\epsilon\ll 1$, and making $|\zeta^{(\theta)}_{\infty}|$ very large (possibly depending on $\dd$).
\end{definition}
For a $(1,\dd)$-dimensional $\mathbb{C}Q_\ff$-module $\rho$, $\zeta^{(\theta)}$-stability is equivalent to the following two conditions:
\begin{enumerate}
\item
If $\rho'$ is the underlying $\mathbb{C}Q$-module of $\rho$, then every $\dd'$ in the Harder--Narasimhan type of $\rho'$ has slope less than or equal to $\theta$.
\item
If $\rho''\subset \rho$ is the smallest sub $\mathbb{C}Q_\ff$-module of $\rho$ satisfying $\dim(\rho'')_{\infty}=1$, then all of the $\dd'$ in the Harder--Narasimhan type of $\rho/\rho''$ have slope greater than $\theta$.
\end{enumerate}
It follows from the first condition that $\rho'$ must have slope less than or equal to $\theta$.  It follows from the indivisibility of the dimension vector $(1,\dd)$ that a $(1,\dd)$-dimensional $\mathbb{C}Q_\ff$-module is $\zeta^{(\theta)}$-stable if it is $\zeta^{(\theta)}$-semistable.  The $G_{\dd}$-action on $X(Q_{\ff})^{\zeta^{(\theta)}\sst}_{(1,\dd)}$ is scheme-theoretically free, by the standard argument recalled in \cite[Prop.3.7]{Efi11}.  Let $\overline{N}\in \mathbb{Z}^n$ be the constant dimension vector $(N,\ldots,N)$.  There are natural open inclusions 
\[
X(Q)_{[0,\theta],\dd}^{\zeta}\times V^{\surj}_{\ff,\dd}\subset X(Q_\ff)_{(1,\dd)}^{\zeta^{(\theta)}\sst}\subset X(Q)^{\zeta}_{[0, \theta],\dd}\times V_{\ff,\dd}
\]
of $G_{\dd}$-equivariant varieties, and so it follows from (\ref{codim}) that the spaces $X(Q_{\overline{N}})_{(1,\dd)}^{\zeta^{(\theta)}\sst}$ provide a system of spaces $U_N$ fulfilling the requirements listed at the start of this section for calculating direct images of vanishing cycles on $\Mst(Q)_{[0,\theta],\dd}^{\zeta}$.  We define 
\begin{equation}
\label{sfrIntro}
\Msp(Q)^{\zeta,\theta\sfr}_{\ff,\dd}:=X(Q_\ff)^{\zeta^{(\theta)}\sst}_{(1,\dd)}/G_{\dd},
\end{equation}
the scheme-theoretic quotient.  We denote by
\[
\pi^{\zeta,\theta\sfr}_{\ff,\dd}\colon \Msp(Q)^{\zeta,\theta\sfr}_{\ff, \dd}\rightarrow \Msp(Q)_{\dd}
\]
the natural projection, taking a stable framed $\mathbb{C}Q_{\ff}$-module to the semisimplification of its underlying $\mathbb{C}Q$-module.  

By \cite[Thm.1]{LBP90}, for an arbitrary finite quiver $Q'$, and dimension vector $\dd'\in\mathbb{N}^{Q'_0}$, the $G_{\dd'}$-invariant functions on $X(Q')_{\dd'}$, are generated by the functions $\rho\mapsto\Tr(\rho(c))$, for $c\in\mathbb{C}Q'/[\mathbb{C}Q',\mathbb{C}Q']$ a cycle.  As $Q_{\ff}$ contains no cycles not already contained in $Q$, we deduce that $\pi^{\zeta,\theta\sfr}_{\ff,\dd}$ is projective, as it is the composition of the GIT quotient map 
\[
X(Q_{\ff})_{(1,\dd)}^{\zeta^{(\theta)}\sst}/G_\dd\rightarrow \Msp(Q_{\ff})_{(1,\dd)}
\]
and the forgetful isomorphism $\Msp(Q_{\ff})_{(1,\dd)}\rightarrow \Msp(Q)_{\dd}$ between the affinizations.  In conclusion, we may make identifications
\begin{align*}
&\Ho\left(p_{[0,\theta],\dd,*}^{\zeta}\phim{\WWW^{\zeta}_{[0,\theta],\dd}}\ICS_{\Mst(Q)_{[0,\theta],\dd}^{\zeta}}(\mathbb{Q})\right)=\\&\lim_{\ff\mapsto\infty}\left(\Ho\left(\pi_{\ff,\dd,*}^{\zeta,\theta\sfr}\phim{\WW^{\zeta,\theta\sfr}_{\ff,\dd}}\ICS_{\Msp(Q)^{\zeta,\theta\sfr}_{\ff,\dd}}(\mathbb{Q})\right)\otimes \LL^{\ff\cdot\dd/2}\right)
\end{align*}
and
\begin{align*}
&\Ho\left(p_{[0,\theta],\dd,!}^{\zeta}\phim{\WWW^{\zeta}_{[0,\theta],\dd}}\ICS_{\Mst(Q)_{[0,\theta],\dd}^{\zeta}}(\mathbb{Q})\right)=\\&\lim_{\ff\mapsto\infty}\left(\Ho\left(\pi_{\ff,\dd,!}^{\zeta,\theta\sfr}\phim{\WW^{\zeta,\theta\sfr}_{\ff,\dd}}\ICS_{\Msp(Q)^{\zeta,\theta\sfr}_{\ff,\dd}}(\mathbb{Q})\right)\otimes \LL^{{}-\ff\cdot\dd/2}\right),
\end{align*}
and we have proved the following proposition.
\begin{proposition}
\label{thapm}
For an arbitrary stability condition $\zeta\in \mathbb{H}_+^n$ slope $\theta\in[0,\pi)$ and dimension vector $\dd\in\mathbb{N}^m$, the map $p_{[0,\theta],\dd}^{\zeta}\colon \Mst(Q)_{[0,\theta],\dd}^{\zeta}\rightarrow \Msp(Q)_{\dd}$ is approximated by proper maps in the sense of Definition \ref{properMapsDef}.
\end{proposition}
\begin{corollary}\label{easyAPM}
The map (\ref{pdef}):
\begin{equation}
p^{\zeta\sst}_{\dd}\colon \Mst(Q)^{\zeta\sst}_\dd\rightarrow\Msp(Q)^{\zeta\sst}_\dd
\end{equation}
is approximated by proper maps.
\end{corollary}
\begin{proof}
Setting $\theta=\Mu^{\zeta}(\dd)$, there is a commutative diagram
\[
\xymatrix{
\Mst(Q)^{\zeta\sst}_\dd\ar[r]^{p^{\zeta\sst}_{\dd}}\ar[d]^=&\Msp(Q)^{\zeta\sst}_\dd\\
\Mst(Q)_{[0,\theta],\dd}^{\zeta}\ar[ur]_{p_{[0,\theta],\dd}^{\zeta}},
}
\]
where the vertical equality follows from the fact that if a $\dd$-dimensional $\mathbb{C}Q$-module $M$ is not semistable, it has a Harder--Narasimhan filtration
\[
0=M_0\subset\ldots \subset M_h=M
\]
with $h\geq 2$, and $\Mu^{\zeta}(\dim(M_h/M_{h-1}))>\Mu^{\zeta}(\dd)$.  The diagonal morphism in the commutative diagram is approximated by proper maps by Proposition \ref{thapm}, and the corollary follows.
\end{proof}

By \cite{Sai90}, (monodromic) vanishing cycle functors commute with taking the direct image along proper maps, and are exact by Proposition \ref{exactness}, and so we deduce from Proposition \ref{thapm} that
\begin{align}
\label{bettStarForm}
&\Ho\left(p_{[0,\theta],\dd,*}^{\zeta}\left(\phim{\WWW^{\zeta}_{[0,\theta],\dd}}\ICS_{\Mst(Q)_{[0, \theta],\dd}^{\zeta}}(\mathbb{Q})\right)\right)\cong\\&\phim{\WW_{\dd}}\lim_{\ff\mapsto\infty}\left(\Ho\left(\pi_{\ff,\dd,*}^{\zeta,\theta\sfr}\ICS_{\Msp(Q)^{\zeta,\theta\sfr}_{\ff,\dd}}(\mathbb{Q})\right)\otimes \LL^{\ff\cdot\dd/2}\right)\nonumber\cong\\&
\phim{\WW_{\dd}}\Ho\left(p_{[0,\theta],\dd,*}^{\zeta}\left(\ICS_{\Mst(Q)_{[0, \theta],\dd}^{\zeta}}(\mathbb{Q})\right)\right)\nonumber
\end{align}
and
\begin{align}
&\Ho\left(p_{[0, \theta],\dd,!}^{\zeta}\left(\phim{\WWW^{\zeta}_{[0,\theta],\dd}}\ICS_{\Mst(Q)_{[0, \theta],\dd}^{\zeta}}(\mathbb{Q})\right)\right)\cong\\&\phim{\WW_{\dd}}\lim_{\ff\mapsto\infty}\left(\Ho\left(\pi_{\ff,\dd,!}^{\zeta,\theta\sfr}\ICS_{\Msp(Q)^{\zeta,\theta\sfr}_{\ff,\dd}}(\mathbb{Q})\right)\otimes \LL^{-\ff\cdot\dd/2}\right)\nonumber\cong\\&
\phim{\WW_{\dd}}\Ho\left(p_{[0,\theta],\dd,!}^{\zeta}\left(\ICS_{\Mst(Q)_{[0, \theta],\dd}^{\zeta}}(\mathbb{Q})\right)\right)\nonumber.
\end{align}
We now have the notation at hand to give the promised proof of the following lemma.
\begin{lemma}\cite[Lem.2.1]{MeRe14}\label{finLem}
The map $\Msp(Q)\times \Msp(Q)\xrightarrow{\oplus} \Msp(Q)$ is finite.
\end{lemma}
\begin{proof}
Quasi-finiteness is clear, since up to isomorphism, a direct sum of simple $\mathbb{C}Q$-modules can be written as a direct sum of direct sums of simple $\mathbb{C}Q$-modules in only finitely many ways.  So all that remains is to prove properness.  Fix two dimension vectors $\dd',\dd''\in\mathbb{N}^{m}$.  Setting $\zeta=(i,\ldots,i)$ the moduli space 
\[
\Msp(Q)_{\ff,\dd}^{\sfr}:=\Msp(Q)_{\ff,\dd}^{\zeta, \pi/2\sfr}
\]
is just the usual noncommutative Hilbert scheme, i.e. the moduli space of framed $\dd$-dimensional $\mathbb{C}Q$-modules for which the framing vector generates the $\mathbb{C}Q$-module, as introduced in \cite{Re05}.  Then for sufficiently large $\ff',\ff''$ (for instance if $\ff'\geq \dd'$ and $\ff''\geq \dd''$ for the natural partial ordering), the moduli spaces $\Msp(Q)^{\sfr}_{\ff',\dd'}$ and $\Msp(Q)^{\sfr}_{\ff'',\dd''}$ are nonempty, and the maps $\Msp(Q)^{\sfr}_{\ff',\dd'}\rightarrow \Msp(Q)_{\dd'}$ and $\Msp(Q)^{\sfr}_{\ff'',\dd''}\rightarrow \Msp(Q)_{\dd''}$ are surjective on geometric points.  

We claim that the natural map $\Msp(Q)^{\sfr}_{\ff',\dd'}\times \Msp(Q)^{\sfr}_{\ff'',\dd''}\rightarrow \Msp(Q)^{\sfr}_{\ff,\dd}$ is a closed embedding, where we have set $\dd=\dd'+\dd''$ and $\ff=\ff'+\ff''$.  We define $\nu'\colon\Msp(Q)^{\sfr}_{\ff,\dd}\rightarrow \mathbb{Z}$ to be the lower semicontinuous function taking a framed representation to the dimension of the $\mathbb{C}Q$-module generated by the first $\ff'$ framing vectors, and define $\nu''$ similarly, by considering the last $\ff''$ framing vectors.  Then $\nu:=\nu'+\nu''$ is lower semicontinuous, and the desired inclusion is a component of the inclusion of the set $\nu^{-1}(\tau)$, where $\tau=\sum_{s\in Q_0}\dd_s$ is the minimal value of $\nu$, thus proving the claim.  The map $\pi_{\ff,\dd}\colon\Msp(Q)^{\sfr}_{\ff,\dd}\rightarrow \Msp(Q)_{\dd}$ is proper, as noted above.  So in the diagram
\[
\xymatrix{
\Msp(Q)^{\sfr}_{\ff',\dd'}\times\Msp(Q)^{\sfr}_{\ff'',\dd''}\ar[dr]\ar[d]_-{\pi_{\ff',\dd'}\times\pi_{\ff'',\dd''}} \\\Msp(Q)_{\dd'}\times\Msp(Q)_{\dd''}\ar[r]_-{\oplus} &\Msp(Q)_{\dd}
}
\]
the diagonal map is proper, while the vertical map is surjective on points, and $\oplus$ is separated since $\Msp(Q)_{\dd'}$ and $\Msp(Q)_{\dd''}$ are quasiprojective.  Properness of the horizontal map then follows via the valuative criterion of properness.
\end{proof}
\subsection{Cohomological wall crossing}
The following theorem is a categorification of the identity in the quantum torus $\hat{\Ror}_Q$ induced by the existence and uniqueness of Harder--Narasimhan filtrations.  This identity is known as the wall crossing formula in the work of Kontsevich, Soibelman, Joyce and Song.  It is this theorem that allows us to categorify the strategy for understanding quantum cluster mutation that starts with Nagao's quote from the introduction.
\begin{theorem}
\label{rays}
There is an isomorphism in $\Dulf(\MMHM(\Msp(Q)))$
\begin{align*}
&\left(\Ho\left(p_{[0,\theta],*}^{\zeta}\left(\phim{\WWW^{\zeta}_{[0,\theta]}}\ICS_{\Mst(Q)_{[0,\theta]}^{\zeta}}(\mathbb{Q})\right)\right)\right)_{\nilp}\cong\\&\Boxtimes_{\oplus,[\theta\xrightarrow{\gamma}0]}^{\tw}\left(\Ho\left(q^{\zeta}_{\gamma,*}p^{\zeta\sst}_{\gamma,*}\left(\phim{\WWW^{\zeta\sst}_{\gamma}}\ICS_{\Mst(Q)^{\zeta\sst}_{\gamma}}(\mathbb{Q})\right)\right)\right)_{\nilp},
\end{align*}
where the monoidal product is taken over descending slopes.
\end{theorem}
\begin{proof}
Let $\HN_{[0,\theta],\dd}$ be the set of all Harder--Narasimhan types $\overline{\dd}=(\dd^1,\ldots,\dd^t)\in\HN_{\dd}$ such that all the slopes $\Mu^{\zeta}(\dd^g)$ belong to $[0,\theta]$.  By \cite[Prop.3.7]{Reineke_HN} there is a partial order $\leq'$ on the set $\HN_{[0,\theta],\dd}$ such that for all $\overline{\ee}\in\HN_{[0,\theta],\dd}$, the closure of $X(Q)^{\zeta}_{\overline{\ee}}$ is contained in $\bigcup_{\overline{\ee}'\leq' \overline{\ee}}X(Q)^{\zeta}_{\overline{\ee}'}$.  We complete $\leq'$ to a total order $\leq$ on $\HN_{[0,\theta],\dd}$.  For $\overline{\ee}\in\HN_{[0,\theta],\dd}$ we define
\begin{align*}
X(Q)^{\zeta}_{\leq \overline{\ee}}:=&\bigcup_{\overline{\ee}'\leq \overline{\mathbf{e}}}X(Q)^{\zeta}_{\overline{\ee}'},\\
X(Q)^{\zeta}_{<\overline{\ee}}:=&\bigcup_{\overline{\ee}'< \overline{\mathbf{e}}}X(Q)^{\zeta}_{\overline{\ee}'}.
\end{align*}
We denote by 
\begin{align*}
i_{\overline{\ee}}\colon&\Mst(Q)_{\overline{\ee}}^{\zeta}\hookrightarrow\Mst(Q)_{\dd} \\
i_{<\overline{\ee}}\colon&\Mst(Q)_{<\overline{\ee}}^{\zeta}\hookrightarrow\Mst(Q)_{\dd}\\
i_{\leq\overline{\ee}}\colon &\Mst(Q)_{\leq \overline{\ee}}^{\zeta}\hookrightarrow\Mst(Q)_{\dd}
\end{align*}
the obvious inclusions.  Let $\mm\in\HN_{[0,\theta],\dd}$ be the maximum element with respect to $\leq$.  The variety $X(Q)_{[0,\theta],\dd}^{\zeta}$ is open in $X(Q)_{\dd}$ (see e.g. the proof of \cite[Prop.4.20]{DMSS13}), and so these schemes have the same dimension, and there is an isomorphism
\[
\Ho\left(p^\zeta_{[0,\theta],\dd,!}i_{\leq\mm}^*\ICS_{\Mst(Q)_{\dd}}(\mathbb{Q})\right)\cong \Ho\left(p^\zeta_{[0,\theta],\dd,!}\ICS_{\Mst(Q)_{[0,\theta],\dd}^{\zeta}}(\mathbb{Q})\right).
\]
For all $\ee\in\HN_{[0,\theta],\dd}$ the inclusion $X_{\overline{\mathbf{e}}}\subset X_{\leq \overline{\mathbf{e}}}$ is open, with complement $X_{<\overline{\mathbf{e}}}$.  As such, there is a distinguished triangle
\begin{align}
\label{trian}
\Ho\left( p_{\dd,!}i_{\overline{\ee},!}i_{\overline{\ee}}^*\ICS_{\Mst(Q)_{\dd}}(\mathbb{Q})\right)\rightarrow \Ho\left( p_{\dd,!}i_{\leq \overline{\ee},!}i_{\leq \overline{\ee}}^*\ICS_{\Mst(Q)_{\dd}}(\mathbb{Q})\right)\rightarrow \Ho\left( p_{\dd,!}i_{<\overline{\ee},!}i_{<\overline{\ee}}^*\ICS_{\Mst(Q)_{\dd}}(\mathbb{Q})\right)\rightarrow.
\end{align}
On the other hand, for $\overline{\ee}=(\dd^1,\ldots,\dd^t)$, there is an isomorphism, via the usual diagram of affine fibrations and the biexactness of $\boxtimes_{\oplus}$ (see e.g. the proof of the cohomological wall crossing isomorphism in \cite[Sec.4.2]{DaMe15b})
\begin{equation}
\label{purebox}
\Ho\left( p_{\dd,!}i_{\overline{\ee},!}i_{\overline{\ee}}^*\ICS_{\Mst(Q)_\dd}(\mathbb{Q})\right)\cong \Ho\left(q^{\zeta}_{\dd^t,!}p^{\zeta\sst}_{\dd^t,!}\ICS_{\Mst(Q)^{\zeta\sst}_{\dd^t}}(\mathbb{Q})\right)\boxtimes_{\oplus}^{\tw}\ldots\boxtimes_{\oplus}^{\tw}\Ho\left(q^{\zeta}_{\dd^1,!}p^{\zeta\sst}_{\dd^1,!}\ICS_{\Mst(Q)^{\zeta\sst}_{\dd^1}}(\mathbb{Q})\right).
\end{equation}
Each of the terms on the right hand side of (\ref{purebox}) is pure since for each $g\leq t$ the map $q^{\zeta}_{\dd^g,!}p^{\zeta\sst}_{\dd^g,!}$ is approximated by proper maps by Corollary \ref{easyAPM} and properness of $q^{\zeta}_{\dd^g,!}$ (it is a GIT quotient map).  It follows that the right hand side of (\ref{purebox}) is pure by properness of the direct sum map (Lemma \ref{finLem}) and Proposition \ref{weightPreserve}, and so the left hand side of (\ref{purebox}) is pure, and we deduce that the first term of (\ref{trian}) is pure.  

It follows by induction with respect to the order $\leq$ that all of the terms in (\ref{trian}) are pure, and the associated long exact sequence breaks up into short exact sequences (one for each cohomological degree), which are moreover split by semisimplicity of the category of pure mixed Hodge modules on $\Msp(Q)_{\dd}$ \cite[(4.5.3)]{Sai90}.  It follows that there is an isomorphism
\[
\Ho\left(p_{[0,\theta],\dd,!}^\zeta\ICS_{\Mst(Q)^{\zeta}_{[0,\theta],\dd}}(\mathbb{Q})\right)\cong\bigoplus_{\substack{(\dd^1,\ldots,\dd^t)\in (\Mu^{\zeta})^{-1}([0,\theta])\\ \Mu^{\zeta}(\dd^1)<\ldots<\Mu^{\zeta}(\dd^t)\\ \sum_{i=1}^t \dd^i=\dd}}\Boxtimes_{\oplus,i=1,\ldots,t}^{\tw}\Ho\left(q^{\zeta}_{\dd^i,!}p^{\zeta\sst}_{\dd^i,!}\ICS_{\Mst(Q)^{\zeta\sst}_{\dd^i}}(\mathbb{Q})\right).
\]
Summing over $\dd$ and rearranging we obtain the isomorphism
\begin{equation}
\label{unswapped}
\Ho\left(p_{[0,\theta],!}^{\zeta}\ICS_{\Mst(Q)_{[0,\theta]}^{\zeta}}(\QQ)\right)\cong\Boxtimes_{\oplus,[0\xrightarrow{\gamma}\theta]}^{\tw}\Ho\left(q^{\zeta}_{\gamma,!}p^{\zeta\sst}_{\gamma,!}\ICS_{\Mst(Q)^{\zeta\sst}_{\gamma}}(\mathbb{Q})\right).
\end{equation}
Taking Verdier duals, we obtain the isomorphism
\begin{equation}
\label{swapped}
\Ho\left(p_{[0,\theta],*}^{\zeta}\ICS_{\Mst(Q)_{[0,\theta]}^{\zeta}}(\QQ)\right)\cong\Boxtimes_{\oplus,[\theta\xrightarrow{\gamma}0]}^{\tw}\Ho\left(q_{\gamma,*}^{\zeta}p^{\zeta\sst}_{\gamma,*}\ICS_{\Mst(Q)^{\zeta\sst}_{\gamma}}(\mathbb{Q})\right).
\end{equation}
The reversal in the order of the product between (\ref{unswapped}) and (\ref{swapped}) is as in Remark \ref{vswap}.  Applying the functor $\phim{\WW}$ to (\ref{swapped}), we obtain
\[
\phim{\WW}\Ho\left(p_{[0,\theta],*}^{\zeta}\ICS_{\Mst(Q)_{[0,\theta]}^{\zeta}}(\QQ)\right)\cong\phim{\WW}\Boxtimes_{\oplus,[\theta\xrightarrow{\gamma}0]}^{\tw}\Ho\left(q^{\zeta}_{\gamma,*}p^{\zeta\sst}_{\gamma,*}\ICS_{\Mst(Q)^{\zeta\sst}_{\gamma}}(\mathbb{Q})\right).
\]
Since $\WWW$ is identically zero on the nilpotent locus, we may apply the Thom--Sebastiani isomorphism (\ref{TSiso}), after restriction to the nilpotent locus, to obtain the isomorphism
\[
\Ho\left(\phim{\WW}p_{[0,\theta],*}^{\zeta}\ICS_{\Mst(Q)_{[0,\theta]}^{\zeta}}(\QQ)\right)_{\nilp}\cong\Boxtimes_{\oplus,[\theta\xrightarrow{\gamma}0]}^{\tw}\Ho\left(\phim{\WW_{\gamma}}q^{\zeta}_{\gamma,*}p^{\zeta\sst}_{\gamma,*}\ICS_{\Mst(Q)^{\zeta\sst}_{\gamma}}(\mathbb{Q})\right)_{\nilp},
\]
where we have also used exactness of $\phim{\WW_{\gamma}}$ (Proposition \ref{exactness}) to commute the vanishing cyce functor past the total cohomology functor.  Since all relevant maps are either proper or can be approximated by proper maps, and so commute with vanishing cycle functors, we deduce that there is an isomorphism
\begin{align*}
&\left(\Ho\left(p_{[0,\theta],*}^{\zeta}\phim{\WWW_{[0,\theta]}^{\zeta}}\ICS_{\Mst(Q)_{[0,\theta]}^{\zeta}}(\QQ)\right)\right)_{\nilp}\cong\\&\Boxtimes_{\oplus,[\theta\xrightarrow{\gamma}0]}^{\tw}\left(\Ho\left(q^{\zeta}_{\gamma,*}p^{\zeta\sst}_{\gamma,*}\phim{\WWW^{\zeta\sst}_{\gamma}}\ICS_{\Mst(Q)^{\zeta\sst}_{\gamma}}(\mathbb{Q})\right)\right)_{\nilp},
\end{align*}
as required.
\end{proof}

\begin{definition}\label{plethDef}
In the following example, we use the plethystic exponential $\EXP$, defined as follows.  If $p(x_1,\ldots,x_s,q^{1/2})\in\mathbb{Z}((q^{1/2}))[[x_1,\ldots,x_s]]$ is a formal power series in commuting variables, written as
\[
p(x_1,\ldots,x_s,q^{1/2})=\sum_{g_1,\ldots,g_s\in\mathbb{N},h\in \mathbb{Z}}a_{g_1,\ldots,g_s,h}x_1^{g_1}\ldots x_s^{g_s}(-q^{1/2})^h
\]
with $a_{g_1,\ldots,g_s,h}=0$ if $g_1=\ldots = g_s=0$, then we define
\[
\EXP\left(p(x_1,\ldots,x_s,q^{1/2})\right)=\prod_{g_1,\ldots,g_s\in\mathbb{N},h\in \mathbb{Z}}(1-x_1^{g_1}\ldots x_s^{g_s}(-q^{1/2})^h)^{-a_{g_1,\ldots,g_s,h}}.
\]
\end{definition}
The signs appearing next to half powers of $q$ in the above definition are there in order to accord with the signs in the definition of $\chi_Q$ (see (\ref{dulfjust})).
\begin{example}
\label{basicWCF}
Let $Q$ be the ice quiver with vertices $\{1,2\}$, with one arrow going from $2$ to $1$, and no other arrows, and for which $Q_{\princ}=Q$.  Then 
\[
\hat{\Ror}_Q= \mathbb{Z}((q^{1/2}))[[Y^{(1,0)},Y^{(0,1)}]],
\]
with the commutation relation $Y^{(1,0)}Y^{(0,1)}=qY^{(0,1)}Y^{(1,0)}$.  Consider a stability condition $\zeta$, for which $\Mu^{\zeta}((1,0))=\pi/4$ and $\Mu^{\zeta}((0,1))=\pi/2$.  We put $W=0$ (there are no nonzero potentials for this quiver, as it is acyclic).  Then the only $\zeta$-semistable $\mathbb{C}Q$-modules are direct sums of the simple representations $\Simp(Q)_s$, for $s\in \{1,2\}$.  Putting $\theta=3\pi/4$ in Theorem \ref{rays}, and applying $\chi_Q$ to the resulting equality in $\KK(\Dulf(\MMHM(\mathbb{N}^2)))$, we obtain the identity
\begin{equation}
\label{oneWay}
\chi_Q\left([\Dim_*(\ICS_{\Mst(Q)}(\mathbb{Q}))]\right)=\EXP\left(\frac{Y^{(0,1)}}{q^{-1/2}-q^{1/2}}\right)\EXP\left(\frac{Y^{(1,0)}}{q^{-1/2}-q^{1/2}}\right),
\end{equation}
where we have used the well-known quantum dilogarithm identity
\begin{equation}
\label{dilogId}
\sum_{g\in\mathbb{N}}\chi_q\left(\HO(\pt/\Gl_g,\QQ)\otimes \LL^{g^2/2},-q^{1/2}\right)x^g=\EXP\left(\frac{x}{q^{-1/2}-q^{1/2}}\right),
\end{equation}
considering $x/(q^{-1/2}-q^{1/2})$ as a formal power series via the expansion
\[
\frac{x}{q^{-1/2}-q^{1/2}}=x(q^{1/2}+q^{3/2}+\ldots).
\]
On the other hand, picking $\zeta'$ so that $\Mu^{\zeta'}((1,0))=\pi/2$ and $\Mu^{\zeta'}((0,1))=\pi/4$, we obtain more semistable modules, namely, direct sums of the Jacobi algebra $\mathbb{C}Q$, considered as a module over itself.  We thus obtain the identity
\begin{equation}
\label{secondWay}
\chi_q\left([\Dim_*(\ICS_{\Mst(Q)}(\QQ))]\right)=\EXP\left(\frac{Y^{(1,0)}}{q^{-1/2}-q^{1/2}}\right)\EXP\left(\frac{Y^{(1,1)}}{q^{-1/2}-q^{1/2}}\right)\EXP\left(\frac{Y^{(0,1)}}{q^{-1/2}-q^{1/2}}\right).
\end{equation}
Multiplying (\ref{oneWay}) and (\ref{secondWay}) on the right by $\EXP\left(\frac{{}-Y^{(0,1)}}{q^{-1/2}-q^{1/2}}\right)$, considering the resulting identity as an identity between formal power series in $Y^{(1,0)}$, and considering just the linear coefficient, we obtain the identity
\begin{equation}
\label{EfId}
\EXP\left(\frac{Y^{(0,1)}}{q^{-1/2}-q^{1/2}}\right)Y^{(1,0)}\EXP\left(\frac{Y^{(0,1)}}{q^{-1/2}-q^{1/2}}\right)^{-1}=Y^{(1,0)}+Y^{(1,1)}=Y^{(1,0)}(1+q^{-1/2}Y^{(0,1)}).
\end{equation}
This is a purely algebraic identity, holding in any sufficiently complete $\mathbb{Z}((q^{1/2}))$-algebra such that $Y^{(1,0)}Y^{(0,1)}=qY^{(0,1)}Y^{(1,0)}$.  

Alternatively, multiplying (\ref{oneWay}) and (\ref{secondWay}) on the left by $\pleth{\frac{-Y^{(1,0)}}{q^{-1/2}-q^{1/2}}}$, we obtain the identity
\begin{equation}
\label{reverseWall}
\pleth{\frac{Y^{(1,0)}}{q^{-1/2}-q^{1/2}}}^{-1}Y^{(0,1)}\pleth{\frac{Y^{(1,0)}}{q^{-1/2}-q^{1/2}}}=Y^{(0,1)}+Y^{(1,1)}.
\end{equation}

The identity between the right hand sides of (\ref{oneWay}) and (\ref{secondWay}) is the simplest of the ``dilogarithm identities'' associated to Dynkin quivers --- see \cite{Keller10} for more details of the relation between these identities and quantum cluster algebras.  
\end{example}
\begin{remark}
Note that written in terms of the $\mathbb{Z}((q^{1/2}))$-basis provided by the monomials $Y^{\dd}$, the right hand sides of equations (\ref{EfId}) and (\ref{reverseWall}) do not involve any powers of $q^{1/2}$, and are in particular invariant under the substitution $q^{1/2}\mapsto q^{-1/2}$.  In the context of quantum cluster algebras, this well-known phenomenon (see \cite[Prop.6.2]{BZ05}) also follows from the Lefschetz property for quantum cluster transformations, part of our main theorem.
\end{remark}

\section{Cluster mutations from derived equivalences}
\subsection{Categorification of cluster mutation}
Let $(Q,W)$ be a quiver with formal potential.  We recall Ginzburg's construction of $\HGamma(Q,W)$ from \cite{ginz}.  Firstly, we form a graded quiver $\tilde{Q}$ from $Q$ by setting 
\begin{align*}
\tilde{Q}_0=&Q_0\\
(\tilde{Q}_1)_0=&Q_1\\
(\tilde{Q}_1)_{{}-1}=&\{a^*|a\in Q_1\}\\
(\tilde{Q}_1)_{-2}=&\{\omega_i|i\in Q_0\}.
\end{align*}
The numbers appearing outside of the brackets in the above expressions specify the degrees of the arrows.  For $a\in Q_1$ we set $s(a^*)=t(a)$ and $t(a^*)=s(a)$.  So in degree -1, $\tilde{Q}$ is the opposite quiver of $Q$.  For $i\in Q_0$ we set $s(\omega_{i})=t(\omega_i)=i$.  We let $\HGamma(Q,0)$ be the free path algebra of $\tilde{Q}$, completed at the two-sided ideal generated by the degree zero arrows.  We define a differential on the generators by setting
\begin{align*}
da=&0&\textrm{for }a\in Q_1,\\
da^*=&\partial W/\partial a, \\
d\omega_i=&e_i\sum_{a\in Q_1}[a,a^*]e_i&\textrm{for }i\in Q_0
\end{align*}
and extend $d$ to a differential on $\HGamma(Q,0)$ by the Leibniz rule, linearity and continuity.  The resulting differential graded algebra is denoted $\HGamma(Q,W)$.  Then $\HJac(Q,W)\cong\HO^0(\HGamma(Q,W))$, and there is an embedding of categories 
\[
\Rmod{\HJac(Q,W)}\subset \Dub(\Rmod{\HGamma(Q,W)})
\]
as the heart of the natural t structure.  Accordingly, we may consider $\Simp(Q)_{s}$ as the simple $\HGamma(Q,W)$-module concentrated in degree zero, with dimension vector $1_s$, and for which all of the arrows act via the zero map.  The algebra $\HGamma(Q,W)$ is 3--Calabi--Yau in the sense defined and proved in \cite{Kel09}, so that there is a bifunctorial isomorphism $\Ext^i(M,N)\cong \Ext^{3-i}(N,M)^*$ on $\Dub^{\fd}(\Rmod{\HGamma(Q,W)})$ \cite[Lem.4.1]{CYTC}.

For a dimension vector $\ff\in\mathbb{N}^{n}$ we define
\begin{align*}
\HJac(Q,W)_{\ff}:=&\bigoplus_{s\leq n}\left(e_s\cdot \HJac(Q,W)\right)^{\oplus\ff_s},
\end{align*}
and define $\HGamma(Q,W)_{\ff}$ and $\Jac(Q,W)_{\ff}$ similarly.  For $s\in Q_0$, tensoring the bimodule resolution of $\HJac(Q,W)$ (see \cite[Prop.5.1.9]{ginz}) with $\Simp(Q)_s$, there is a canonical resolution
\begin{align}
\label{resolution}
&0\rightarrow e_s\cdot \HGamma(Q,W)\rightarrow \bigoplus_{a\in Q_1|s(a)=s} e_{t(a)}\cdot\HGamma(Q,W)\rightarrow\\ \rightarrow &\bigoplus_{a\in Q_1|t(a)=s}e_{s(a)}\cdot\HGamma(Q,W) \rightarrow e_s\cdot\HGamma(Q,W)\rightarrow \Simp(Q)_{s}\rightarrow 0.\nonumber
\end{align}
The map 
\[
\DIM\colon\KK\left(\Dub^{\fd}(\Rmod{\HGamma(Q,W)})\right)\rightarrow \mathbb{Z}^{n}
\]
is an isomorphism, which extends in the obvious way to morphisms $\KK(\rmod{\Jac(Q,W)})\rightarrow \mathbb{Z}^{n}$ in the case of algebraic $W$, and $\KK(\rmod{\mathbb{C}Q})\rightarrow\mathbb{Z}^{n}$, which are isomorphisms if $\mathbb{C}Q$ is acyclic.  We denote by $\Dub_{\princ}^{\fd}(\Rmod{\HGamma(Q,W)})$ the full subcategory of $\Dub(\Rmod{\HGamma(Q,W)})$ whose objects have finite-dimensional total cohomology, supported on the principal part of $Q$.  We identify $\mathbb{Z}^{m}$ with $\KK\left(\Dub_{\princ}^{\fd}(\Rmod{\HGamma(Q,W)})\right)$ via the map sending $1_s\mapsto [\Simp(Q)_{s}]$.  

Recall that we define $\Perf(\HGamma(Q,W))\subset \Dub(\Rmod{\HGamma(Q,W)})$ to be the smallest strictly full subcategory, closed under shifts, cones and direct summands, containing the modules $e_s\cdot \HGamma(Q,W)$, for $s\in Q_0$.  We identify $\mathbb{Z}^{n}$ with $\KK(\Perf(\HGamma(Q,W)))$ via the map sending $1_s\mapsto [e_s\cdot \HGamma(Q,W)]$.  This map is indeed an isomorphism --- this appears to be a standard fact, which one may prove\footnote{Thanks to Bernhard Keller for pointing out this argument.} by combining \cite[Lem.5.2.1]{Bon10} and \cite[Lem.2.17]{KellerMutations}, or using the fact that under the Koszul duality functor $\Hom(\bigoplus_{i\in Q_0}\Simp(Q)_i,-)$ the set of perfect modules $e_i\cdot\HGamma(Q,W)$ is sent to the set of simple modules $\Simp_i$ for the Koszul dual quiver algebra \cite[Thm.5.4]{Koszul}, and the category of iterated shifts and extensions of this set of objects is the category of differential graded modules with nilpotent finite-dimensional total cohomology, which is closed under taking direct summands.  

The following proposition follows from compatibility of $\tilde{B}$ and $\Lambda$, and the existence of the resolution (\ref{resolution}).

\begin{proposition}\label{lattComm}
There is an inclusion of triangulated categories 
\[
\nu\colon\Dub_{\princ}^{\fd}(\Rmod{\HGamma(Q,W)})\hookrightarrow \Perf(\HGamma(Q,W)),
\]
and the diagram
\[
\xymatrix{
\KK\left(\Dub_{\princ}^{\fd}(\Rmod{\HGamma(Q,W)})\right)\ar[d]\ar[rr]^-{\KK(\nu)}&&\KK\left(\Perf(\HGamma(Q,W))\right)\ar[d]\\
\mathbb{Z}^{m}\ar[rr]^{ \tilde{B}\cdot}&&\mathbb{Z}^{n}
}
\]
commutes.  Furthermore, giving $\mathbb{Z}^n$ the bilinear form induced by $\Lambda$, and $\mathbb{Z}^m$ the bilinear form ${}-\langle\bullet ,\bullet\rangle_Q$, and $\KK(\Dub_{\princ}^{\fd}(\Rmod{\HGamma(Q,W)}))$ the bilinear form
\[
\langle [N],[N']\rangle=\sum_{g\in\ZZ}(-1)^g\dim\left(\Ext^g(N,N')\right),
\]
all bilinear forms are preserved in the above diagram.
\end{proposition}

We set $\iota= {}-\KK(\nu)$, a homomorphism of lattices with inner product, and also denote by $\iota$ the induced inclusions
\begin{align*}
\iota\colon&\Ror_Q\rightarrow\Tor_{\Lambda}\\
\iota\colon&\hat{\Ror}_Q\rightarrow\hat{\Tor}_{\Lambda}.
\end{align*}
By Proposition \ref{lattComm}, this definition of the map $\iota$ agrees with our previous definition (\ref{iotadef}) of $\iota$ as the map sending $Y^{\dd}\mapsto X^{-\tilde{B}\cdot \dd}$.
\begin{remark}
Via the identification $\mathbb{Z}^m\cong \KK\left(\Dub_{\princ}^{\fd}(\Rmod{\HGamma(Q,W)})\right)$, there are now three skew-symmetric bilinear forms on $\mathbb{Z}^m$, which we denote $\langle\bullet ,\bullet\rangle_Q$, $\tilde{B}(\bullet,\bullet)$, and $\langle \bullet,\bullet\rangle_{\chi}$, where the second is defined via the principal part of $\tilde{B}$, and the third is the Euler pairing on $\KK\left(\Dub_{\princ}^{\fd}(\Rmod{\HGamma(Q,W)})\right)$.  Since this is potentially quite confusing, we collect together the relations between these three pairings:
\begin{equation}
\label{differentStrokes}
\langle \bullet,\bullet\rangle_{\chi}=\tilde{B}(\bullet,\bullet)=-\langle \bullet,\bullet\rangle_{Q}.
\end{equation}
\end{remark}
\begin{remark}
\label{muComm}
Given $M\in\Dub^{\fd}_{\princ}\left(\Rmod{\HGamma(Q,W)}\right)$ and $N\in\Perf\left(\HGamma(Q,W)\right)$, there is an equality
\[
\Lambda\left([N],\iota([M])\right)=\sum_{g\in\mathbb{Z}}(-1)^g\dim\left(\Ext^g(N,M)\right).
\]
To show this, we remark that it is sufficient to prove the statement for the natural generating sets of $\KK\left(\Dub_{\princ}^{\fd}\left(\Rmod{\HGamma(Q,W)}\right)\right)$ and $\KK\left(\Perf\left(\HGamma(Q,W)\right)\right)$.  If we identify $\Lambda$ with the $n\times n$ matrix defining it with respect to the standard basis, then by the compatibility of $\Lambda$ and $\tilde{B}$, we have the equality
\[
-\Lambda \tilde{B}=\tilde{I}^T
\]
where $\tilde{I}$ is the $n\times m$ matrix with $I_{m\times m}$ for the first $m$ columns, and zeroes elsewhere.  By the definition of $\iota$, it then follows that then $\Lambda(1_{s'},\iota(1_{s''}))=\delta_{s',s''}$, where $\delta_{s',s''}$ is the Kronecker delta function.  Then the claim follows from the equation
\begin{align*}
\sum_{g\in\mathbb{Z}}(-1)^g\dim\left(\Ext^g(e_{s'}\cdot \HGamma(Q,W),\Simp(Q)_{s''})\right)=&\dim\left(\Hom(e_{s'}\cdot \HGamma(Q,W),\Simp(Q)_{s''})\right)\\=&\delta_{s',s''}.
\end{align*}
It follows that if 
\[
\Psi\colon \Perf\left(\HGamma(Q,W)\right)\rightarrow\Perf\left(\HGamma(Q,W)\right)
\]
is an autoequivalence of triangulated categories, restricting to an autoequivalence $\Psi\colon\Dub^{\fd}_{\princ}\left(\Rmod{\HGamma(Q,W)}\right)\rightarrow\Dub^{\fd}_{\princ}\left(\Rmod{\HGamma(Q,W)}\right)$, then there is an identity
\[
\Lambda\big(\Psi([N]),\iota\Psi([M])\big)=\Lambda([N],[M])
\]
for all $N,M$ as above, where we have used the same letter $\Psi$ to denote the induced automorphisms of the respective Grothendieck groups.
\end{remark}

\subsection{Nagao's torsion pair}\label{Tpair}

The following is proved in \cite[Thm.3.2, Rem.3.3]{KellerMutations}.
\begin{theorem}
\label{KYThm}
Let $(Q,W)$ be a QP which is nondegenerate with respect to the sequence of principal vertices $\Ss=(s_1,\ldots,s_t)$, and let $\epsilon\in\{\pm \}^{t}$ be a sequence of signs of length $t$.  There is a quasi-equivalence
\[
\Phi_{\Ss,\epsilon}\colon \Dub\left(\Rmod{\HGamma(Q,W)}\right)\rightarrow \Dub\left(\Rmod{\HGamma(\mu_{\Ss}(Q,W))}\right)
\]
restricting to quasi-equivalences
\[
\Phi_{\Ss,\epsilon}\colon\Perf\left(\HGamma(Q,W)\right)\xrightarrow{\sim}\Perf\left(\HGamma(\mu_{\Ss}(Q,W))\right)
\]
and
\begin{equation}
\label{fdRes}
\Phi_{\Ss,\epsilon}\colon\Dub^{\fd}_{\princ}\left(\Rmod{\HGamma(Q,W)}\right)\xrightarrow{\sim}\Dub^{\fd}_{\princ}\left(\Rmod{\HGamma(\mu_{\Ss}(Q,W))}\right).
\end{equation}
If $\epsilon_{t}=+$, there is an exact triangle
\begin{equation}
\label{plusCase}
\Phi^{-1}_{\Ss,\epsilon}\left(e_{s_t}\cdot\HGamma(\mu_{\Ss}(Q,W))\right)\rightarrow\bigoplus_{a\in \mu_{\Ss'}(Q)_1|t(a)=s_t}\Phi^{-1}_{\Ss',\epsilon'}\left(e_{s(a)}\cdot\HGamma(\mu_{\Ss'}(Q,W))\right)\rightarrow \Phi^{-1}_{\Ss',\epsilon'}\left(e_{s_t}\cdot \HGamma(\mu_{\Ss'}(Q,W))\right)\rightarrow
\end{equation}
while if $\epsilon_t=-$ there is an exact triangle
\begin{equation}
\label{minusCase}
\Phi^{-1}_{\Ss',\epsilon'}\left(e_{s_t}\cdot \HGamma(\mu_{\Ss'}(Q,W))\right)\rightarrow\bigoplus_{a\in \mu_{\Ss'}(Q)_1|s(a)=s_t}\Phi^{-1}_{\Ss',\epsilon'}\left(e_{t(a)}\cdot \HGamma(\mu_{\Ss'}(Q,W))\right)\rightarrow \Phi_{\Ss,\epsilon}^{-1}\left(e_{s_t}\cdot \HGamma(\mu_{\Ss}(Q,W))\right)\rightarrow.
\end{equation}
In either case, there is a quasi-isomorphism 
\begin{equation}
\label{unchanged}
\Phi_{\Ss,\epsilon}^{-1}\left(e_{i}\cdot\HGamma(\mu_{\Ss}(Q,W))\right)\cong \Phi_{\Ss',\epsilon'}^{-1}\left(e_{i}\cdot \HGamma(\mu_{\Ss'}(Q,W))\right)
\end{equation}
for $i\neq s_t$.
\end{theorem}
\begin{remark}
Let $\Ss=(s)$ be a sequence consisting of a single vertex.  Then $\Phi_{\Ss,-}$ is the quasi-inverse to the functor $F$ of \cite[Thm.3.2]{KellerMutations}.
\end{remark}
\begin{definition}
A torsion structure $(\mathcal{T},\mathcal{F})$ on an Abelian category $\mathcal{A}$ is a pair of full subcategories $\mathcal{T},\mathcal{F}\subset\mathcal{A}$ such that
\begin{enumerate}
\item
For all $M'\in\mathcal{T}$ and $M''\in\mathcal{F}$, $\Hom_{\mathcal{A}}(M',M'')=0$.
\item
For every $M\in \mathcal{A}$ there exists a short exact sequence
\[
0\rightarrow M'\rightarrow M\rightarrow M''\rightarrow 0
\]
with $M'\in\mathcal{T}$ and $M''\in\mathcal{F}$.
\end{enumerate}
\end{definition}

Let $\Ss=(s_1,\ldots,s_t)$ be a sequence of principal vertices of the ice quiver $Q$, and let $W\in\mathbb{C}Q/[\mathbb{C}Q,\mathbb{C}Q]$ be a nondegenerate algebraic potential with respect to $\Ss$.  Following \cite{Na13}, though with the modification discussed in Remark \ref{differenceRem}, we recursively define a torsion structure $(\mathcal{T}_{\Ss},\mathcal{F}_{\Ss})$ on the category $\Rmod{\HJac(Q,W)}$, as well as a sequence of signs $\epsilon_{\Ss}$.  We start by setting $\mathcal{T}_{\emptyset}=\Rmod{\HJac(Q,W)}$ and setting $\mathcal{F}_{\emptyset}$ to be the full subcategory containing the zero module.  This obviously provides a torsion structure for $\Rmod{\HJac(Q,W)}$.  

Now assume that we have defined $\mathcal{T}_{\Ss'}$, $\mathcal{F}_{\Ss'}$ and $\epsilon_{\Ss'}$.  We define 
\[
\Simp_{\Ss}:=\Phi_{\Ss',\epsilon_{\Ss'}}^{-1}(\Simp(\mu_{\Ss'}(Q))_{s_t}).  
\]
As in \cite[Thm.3.4]{Na13} there is an isomorphism
\[
\Phi_{\Ss',\epsilon_{\Ss'}}^{-1}\left(\Rmod{\HJac(\mu_{\Ss'}(Q,W))}\right)\cong \Rmod{\HJac(Q,W)}^{(\mathcal{F}_{\Ss'}[1],\mathcal{T}_{\Ss'})},
\]
and so, in particular, $(\Phi_{\Ss',\epsilon_{\Ss'}}(\mathcal{F}_{\Ss'})[1],\Phi_{\Ss',\epsilon_{\Ss'}}(\mathcal{T}_{\Ss'}))$ is a torsion structure on $\Rmod{\HJac(\mu_{\Ss'}(Q,W))}$.  Since $\Simp(\mu_{\Ss'}(Q))_{s_t}$ is simple, it follows that either $\Simp_{\Ss}\in \mathcal{F}_{\Ss'}[1]$ or $\Simp_{\Ss}\in \mathcal{T}_{\Ss'}$.  Given two subcategories $\mathcal{A}'$ and $\mathcal{A}''$ of an Abelian category $\mathcal{A}$, we define $\mathcal{A}'\star\mathcal{A}''\subset \mathcal{A}$ to be the full subcategory containing those objects $M$ for which there is a short exact sequence
\[
0\rightarrow M'\rightarrow M\rightarrow M''\rightarrow 0
\]
with $M'\in\mathcal{A}'$ and $M''\in \mathcal{A}''$.  Given an object $M\in\mathcal{A}$, we let $M^{\oplus}\subset \mathcal{A}$ be the full subcategory containing all direct sums of $M$.  We set
\begin{equation}
\label{TSdef}
\mathcal{F}_{\Ss}=\begin{cases} \Simp_{\Ss}^{\oplus}\star\mathcal{F}_{\Ss'}& \mathrm{if } \Simp_{\Ss}\in \mathcal{T}_{\Ss'}\\ \mathcal{F}_{\Ss'}\cap (\Simp_{\Ss}[-1]^{\perp})&\textrm{if }\Simp_{\Ss}\in\mathcal{F}_{\Ss'}[1].\end{cases}
\end{equation}
We then define $\mathcal{T}_{\Ss}:={}^\perp\mathcal{F}_{\Ss}$.  In the first instance, we extend $\epsilon_{\Ss'}$ to a sequence $\epsilon_{\Ss}$ by adding a $-$, in the second instance we add a $+$.
\begin{definition}
If the sequence of signs $\epsilon_{\Ss}$ ends with a $-$ we call it \textit{additive}, otherwise we call it \textit{subtractive}.  
\end{definition}
The reason for this odd-looking convention is that in the part of the torsion structure we care about the most is $\mathcal{F}_{\Ss}$, and in the additive case, we \textit{add} objects to $\mathcal{F}_{\Ss}$ at the final stage of its recursive definition.  We will sometimes abbreviate $\Phi_{\Ss,\epsilon_{\Ss}}$ to $\Phi_{\Ss}$ for notational convenience.  
\begin{definition}\label{shaveDef}
We define 
\[
\shave:=\begin{cases} \Simp_{\Ss}&\textrm{if }\Ss\textrm{ is additive}\\ \Simp_{\Ss}[-1]&\textrm{if }\Ss\textrm{ is subtractive.}\end{cases}
\]
\end{definition}
In other words, we define $\shave$ to be the shift of $\Simp_{\Ss}$ that belongs to the heart of the natural t structure of $\Dub(\Rmod{\HGamma(Q,W)})$.  Accordingly, in both cases, we consider $\shave$ as a $\HJac(Q,W)$-module.  

The following is proved in \cite[Lem.3.11]{KellerMutations}.
\begin{lemma}
\label{simpTransf}
There are isomorphisms
\begin{align}
\Phi_{(s),+}(\Simp(Q)_{s})\cong &\Simp(\mu_s(Q))_s[1]\nonumber \\
\Phi_{(s),-}(\Simp(Q)_{s})\cong &\Simp(\mu_s(Q))_s[-1]\label{shiftClust}
\end{align}
and distinguished triangles 
\begin{align}
\Phi_{(s),+}^{-1}\left(\Simp(\mu_s(Q))_j\right)\rightarrow \Ext^1\left(\Simp(Q)_s,\Simp(Q)_j\right)\otimes \Simp(Q)_s\rightarrow \Simp(Q)_j[1]\rightarrow \label{simpEx}\\
\Simp(Q)_j[-1]\rightarrow \Ext^2\left(\Simp(Q)_s,\Simp(Q)_j\right)\otimes \Simp(Q)_s\rightarrow \Phi_{(s),-}^{-1}\left(\Simp(\mu_s(Q))_j\right)\rightarrow.\label{readback}
\end{align}
\end{lemma}
The existence of the second distinguished triangle follows from the existence of the first one, the isomorphism (\ref{shiftClust}), the 3--Calabi--Yau pairing on $\Dub^{\fd}(\Rmod{\HGamma(Q,W)})$, and the fact, proved in \cite[Lem.3.11]{KellerMutations}, that 
\[
\Phi_{(s),-}\colon \Dub\left(\Rmod{\HGamma(Q,W)}\right)\rightarrow\Dub\left(\Rmod{\HGamma(\mu_{s}(Q,W))}\right)
\]
is a quasi-inverse to 
\[
\Phi_{(s),+}\colon \Dub\left(\Rmod{\HGamma(\mu_s(Q,W))}\right)\rightarrow \Dub\left(\Rmod{\HGamma(Q,W)}\right).
\]
Here the second functor is as in Theorem \ref{KYThm}, considering $(s)$ as a sequence of vertices of the mutated quiver $\mu_s(Q)$.  Given a sequence of vertices $\Ss$ and a sequence of signs $\epsilon$, both of the same length, we define $\Psi_{\Ss,\epsilon}$ to  be the map making the following diagram commute
\[
\xymatrix{
\KK\left(\Dub^{\fd}_{\princ}(\Rmod{\HGamma(Q,W)})\right)\ar[d]^=\ar[rr]^{\KK\left(\Phi_{\Ss,\epsilon}\right)}&&\KK\left(\Dub^{\fd}_{\princ}(\Rmod{\HGamma(\mu_{\Ss}(Q,W))})\right)\ar[d]^=\\
\mathbb{Z}^m\ar[rr]^{\Psi_{\Ss,\epsilon}}&&\mathbb{Z}^m.
}
\]
Note that $\Psi_{\Ss,\epsilon}$ depends on $\Ss$ and $\epsilon$, but not on $W$.  We abbreviate $\Psi_{\Ss,\epsilon_{\Ss}}$ to $\Psi_{\Ss}$.
\begin{proposition}
\label{invOD}
For all $\Ss=(s_1,\ldots,s_t)$, all sequences of signs $\epsilon$ of length $t$, and for all $\dd\in\mathbb{Z}^m$ we have 
\[
(\Psi_{\Ss,\epsilon}(\dd),\Psi_{\Ss,\epsilon}(\dd))_{\mu_{\Ss}(Q)}=(\dd,\dd)_Q\mod 2.
\]
\end{proposition}
\begin{proof}
For the inductive step we can assume that
\[
(\Psi_{\Ss,\epsilon}(\dd),\Psi_{\Ss,\epsilon}(\dd))_{\mu_{\Ss}(Q)}=(\Psi_{s_1,\epsilon_1}(\dd),\Psi_{s_1,\epsilon_1}(\dd))_{\mu_{s_1}(Q)}\mod 2.
\]
For $M'$ and $M''$ a pair of $\HJac(Q,W)$-modules, there is an identity 
\begin{align*}
([M'],[M''])_Q+([M''],[M'])_Q=&\langle [M'],[M'']\rangle_Q\mod 2\\
=&\sum_{g\in\mathbb{Z}}\dim\left(\Ext^g(M',M'')\right)\mod 2,
\end{align*}
and so, since the Euler form of a category is invariant under derived equivalence, we deduce that for all dimension vectors $\dd',\dd''$, there is an equality
\[
(\Psi_{s,\pm}^{-1}(\dd'),\Psi_{s,\pm}^{-1}(\dd''))_Q+(\Psi_{s,\pm}^{-1}(\dd''),\Psi_{s,\pm}^{-1}(\dd'))_Q=(\dd',\dd'')_{\mu_s(Q)}+(\dd'',\dd')_{\mu_s(Q)}\mod 2.
\]
It follows that it is enough to show that
\[
(\Psi^{-1}_{s,\pm }(1_i),\Psi^{-1}_{s,\pm }(1_i))_{Q}=1\mod 2
\]
for all $i,s\leq m$.  By Lemma \ref{simpTransf}, $\Psi_{s,\pm}^{-1}(1_s)={}-1_s$, and so we only need to consider the case in which $i\neq s$.  Then we have $\Psi^{-1}_{s,\pm}(1_i)=1_i-\max(0,\pm b_{si})\cdot 1_s$ from (\ref{simpEx}), and so in both the additive and subtractive cases we calculate
\[
(\Psi^{-1}_{s,\pm }(1_i),\Psi^{-1}_{s,\pm }(1_i))_{Q}=1+\max(0,\pm 2b_{si}^2).
\]
\end{proof}

Given a stability condition $\zeta\in\mathbb{H}_+^{Q_0}$ on a quiver $Q$, an algebraic potential $W$ for $Q$, and an interval $S\subset [0,\pi)$, we denote by $(\rmod{\Jac(Q,W)})^{\zeta}_{S}$ the full subcategory of $\rmod{\Jac(Q,W)}$ containing those modules $N$ such that the Harder--Narasimhan type $(\dd^1,\ldots,\dd^t)$ of $N$ only contains terms satisfying $\Mu^{\zeta}(\dd^g)\in S$, and for $W$ a formal potential we define $(\rmod{\HJac(Q,W)})^{\zeta}_S$ similarly.

The next proposition is only a slight modification of \cite[Prop.4.1]{Na13}, but we offer the proof for completeness, modifying the proof of \cite[Thm.4.17]{Efi11} for our purposes.
\begin{proposition}
\label{zetaTheta}
Let $W$ be a potential for $Q$, nondegenerate with respect to the sequence of vertices $\Ss$.  Then there is a Bridgeland stability condition $\zeta_{\Ss}\in\mathbb{H}_+^{n}$, and an angle $\theta_{\Ss}\in [0,\pi)$ such that 
\[
(\rmod{\HJac(Q,W)})^{\zeta_{\Ss}}_{[0,\theta_{\Ss}]}=\mathcal{F}_{\Ss}\cap \rmod{\HJac(Q,W)}
\]
and
\[
(\rmod{\HJac(Q,W)})^{\zeta_{\Ss}}_{(\theta_{\Ss},\pi)}=\mathcal{T}_{\Ss}\cap\rmod{\HJac(Q,W)}.
\]
Furthermore, we can choose $\zeta_{\Ss}$ and $\theta_{\Ss}$ so that 
\[
(\rmod{\HJac(Q,W)})^{\zeta_{\Ss}}_{(\theta_{\Ss}-\delta,\theta_{\Ss}+\delta)}=\shave^{\oplus}\bigcap \left(\rmod{\HJac(Q,W)}\right)
\]
for sufficiently small $\delta$.

\end{proposition}
\begin{proof}
Let $\Ss=(s_1,\ldots,s_t)$.  First consider the additive case.  The conditions on $\zeta_{\Ss}$ and $\theta_{\Ss}$ are implied by the conditions
\begin{enumerate}
\item
\label{condi1}
\[
\IIm\left(\exp(-\theta_{\Ss}\sqrt{-1})\Z^{\zeta_{\Ss}}\left([\Phi_{\Ss,\epsilon_{\Ss}}^{-1}(\Simp(\mu_{\Ss}(Q))_j)]\right)\right)> 0
\]
for all $j\neq s_t$, and
\item
\label{condi2}
\begin{align*}
\IIm\left(\exp(-\theta_{\Ss}\sqrt{-1})\Z^{\zeta_{\Ss}}([\Simp_{\Ss}])\right)= &0\\
\RRe\left(\exp(-\theta_{\Ss}\sqrt{-1})\Z^{\zeta_{\Ss}}([\Simp_{\Ss}])\right)> &0.
\end{align*}
\end{enumerate}
To see this, we first note that the conditions imply that 
\begin{align*}
\mathcal{F}_{\Ss}\cap \rmod{\HJac(Q,W)}\subset &(\rmod{\HJac(Q,W)})^{\zeta_{\Ss}}_{[0,\theta_{\Ss}]}\\
\mathcal{T}_{\Ss}\cap\rmod{\HJac(Q,W)}\subset &(\rmod{\HJac(Q,W)})^{\zeta_{\Ss}}_{(\theta_{\Ss},\pi)}\\
\shave^{\oplus}\bigcap \left(\rmod{\HJac(Q,W)}\right)=&(\rmod{\HJac(Q,W)})^{\zeta_{\Ss}}_{(\theta_{\Ss}-\delta,\theta_{\Ss}+\delta)}.
\end{align*}
Then equality follows from the inclusions
\begin{align*}
(\rmod{\HJac(Q,W)})^{\zeta_{\Ss}}_{[0,\theta_{\Ss}]}=((\rmod{\HJac(Q,W)})^{\zeta_{\Ss}}_{(\theta_{\Ss},\pi)})^{\perp}\subset (\mathcal{T}_{\Ss}\cap\rmod{\HJac(Q,W)})^{\perp}=\mathcal{F}_{\Ss}\cap \rmod{\HJac(Q,W)}\\
(\rmod{\HJac(Q,W)})^{\zeta_{\Ss}}_{(\theta_{\Ss},\pi)}=^\perp\!\!(\rmod{\HJac(Q,W)})^{\zeta_{\Ss}}_{[0,\theta_{\Ss}]}\subset^\perp\!\!(\mathcal{F}_{\Ss}\cap \rmod{\HJac(Q,W)})=\mathcal{T}_{\Ss}\cap\rmod{\HJac(Q,W)}
\end{align*}
We achieve conditions (\ref{condi1}) and (\ref{condi2}) by setting 
\begin{equation}
\label{strictStab}
\IIm(\Z^{\zeta_{\Ss}}(1_s))=1
\end{equation}
for all $s\in Q_0$, and 
\[
\RRe(\Z^{\zeta_{\Ss}}(\Phi_{\Ss,\epsilon_{\Ss}}^{-1}(\Simp(\mu_{\Ss}(Q,W))_j)))={}-1
\]
for $j\neq s_t$ and $\RRe(\Z^{\zeta_{\Ss}}(\Simp_{\Ss}))=0$, and setting $\theta_{\Ss}=\pi/2$.  For the subtractive case, we set $\RRe(\Z^{\zeta_{\Ss}}([\shave]))=-\delta'$ for $0<\delta'\ll 1$.
\end{proof}

We denote by $(D_{\Ss}^{\leq 0},D_{\Ss}^{\geq 1})$ the t structure on $\Dub(\Rmod{\HGamma(Q,W)})$ obtained by pulling back the standard t structure on $\Dub(\Rmod{\HGamma(\mu_{\Ss}(Q,W))})$ along the quasi-equivalence $\Phi_{\Ss}$, and we denote by $\HO_{\Ss}^n(M)$ the $n$th cohomology of $M$ with respect to this t structure.  Let $\ff\in\mathbb{N}^n$.  Since this t structure is a tilt of the usual t structure with respect to the torsion structure $(\mathcal{T}_{\Ss},\mathcal{F}_{\Ss})$ it follows that $\HGamma(Q,W)_{\ff}\in D_{\Ss}^{\leq 1}$.  On the other hand, since each simple module $\Simp(\mu_{\Ss}(Q))_{j}$ is in $\Phi_{\Ss}(\mathcal{F}_{\Ss})[1]$ or $\Phi_{\Ss}(\mathcal{T}_{\Ss})$ it follows that $\HGamma(Q,W)_{\ff}\in \:^\perp \!D_{\Ss}^{\leq -1}$ .  It follows by \cite[Lem.2.11]{Pl11} that there is an isomorphism
\begin{equation}
\label{PlamIso}
\Phi_{\Ss}(\HGamma(Q,W)_{\ff})\cong\HGamma(\mu_{\Ss}(Q,W))_{\ff'}\oplus \HGamma(\mu_{\Ss}(Q,W))_{\ff''}[-1]
\end{equation}
for some dimension vectors $\ff',\ff''\in\mathbb{N}^n$, where the isomorphism is in the category of graded modules (forgetting the differential).  In other words, i.e. in the terminology of \cite{Pl11}, $\Phi_{\Ss}(\HGamma(Q,W)_{\ff})\in\textrm{pr}_{\Dub(\Rmod{\HGamma(\mu_{\Ss}(Q,W))})}\HGamma(\mu_{\Ss}(Q,W))[-1]$.
\begin{proposition}
Let $Q$ be a quiver, and let $W$ be a potential for $Q$, nondegenerate with respect to the sequence of vertices $\Ss=(s_1,\ldots,s_t)$.  Let $\ff\in\mathbb{N}^n$ be a dimension vector.  Then $\HO^1_{\Ss}\left(\HGamma(Q,W)_{\ff}\right)$ is represented by a finite dimensional $\HJac(Q,W)$-module.
\end{proposition}
This is basically \cite[Cor.4.11]{Efi11}, but our proof is a little different, in part because our setup is different, by Remark \ref{differenceRem}.
\begin{proof}
It is sufficient to consider the case $\ff=1_s$ for some $s\in Q_0$.  By isomorphism (\ref{PlamIso}), $\Phi_{\Ss}(\HGamma(Q,W)_{1_s})$ is concentrated in degrees 1 and below, and so there is a natural map $h\colon\HJac(Q,W)_{1_s}\rightarrow R_{\Ss,s}$, where we define 
\[
R_{\Ss,s}=\HO^1_{\Ss}(\HGamma(Q,W)_{1_s}).
\]
The map $h$ is just the map from $\HJac(Q,W)_{1_s}$ to its torsion-free part with respect to the torsion structure $(\mathcal{T}_{\Ss},\mathcal{F}_{\Ss})$, and so in particular $h$ is a surjection onto a $\HJac(Q,W)$-module.  Recall that the \textit{top} of a module is its largest semisimple quotient.  The map $h$ induces a surjection
\[
\Simp(Q)_s=\Top\left(\HJac(Q,W)_{1_s}\right)\rightarrow \Top\left(R_{\Ss,s}\right)
\]
from which we deduce that there is an isomorphism $\Top(R_{\Ss,s})\cong\Simp(Q)_s$, and $R_{\Ss,s}$ is indecomposable.  Now the proof is by induction.  First, assume that $\Ss$ is subtractive, then the surjection $h$ factors through a map $R_{\Ss',s}\rightarrow R_{\Ss,s}$ and the finite-dimensionality of $R_{\Ss,s}$ follows from the finite-dimensionality of $R_{\Ss',s}$.  On the other hand, if $\Ss$ is additive, then since $R_{\Ss,s}\in\mathcal{F}_{\Ss}$ there is a short exact sequence
\[
0\rightarrow \bigoplus_{P}S_{\Ss}\rightarrow R_{\Ss,s}\rightarrow R_{\Ss',s}\rightarrow 0
\]
where $P$ is some indexing set, by the construction of the torsion structure (\ref{TSdef}).  Since $R_{\Ss,s}$ is indecomposable, it follows that $|P|\leq \dim(\Ext^1(R_{\Ss',s},S_{\Ss}))$, which is finite by finite dimensionality of $R_{\Ss',s}$, and so $R_{\Ss,s}$ is an extension of two finite-dimensional modules.

\end{proof}

\begin{proposition}
\label{eqSchemes}
Let $Q$ be a quiver, let $W$ be an algebraic potential for $Q$, nondegenerate with respect to the sequence of principal vertices $\Ss$.  Let $\theta_{\Ss}$ and $\zeta_{\Ss}$ be as in Proposition \ref{zetaTheta}, and let $\dd\in\mathbb{N}^m$ be a dimension vector with $\Mu^{\zeta_{\Ss}}(\dd)\leq \theta_{\Ss}$.  Then there is an isomorphism of schemes
\begin{equation}
\label{diffIso}
\Gr_{{}-\Psi_{\Ss}(\dd)}\left(\HO^1(\Phi_{\Ss}(\HGamma(Q,W)_{\ff}))\right)\cong \Msp(Q,W)^{\zeta_{\Ss},\theta_{\Ss}\sfr,\nilp}_{\ff,\dd},
\end{equation}
and the right hand side of (\ref{diffIso}) is a union of connected components of $\Msp(Q,W)^{\zeta_{\Ss},\theta_{\Ss}\sfr}_{\ff,\dd}$.
\end{proposition}
The right hand side of (\ref{diffIso}) is introduced in equation (\ref{sfrIntro}), which uses the stability condition introduced just above (\ref{sfrIntro}).  The notation on the right hand side of (\ref{diffIso}) uses Convention \ref{QWconvention}.  On the left hand side of (\ref{diffIso}), the first subscript $-\Psi_{\Ss}(\dd)$ is as defined before Proposition \ref{invOD}.  
\begin{proof}
This is almost the result stated in \cite[Prop.6.3]{Efi11}, and carefully proved as \cite[Prop.4.35]{DMSS13}.  The proof of \cite[Prop.4.35]{DMSS13} gives the isomorphism (\ref{diffIso}) without modificaton.  The second statement is a consequence of \cite[Prop.3.1]{Efi11} for the case of generic $W$, and is given by the proof of \cite[Prop.4.28]{DMSS13} for general nondegenerate $W$.
\end{proof}
\begin{corollary}
\label{allW}
If $W\in\mathbb{C}Q/[\mathbb{C}Q,\mathbb{C}Q]$ is a nondegenerate algebraic potential with respect to the sequence of mutations $\Ss$, and $\dd\in\mathbb{N}^m$ satisfies $\Mu^{{\zeta}_{\Ss}}(\dd)\leq \theta_{\Ss}$, then the stack $\crit\left(\WWW_{[0,\theta_{\Ss}],\dd}^{\zeta_{\Ss}}\right)\cap\Mst(Q)_{[0,\theta_{\Ss}],\dd}^{\zeta_{\Ss},\nilp}$ is a union of connected components of $\crit(\WWW_{[0,\theta_{\Ss}],\dd}^{\zeta_{\Ss}})$.
\end{corollary}
For generic $W$, this is a direct consequence of \cite[Prop.3.1]{Efi11}.
\begin{proof}
The claim is equivalent to the claim that $\crit\left(\Tr(W)^{\zeta_{\Ss}}_{[0,\theta_{\Ss}],\dd}\right)\cap X(Q)_{[0,\theta_{\Ss}],\dd}^{\zeta_{\Ss},\nilp}$ is a union of connected components of $\crit\left(\Tr(W)^{\zeta_{\Ss}}_{[0,\theta_{\Ss}],\dd}\right)$.  Pick $\ff\in\mathbb{N}^m$ satisfying $\ff\geq \dd$.  Let $f$ be the function induced by $\Tr(W)_{\dd}$ on the stack
\[
(X(Q)^{\zeta_{\Ss}}_{[0,\theta_{\Ss}],\dd}\times V^{\surj}_{\ff,\dd})/G_{\dd}, 
\]
where $V^{\surj}_{\ff,\dd}$ is as introduced after (\ref{VDef}).  This stack is a scheme by \cite[Prop.23]{EdGr98}.  It is sufficient to prove that
\[
\crit(f)\cap \left(X(Q)^{\zeta_{\Ss},\nilp}_{[0,\theta_{\Ss}],\dd}\times V^{\surj}_{\ff,\dd}\right)/G_{\dd}
\]
is a union of connected components of $\crit(f)$.  This follows from the fact that 
\[
\left(X(Q)^{\zeta_{\Ss}}_{[0,\theta_{\Ss}],\dd}\times V^{\surj}_{\ff,\dd}\right)/G_{\dd}\subset \Msp(Q)^{\zeta_{\Ss},\theta_{\Ss}\sfr}_{\ff,\dd}
\]
is an open inclusion of schemes, along with the last part of Proposition \ref{eqSchemes}.

\end{proof}
Let $\dd\in\mathbb{N}^m$ satisfy $\Mu^{\zeta_{\Ss}}(\dd)\leq \theta_{\Ss}$.  Following our conventions, we define
\begin{align*}
&\Ho\left(p_{[0,\theta_{\Ss}],\dd,*}^{\zeta_{\Ss}}\left(\phim{\WWW^{\zeta_{\Ss}}_{[0,\theta_{\Ss}],\dd}}\mathbb{Q}_{\Mst(Q)^{\zeta_{\Ss}}_{[0,\theta_{\Ss}],\dd}}\right)_{\nilp}\right):=\\&\lim_{\ff\mapsto \infty}\left( \Ho\left( \pi^{\zeta_{\Ss},\theta_{\Ss}\sfr}_{\ff,\dd,*}\left( \phim{\WW_{\ff,\dd}^{\zeta_{\Ss},\theta_{\Ss}\sfr}}\mathbb{Q}_{\Msp(Q)_{\ff,\dd}^{\zeta_{\Ss},\theta_{\Ss}\sfr}}\right)_{\nilp}\right)\right)
\end{align*}
and
\begin{align*}
&\Ho\left(\Dim_{*}\left(\phim{\WWW^{\zeta_{\Ss}}_{[0,\theta_{\Ss}],\dd}}\mathbb{Q}_{\Mst(Q)^{\zeta_{\Ss}}_{[0,\theta_{\Ss}],\dd}}\right)_{\nilp}\right):=\\&\lim_{\ff\mapsto \infty}\left( \Ho\left( \dim_*\pi^{\zeta_{\Ss},\theta_{\Ss}\sfr}_{\ff,\dd,*}\left( \phim{\WW_{\ff,\dd}^{\zeta_{\Ss},\theta_{\Ss}\sfr}}\mathbb{Q}_{\Msp(Q)_{\ff,\dd}^{\zeta_{\Ss},\theta_{\Ss}\sfr}}\right)_{\nilp}\right)\right).
\end{align*}
By Corollary \ref{allW},
\begin{align*}
&\Msp(Q)^{\zeta_{\Ss},\theta_{\Ss}\sfr,\nilp}_{\ff,\dd}\cap \supp\left(\phim{\WW_{\ff,\dd}^{\zeta_{\Ss},\theta_{\Ss}\sfr}}\mathbb{Q}_{\Msp(Q)_{\ff,\dd}^{\zeta_{\Ss},\theta_{\Ss}\sfr}}\right)=
\\&(\pi_{\ff,\dd}^{\zeta_{\Ss},\theta_{\Ss}\sfr})^{-1}(0)\cap \supp\left(\phim{\WW_{\ff,\dd}^{\zeta_{\Ss},\theta_{\Ss}\sfr}}\mathbb{Q}_{\Msp(Q)_{\ff,\dd}^{\zeta_{\Ss},\theta_{\Ss}\sfr}}\right)
\end{align*}
is a union of connected components of $\supp\left(\phim{\WW_{\ff,\dd}^{\zeta_{\Ss},\theta_{\Ss}\sfr}}\mathbb{Q}_{\Msp(Q)_{\ff,\dd}^{\zeta_{\Ss},\theta_{\Ss}\sfr}}\right)$, and so there are natural isomorphisms
\begin{align*}
&\Ho\left( \pi^{\zeta_{\Ss},\theta_{\Ss}\sfr}_{\ff,\dd,*}\left( \phim{\WW_{\ff,\dd}^{\zeta_{\Ss},\theta_{\Ss}\sfr}}\mathbb{Q}_{\Msp(Q)_{\ff,\dd}^{\zeta_{\Ss},\theta_{\Ss}\sfr}}\right)_{\nilp}\right)\cong
\\&\left(\Ho\left( \pi^{\zeta_{\Ss},\theta_{\Ss}\sfr}_{\ff,\dd,*}\left( \phim{\WW_{\ff,\dd}^{\zeta_{\Ss},\theta_{\Ss}\sfr}}\mathbb{Q}_{\Msp(Q)_{\ff,\dd}^{\zeta_{\Ss},\theta_{\Ss}\sfr}}\right)\right)\right)_{\nilp}
\end{align*}
and
\begin{align}\label{Hswap}
&\Ho\left(p_{[0,\theta_{\Ss}],\dd,*}^{\zeta_{\Ss}}\left(\phim{\WWW^{\zeta_{\Ss}}_{[0,\theta_{\Ss}],\dd}}\ICS_{\Mst(Q)^{\zeta_{\Ss}}_{[0,\theta_{\Ss}],\dd}}(\mathbb{Q})\right)_{\nilp}\right)\cong
\\&\left(\Ho\left(p_{[0,\theta_{\Ss}],\dd,*}^{\zeta_{\Ss}}\left(\phim{\WWW^{\zeta_{\Ss}}_{[0,\theta_{\Ss}],\dd}}\ICS_{\Mst(Q)^{\zeta_{\Ss}}_{[0,\theta_{\Ss}],\dd}}(\mathbb{Q})\right)\right)\right)_{\nilp}.\nonumber
\end{align}
There is a natural inclusion of monoids $\mathbb{N}^m\rightarrow \Msp(Q)$ sending $\dd$ to the point representing the direct sum $\bigoplus_{s\leq m}\Simp(Q)^{\dd_s}$, which is an isomorphism onto $\Msp(Q)^{\nilp}$, and has left inverse $\dim\colon\Msp(Q)\rightarrow\mathbb{N}^m$.
We deduce from (\ref{Hswap}) that there is an isomorphism

\begin{align}\label{nilpCom}
&\Ho\left(\Dim_*\left(\phim{\WWW^{\zeta_{\Ss}}_{[0,\theta_{\Ss}],\dd}}\ICS_{\Mst(Q)^{\zeta_{\Ss}}_{[0,\theta_{\Ss}],\dd}}(\mathbb{Q})\right)_{\nilp}\right)\cong \\&\dim_*\left(\left(\Ho\left(p_{[0,\theta_{\Ss}],\dd,*}^{\zeta_{\Ss}}\left(\phim{\WWW^{\zeta_{\Ss}}_{[0,\theta_{\Ss}],\dd}}\ICS_{\Mst(Q)^{\zeta_{\Ss}}_{[0,\theta_{\Ss}],\dd}}(\mathbb{Q})\right)\right)\right)_{\nilp}\right)\nonumber
\end{align}
and by the same argument, there is an isomorphism
\begin{align}\label{nilpCom2}
&\Ho\left(\Dim_*\left(\phim{\WWW^{\zeta_{\Ss}\sst}_{\dd}}\ICS_{\Mst(Q)^{\zeta_{\Ss}\sst}_{\dd}}(\mathbb{Q})\right)_{\nilp}\right)\cong \\&\dim_*\left(\left(\Ho\left(p_{\dd,*}^{\zeta_{\Ss}\sst}\left(\phim{\WWW^{\zeta_{\Ss}\sst}_{\dd}}\ICS_{\Mst(Q)^{\zeta_{\Ss}\sst}_{\dd}}(\mathbb{Q})\right)\right)\right)_{\nilp}\right).\nonumber
\end{align}
 
\subsection{Cluster mutation from torsion pairs}

\begin{proposition}
\label{TPure}
Let $Q$ be a quiver, and let $W\in\mathbb{C}Q/[\mathbb{C}Q,\mathbb{C}Q]$ be an algebraic potential, nondegenerate with respect to the sequence of principal vertices $\Ss=(s_1,\ldots, s_t)$.  Let $\zeta_{\Ss}$ be as in Proposition \ref{zetaTheta}, and set $\gamma=\Mu^{\zeta_{\Ss}}(\shave)$.  Then
\begin{equation}
\Ho\left(\Dim_* \left(\phim{\WWW^{\zeta_{\Ss}\sst}_{\gamma}}\ICS_{\Mst(Q)_{\gamma}^{\zeta_{\Ss}\sst}}(\mathbb{Q})\right)_{\nilp}\right)
\end{equation}
is pure, of Tate type.  Furthermore, the pure monodromic mixed Hodge module
\begin{equation}
\label{secondOne}
\Ho\left(\Dim_* \left(\phim{\WWW^{\zeta_{\Ss}\sst}_{\gamma}}\QQ_{\Mst(Q)_{\gamma}^{\zeta_{\Ss}\sst}}\right)_{\nilp}\right)
\end{equation}
has trivial monodromy in the sense of Definition \ref{trivMonDef}.
\end{proposition}
\begin{proof}
By the construction of $\zeta_{\Ss}$ (see Proposition \ref{zetaTheta}), the only semistable nilpotent $\Jac(Q,W)$-modules of dimension vector $\dd$, where $\Mu^{\zeta_{\Ss}}(\dd)=\gamma$, are direct sums of $\shave$.  Fix $\dd=k\dim(\shave)$.  Then by Corollary \ref{allW}, $\Mst(Q,W)_{\dd}^{\zeta_{\Ss}\sst,\nilp}$ is a connected component of $\Mst(Q,W)^{\zeta_{\Ss}\sst}_{\dd}$, and is furthermore isomorphic to the smooth stack $\pt/ \Gl_{k}$, since it is isomorphic to the stack of nilpotent $\dd$-dimensional representations of $\Jac(Q,W)$, and so by the above comment, is isomorphic to the classifying stack of $\Aut(\shave^{\oplus k})$.

Since $\Mu^{\zeta_{\Ss}}(\dd)=\gamma$ there is an equality $\Mst(Q)_{\dd}^{\zeta_{\Ss}\sst}=\Mst(Q)_{[0,\gamma],\dd}^{\zeta_{\Ss}}$.  By (\ref{bettStarForm}), there is an isomorphism
\begin{align}\label{frexp}
&\Ho\left(\Dim_* \left(\phim{\WWW^{\zeta_{\Ss}\sst}_{\dd}}\ICS_{\Mst(Q)_{\dd}^{\zeta_{\Ss}\sst}}(\mathbb{Q})\right)_{\nilp}\right)\cong
\\&\lim_{\ff\mapsto\infty}\Ho\left(\dim_* \left(\phim{\WW^{\zeta_{\Ss},\gamma\sfr}_{\ff,\dd}}\ICS_{\Msp(Q)_{\ff,\dd}^{\zeta_{\Ss},\gamma\sfr}}(\mathbb{Q})\right)_{\nilp}\otimes\mathbb{L}^{\ff\cdot\dd/2}\right).\nonumber
\end{align}
By Proposition \ref{eqSchemes} there is an isomorphism of schemes 
\begin{align}
\label{firstGr}
\Msp(Q,W)_{\ff,\dd}^{\zeta_{\Ss},\gamma\sfr,\nilp}\cong \Gr_{k\cdot 1_{s_t}}(\HO^1(\Phi_{\Ss}(\HGamma(Q,W)_{\ff})))
\end{align}
taking points of the left hand side to surjections $\HO^1(\Phi_{\Ss}(\HGamma(Q,W)_{\ff}))\rightarrow \Simp(\mu_{\Ss}(Q))_{s_t}^{\oplus k}$.  We write 
\[
\Top(\HO^1(\Phi_{\Ss}(\HGamma(Q,W)_{\ff})))=\bigoplus_{s\leq n}\Simp(\mu_{\Ss}(Q))_s^{\oplus c_s}
\]
for some integers $c_s$.  Then 
\begin{align}
\label{secondGr}
\Gr_{k\cdot 1_{s_t}}(\HO^1(\Phi_{\Ss}(\HGamma(Q,W)_{\ff})))\cong \Gr(k,c_{s_t}).
\end{align}
In particular, we deduce that $\Gr_{k\cdot 1_{s_t}}(\HO^1(\Phi_{\Ss}(\HGamma(Q,W)_{\ff})))$ is scheme-theoretically smooth, and has trivial fundamental group.  Since $\Msp(Q,W)_{\ff,\dd}^{\zeta_{\Ss},\gamma\sfr,\nilp}$ is a scheme-theoretically smooth component of the critical locus of $\WW^{\zeta_{\Ss},\gamma\sfr}_{\ff,\dd}$, we deduce from the holomorphic Bott--Morse lemma (proved as in \cite[Lem.2.2]{Mi63}) that for any $x\in\Msp(Q,W)_{\ff,\dd}^{\zeta_{\Ss},\gamma\sfr,\nilp}$ there is an analytic open neighbourhood $U$ of $x$ in $\Msp(Q)_{\ff,\dd}^{\zeta_{\Ss},\gamma\sfr}$ where $\WW^{\zeta_{\Ss},\gamma\sfr}_{\ff,\dd}$ is written, after complex analytic change of coordinates, as $\sum_{i=1}^t x_i^2$, with $x_1,\ldots, x_t$ local defining equations for the variety $\Msp(Q,W)_{\ff,\dd}^{\zeta_{\Ss},\gamma\sfr,\nilp}$, and $t$ its codimension inside $\Msp(Q,W)_{\ff,\dd}^{\zeta_{\Ss},\gamma\sfr}$.  By the Thom--Sebastiani isomorphism, and the fact that the dimension of $\HO_c(\mathbb{A}^1,\phi_{x^2})$ is one, the complex of perverse sheaves 
\[
\phi_{\WW^{\zeta_{\Ss},\gamma\sfr}_{\ff,\dd}}\QQ_{\Msp(Q)^{\zeta_{\Ss},\gamma\sfr}_{\ff,\dd}}[\dim(\Msp(Q)^{\zeta_{\Ss},\gamma\sfr}_{\ff,\dd}) -1], 
\]
restricted to $\Msp(Q,W)_{\ff,\dd}^{\zeta_{\Ss},\gamma\sfr,\nilp}$, is a rank one local system, and a perverse sheaf, since $\phi_{\WW^{\zeta_{\Ss},\gamma\sfr}_{\ff,\dd}}[-1]$ preserves the perverse t structure \cite{BBD}.  By triviality of the fundamental group of its support, this local system is globally trivial, and so
\begin{align}\label{ZM}
\left(\phim{\WW^{\zeta_{\Ss},\gamma\sfr}_{\ff,\dd}}\ICS_{\Msp(Q)^{\zeta_{\Ss},\gamma\sfr}_{\ff,\dd}}(\mathbb{Q})\right)_{\nilp}\cong&\QQ_{\Msp(Q)^{\zeta_{\Ss},\gamma\sfr,\nilp}_{\ff,\dd}}\otimes\LL^{t/2}\otimes\LL^{-\dim(\Msp(Q)^{\zeta_{\Ss},\gamma\sfr}_{\ff,\dd})/2}\\
\cong & \ICS_{\Msp(Q)^{\zeta_{\Ss},\gamma\sfr,\nilp}_{\ff,\dd}}(\mathbb{Q}).\nonumber
\end{align}
It then follows from (\ref{firstGr}) and (\ref{secondGr}) that 
\begin{equation}
\label{normali}
\Ho\left(\dim_{*}\left(\phim{\WW^{\zeta_{\Ss},\gamma\sfr}_{\ff,\dd}}\ICS_{\Msp(Q)_{\ff,\dd}^{\zeta_{\Ss},\gamma\sfr}}(\mathbb{Q})\right)_{\nilp}\right)\cong\HO(\Gr(k,c_{s_t}),\mathbb{Q})_{\vir},
\end{equation}
and so both sides of (\ref{frexp}) are pure.  

For the monodromy statement, we tensor both sides of (\ref{normali}) with $\LL^{\dim(\Msp(Q)_{\ff,\dd}^{\zeta_{\Ss},\gamma\sfr})/2}$ to obtain
\begin{equation}
\label{unnormali}
\Ho\left(\dim_*\left(\phim{\WW^{\zeta_{\Ss},\gamma\sfr}_{\ff,\dd}}\QQ_{\Msp(Q)_{\ff,\dd}^{\zeta_{\Ss},\gamma\sfr}}\right)_{\nilp}\right)\cong \HO(\Gr(k,c_{s_t}),\mathbb{Q})\otimes\LL^{t/2},
\end{equation}
where $t$ is the codimension of $\Mst(Q,W)_{\dd}^{\zeta_{\Ss}\sst,\nilp}$ inside $\Mst(Q)_{\dd}^{\zeta_{\Ss}\sst}$.  By definition, the monodromic mixed Hodge module (\ref{secondOne}) is given by (\ref{unnormali}) as we let $\ff\mapsto\infty$.  The number $t$ is equal to the difference
\[
(\Psi_{\Ss}^{-1}(k\cdot 1_{s_t}),\Psi_{\Ss}^{-1}(k \cdot 1_{s_t}))_Q-(k\cdot 1_{s_t},k\cdot 1_{s_t})_{\mu_{\Ss}(Q)}
\]
where the notation is as in Proposition \ref{invOD}.  By Proposition \ref{invOD} this number is even, and so the right hand side of (\ref{unnormali}) has trivial monodromy by Remark \ref{monRem}.
\end{proof}
In the course of the proof we have shown that for $\gamma=\Mu^{\zeta_{\Ss}}(\shave)$, the monodromic mixed Hodge module $\Ho\left(\Dim_*\left(\phim{\WWW^{\zeta_{\Ss}\sst}_{\gamma}}\ICS_{\Mst(Q)^{\zeta_{\Ss}\sst}_{\gamma}}(\mathbb{Q})\right)_{\nilp}\right)$ is isomorphic to $\Ho\left(\Dim_*\ICS_{\Mst(Q,W)_{\gamma}^{\zeta_{\Ss}\sst,\nilp}}(\mathbb{Q})\right)$.  Loosely\footnote{Loose, because we do not actually define these mixed Hodge modules, but work with approximations to them on algebraic varieties.} speaking, this means that for $\dd$ of slope $\gamma$, we can replace the restriction to the nilpotent locus of the vanishing cycle monodromic mixed Hodge module on $\Mst(Q)^{\zeta_{\Ss}\sst}_{\dd}$ by (a twist of) the constant monodromic mixed Hodge module supported on the smooth connected component of the critical locus of $\WWW_{\dd}^{\zeta_{\Ss}\sst}$ corresponding to nilpotent modules.  More explicitly, the proof shows that there are isomorphisms
\begin{align}
\Ho\left(\Dim_*\left(\phim{\WWW^{\zeta_{\Ss}\sst}_{\dd}}\ICS_{\Mst(Q)^{\zeta_{\Ss}\sst}_{\dd}}(\mathbb{Q})\right)_{\nilp}\right)\cong\nonumber
&
\Ho\left(\Dim_*\ICS_{\Mst(Q,W)_{\dd}^{\zeta_{\Ss}\sst,\nilp}}(\mathbb{Q})\right)\\
\cong&\HO(\pt/\Gl_k)_{\vir}\label{Glout}
\end{align}
where $\dd=k\dim(\shave)$.

The following proposition is a consequence of Theorem \ref{rays} (the wall-crossing isomorphism), and isomorphisms (\ref{nilpCom}) and (\ref{nilpCom2}).  It gives a recursive formula for the vanishing cycle cohomology of the stack of all finite-dimensional modules in $\mathcal{F}_{\Ss}$.

\begin{proposition}
Let $Q$ be an ice quiver, let $\Ss=(s_1,\ldots,s_t)$ be a sequence of principal vertices of $Q$, and let $W$ be an algebraic potential for $Q$, nondegenerate with respect to $\Ss$.  Let $\theta_{\Ss}$ and $\zeta_{\Ss}$ be as in Proposition \ref{zetaTheta}.  Let $\gamma=\Mu^{\zeta_{\Ss}}(\shave)$.  If $\Ss$ is additive, then there is an isomorphism of monodromic mixed Hodge modules.
\begin{align}
\label{indIso1}
&\Ho\left(\Dim_* \left(\phim{\WWW^{\zeta_{\Ss}}_{[0,\theta_{\Ss}]}}\ICS_{\Mst(Q)_{[0,\theta_{\Ss}]}^{\zeta_{\Ss}}}(\mathbb{Q})\right)_{\nilp}\right)\cong \\
&\Ho\left(\Dim_*\ICS_{\Mst(Q,W)_{\gamma}^{\zeta_{\Ss}\sst,\nilp}}(\mathbb{Q})\right)\boxtimes^{\tw}_+\Ho\left(\Dim_* \left(\phim{\WWW_{[0,\theta_{\Ss'}]}^{\zeta_{\Ss'}} }\ICS_{\Mst(Q)_{[0,\theta_{\Ss'}]}^{\zeta_{\Ss'}} }(\mathbb{Q})\right)_{\nilp}\right),\nonumber
\end{align}
while if $\Ss$ is subtractive, there is an isomorphism
\begin{align}
\label{indIso2}
&\Ho\left(\Dim_* \left(\phim{\WWW_{[0,\theta_{\Ss'}]}^{\zeta_{\Ss'}} }\ICS_{\Mst(Q)_{[0,\theta_{\Ss'}]}^{\zeta_{\Ss'}} }(\mathbb{Q})\right)_{\nilp}\right)\cong \\
&\Ho\left(\Dim_*\ICS_{\Mst(Q,W)_{\gamma}^{\zeta_{\Ss}\sst,\nilp}}(\mathbb{Q})\right)\boxtimes^{\tw}_+\Ho\left(\Dim_* \left(\phim{\WWW^{\zeta_{\Ss}}_{[0,\theta_{\Ss}]}}\ICS_{\Mst(Q)_{[0,\theta_{\Ss}]}^{\zeta_{\Ss}}}(\mathbb{Q})\right)_{\nilp}\right).\nonumber
\end{align}
\end{proposition}

\begin{remark}
\label{untwist}
For future reference, we write down the untwisted versions of isomorphisms (\ref{indIso1}) and (\ref{indIso2}).  Tensoring both sides of the component of (\ref{indIso1}) supported at $\dd$ by $\LL^{{}-(\dd,\dd)/2}$, the isomorphism becomes
\begin{align}
&\Ho\left(\Dim_* \left(\phim{\WWW^{\zeta_{\Ss}}_{[0,\theta_{\Ss}]}}\QQ_{\Mst(Q)_{[0,\theta_{\Ss}]}^{\zeta_{\Ss}}}\right)_{\nilp}\right)\cong \\
&\bigoplus_{\substack{\dd''\in\mathbb{N}^m|\Mu^{\zeta_{\Ss}}(\dd'')\leq \theta_{\Ss}\\ \dd'\in\mathbb{N}\dim(\shave)}}\Ho\left(\Dim_*\QQ_{\Mst(Q,W)_{\dd'}^{\zeta_{\Ss}\sst,\nilp}}\right)\boxtimes_+\nonumber\\
&\boxtimes_+\Ho\left(\Dim_* \left(\phim{\WWW_{[0,\theta_{\Ss'}],\dd''}^{\zeta_{\Ss'}} }\QQ_{\Mst(Q)_{[0,\theta_{\Ss'}],\dd''}^{\zeta_{\Ss'}} }\right)_{\nilp}\right)\otimes\LL^{{}-(\dd',\dd'')},\nonumber
\end{align}
while (\ref{indIso2}) becomes
\begin{align}
&\Ho\left(\Dim_* \left(\phim{\WWW_{[0,\theta_{\Ss'}]}^{\zeta_{\Ss'}} }\QQ_{\Mst(Q)_{[0,\theta_{\Ss'}]}^{\zeta_{\Ss'}} }\right)_{\nilp}\right)\cong \\
&\bigoplus_{\substack{\dd''\in\mathbb{N}^m|\Mu^{\zeta_{\Ss}}(\dd'')\leq \theta_{\Ss}\\ \dd'\in\mathbb{N}\dim(\shave)}}\Ho\left(\Dim_*\QQ_{\Mst(Q,W)_{\dd'}^{\zeta_{\Ss}\sst,\nilp}}\right)\boxtimes_+\nonumber\\
&\boxtimes_+\Ho\left(\Dim_* \left(\phim{\WWW^{\zeta_{\Ss}}_{[0,\theta_{\Ss}],\dd''}}\QQ_{\Mst(Q)_{[0,\theta_{\Ss}],\dd''}^{\zeta_{\Ss}}}\right)_{\nilp}\right)\otimes\LL^{{}-(\dd',\dd'')}.\nonumber
\end{align}
\end{remark}

Applying $\chi_Q$ to the equalities in the Grothendieck group of $\Dulf(\MMHM(\mathbb{N}^{Q_0}))$ induced by the isomorphisms (\ref{indIso1}) and (\ref{indIso2}) respectively, we deduce the following corollary.
\begin{corollary}
\label{indCor}
In the ring $\hat{\Ror}_Q$, there is an equality
\begin{align*}
&\chi_Q\left(\left[\Dim_* \left(\phim{\WWW^{\zeta_{\Ss}}_{[0,\theta_{\Ss}]}}\ICS_{\Mst(Q)_{[0,\theta_{\Ss}]}^{\zeta_{\Ss}}}(\mathbb{Q})\right)_{\nilp}\right]\right)=\\
&\EXP\left(\frac{Y^{\pm\dim(\Simp_{\Ss})}}{q^{-1/2}-q^{1/2}}\right)^{\pm 1}\chi_Q\left(\left[\Dim_* \left(\phim{\WWW_{[0,\theta_{\Ss'}]}^{\zeta_{\Ss'}} }\ICS_{\Mst(Q)_{[0,\theta_{\Ss'}]}^{\zeta_{\Ss'}} }(\mathbb{Q})\right)_{\nilp}\right]\right),
\end{align*}
where the sign is positive in the additive case, and negative in the subtractive case.
\end{corollary}
In the above corollary, we have used equation (\ref{dilogId}) and (\ref{Glout}) to write
\[
\chi_Q\left(\Ho\left(\Dim_*\ICS_{\Mst(Q,W)_{\gamma}^{\zeta_{\Ss}\sst,\nilp}}(\mathbb{Q})\right)\right)=\pleth{\frac{Y^{\dim(\shave)}}{q^{-1/2}-q^{1/2}}},
\]
where $\shave$ is as in Definition \ref{shaveDef}.
\begin{theorem}\cite[Thm.5.11]{Efi11}\label{sfrThm}
Let $Q$ be an ice quiver, let $\Ss$ be a sequence of vertices of $Q_{\princ}$, and let $W$ be an algebraic potential, nondegenerate with respect to $\Ss$.  Then there is an identity in $\hat{\Tor}_{\Lambda}$
\begin{align}
\label{mutId}
\mu_{\Ss}(M)(\ff)= &\iota\chi_Q\left(\left[\Ho\left(\Dim_* \left(\phim{\WWW_{[0,\theta_{\Ss}]}^{\zeta_{\Ss}} }\ICS_{\Mst(Q)_{[0,\theta_{\Ss}]}^{\zeta_{\Ss}} }(\mathbb{Q})\right)_{\nilp}\right)\right]\right)X^{[\HGamma(Q,W)_{\ff}]}\\&\left(\iota\chi_Q\left(\left[\Ho\left(\Dim_* \left(\phim{\WWW_{[0,\theta_{\Ss}]}^{\zeta_{\Ss}} }\ICS_{\Mst(Q)_{[0,\theta_{\Ss}]}^{\zeta_{\Ss}} }(\mathbb{Q})\right)_{\nilp}\right)\right]\right)\right)^{-1}.\nonumber
\end{align}
\end{theorem}
\begin{proof}
Let $\Ss=(s_1,\ldots,s_t)$.  As ever the proof is by induction on the length of $\Ss$.  There is an equality
\begin{equation}
\label{commDang}
\pleth{\frac{\iota(Y^{\pm[\Simp_{\Ss}]})}{q^{-1/2}-q^{1/2}}}^{\pm 1}X^{[\Phi^{-1}_{\Ss'}(e_s\cdot \HGamma(\mu_{\Ss'}(Q,W)))]}\pleth{\frac{\iota(Y^{\pm[\Simp_{\Ss}]})}{q^{-1/2}-q^{1/2}}}^{\mp 1}=X^{[\Phi^{-1}_{\Ss'}(e_s\cdot \HGamma(\mu_{\Ss'}(Q,W)))]}
\end{equation}
if $s_t\neq s$, since then $\iota(Y^{\pm[\Simp_{\Ss}]})$ and $X^{[\Phi^{-1}_{\Ss'}(e_s\cdot \HGamma(\mu_{\Ss'}(Q,W)))]}$ commute by Remark \ref{muComm}.  In the additive case, we set the first sign to be positive, in the subtractive case we set it to be negative.  By the last statement of Theorem \ref{KYThm}, if $s\neq s_t$ we have the equality
\[
X^{[\Phi^{-1}_{\Ss'}(e_s\cdot \HGamma(\mu_{\Ss'}(Q,W)))]}=X^{[\Phi^{-1}_{\Ss}(e_s\cdot \HGamma(\mu_{\Ss}(Q,W)))]}.
\]
This demonstrates the inductive step, for $\Ss$ additive or subtractive, and for $s_t\neq s$.  

Now say $s_t=s$.  Firstly, assume that $\Ss$ is additive.  Then
\[
X^{[\Phi^{-1}_{\Ss'}(e_s\cdot \HGamma(\mu_{\Ss'}(Q,W)))]} \iota(Y^{[\Simp_{\Ss}]})= q\iota(Y^{[\Simp_{\Ss}]}) X^{[\Phi^{-1}_{\Ss'}(e_s\cdot \HGamma(\mu_{\Ss'}(Q,W)))]}
\]
by Remark \ref{muComm}.  By Example \ref{basicWCF}, the left hand side of (\ref{commDang}) is equal to
\[
X^{[\Phi^{-1}_{\Ss'}(e_s\cdot \HGamma(\mu_{\Ss'}(Q,W)))]}+X^{[\Phi^{-1}_{\Ss'}(e_s\cdot \HGamma(\mu_{\Ss'}(Q,W)))]+\iota[\Simp_{\Ss}]}.
\]
By (\ref{minusCase}) we have the identity
\[
[\Phi^{-1}_{\Ss}(e_{s_t}\cdot \HGamma(\mu_{\Ss}(Q,W)))]=\sum_{a\in \mu_{\Ss'}(Q)_1|s(a)=s_t}[\Phi_{\Ss'}^{-1}(e_{t(a)}\cdot \HGamma(\mu_{\Ss'}(Q,W)))]-[\Phi^{-1}_{\Ss'}(e_{s_t}\cdot \HGamma(\mu_{\Ss'}(Q,W)))]
\]
while Proposition \ref{lattComm} and (\ref{resolution}) give the identity
\begin{align}\label{Sid}
\iota[\Simp_{\Ss}]:=&{}-\KK(\nu)([\Simp_{\Ss}])\\=&-\left(\sum_{a\in \mu_{\Ss'}(Q)_1|s(a)=s_t}[\Phi_{\Ss}^{-1}(e_{t(a)}\cdot \HGamma(\mu_{\Ss}(Q,W))]\right)+\left(\sum_{a\in \mu_{\Ss'}(Q)_1|t(a)=s_t}[\Phi_{\Ss}^{-1}(e_{s(a)}\cdot \HGamma(\mu_{\Ss}(Q,W))]\right).\nonumber
\end{align}
We have again used that $[\Phi_{\Ss'}^{-1}(e_{s}\cdot \HGamma(\mu_{\Ss'}(Q,W)))]=[\Phi_{\Ss}^{-1}(e_{s}\cdot \HGamma(\mu_{\Ss}(Q,W)))]$ for $s\neq s_t$.  Finally, we deduce that in the case $s=s_t$, with $\Ss$ additive, the left hand side of (\ref{commDang}) is equal to
\begin{align}\label{desClus}
&\pleth{\frac{\iota(Y^{[\Simp_{\Ss}]})}{q^{-1/2}-q^{1/2}}}X^{[\Phi^{-1}_{\Ss'}(e_s\cdot \HGamma(\mu_{\Ss'}(Q,W)))]}\pleth{\frac{-\iota(Y^{[\Simp_{\Ss}]})}{q^{-1/2}-q^{1/2}}}=\\
\nonumber
&X^{\sum_{a\in \mu_{\Ss'}(Q)_1|s(a)=s}[\Phi_{\Ss}^{-1}(e_{t(a)}\cdot \HGamma(\mu_{\Ss}(Q,W)))]-[\Phi_{\Ss}^{-1}(e_{s}\cdot\HGamma(\mu_{\Ss}(Q,W)))]}+\\
&X^{\sum_{a\in \mu_{\Ss'}(Q)_1|t(a)=s}[\Phi_{\Ss}^{-1}(e_{s(a)}\cdot \HGamma(\mu_{\Ss}(Q,W)))]-[\Phi_{\Ss}^{-1}(e_{s}\cdot\HGamma(\mu_{\Ss}(Q,W)))]}\nonumber
\end{align}
as required.

By (\ref{plusCase}), in the subtractive case, we have the identity
\begin{equation}\label{Scas}
[\Phi^{-1}_{\Ss}(e_{s_t}\cdot \HGamma(\mu_{\Ss}(Q,W)))]=\sum_{a\in \mu_{\Ss'}(Q)_1|t(a)=s_t}[\Phi_{\Ss'}^{-1}(e_{s(a)}\cdot \HGamma(\mu_{\Ss'}(Q,W)))]-[\Phi^{-1}_{\Ss'}(e_{s_t}\cdot \HGamma(\mu_{\Ss'}(Q,W)))].
\end{equation}
By Remark \ref{muComm}, we have the commutation relation
\[
X^{[\Phi^{-1}_{\Ss'}(e_s\cdot \HGamma(\mu_{\Ss'}(Q,W)))]} \iota(Y^{[\Simp_{\Ss}[-1]]})= q^{-1}\iota(Y^{[\Simp_{\Ss}[-1]]}) X^{[\Phi^{-1}_{\Ss'}(e_s\cdot \HGamma(\mu_{\Ss'}(Q,W)))]}.
\]
By (\ref{reverseWall}), there is an identity
\begin{align}
&\pleth{\frac{\iota(Y^{-[\Simp_{\Ss}]})}{q^{-1/2}-q^{1/2}}}^{- 1}X^{[\Phi^{-1}_{\Ss'}(e_s\cdot \HGamma(\mu_{\Ss'}(Q,W)))]}\pleth{\frac{\iota(Y^{-[\Simp_{\Ss}]})}{q^{-1/2}-q^{1/2}}}=\\&X^{[\Phi^{-1}_{\Ss'}(e_s\cdot \HGamma(\mu_{\Ss'}(Q,W)))]}+X^{[\Phi^{-1}_{\Ss'}(e_s\cdot \HGamma(\mu_{\Ss'}(Q,W)))]+\iota([S_{\Ss}[-1]])}=\\
\nonumber
&X^{\sum_{a\in \mu_{\Ss'}(Q)_1|t(a)=s}[\Phi_{\Ss}^{-1}(e_{s(a)}\cdot \HGamma(\mu_{\Ss}(Q,W)))]-[\Phi_{\Ss}^{-1}(e_{s}\cdot\HGamma(\mu_{\Ss}(Q,W)))]}
+\\
&X^{\sum_{a\in \mu_{\Ss'}(Q)_1|s(a)=s}[\Phi_{\Ss}^{-1}(e_{t(a)}\cdot \HGamma(\mu_{\Ss}(Q,W)))]-[\Phi_{\Ss}^{-1}(e_{s}\cdot\HGamma(\mu_{\Ss}(Q,W)))]}
\nonumber
\end{align}
where the final identity is given by (\ref{Sid}) and (\ref{Scas}).

\end{proof}
It follows that the right hand side of (\ref{mutId}) is in $\Tor_{\Lambda}$, as opposed to the completion $\hat{\Tor}_{\Lambda}$, by \cite[Cor.5.2]{BZ05}.  We will see how to derive that result within the present framework in Section \ref{cluVia}.

\section{Proof of the purity conjecture}

The goal of this section is to prove Theorem \ref{KConj}, which is a purity result for the monodromic mixed Hodge module categorifying quantum cluster coefficients.

\begin{proposition}
\label{stackPure}
Let $Q$ be an ice quiver.  Let $\Ss=(s_1,\ldots,s_t)$ be a sequence of principal vertices of $Q$, let $W$ be an algebraic potential for $Q$, nondegenerate with respect to $\Ss$, and let $\zeta_{\Ss}$ and $\theta_{\Ss}$ be as in Proposition \ref{zetaTheta}.  Then the monodromic mixed Hodge module
\begin{equation}
\label{stacPure}
\Ho\left(\Dim_* \left(\phim{\WWW^{\zeta_{\Ss}}_{[0,\theta_{\Ss}]}}\ICS_{\Mst(Q)^{\zeta_{\Ss}}_{[0,\theta_{\Ss}]}}(\mathbb{Q})\right)_{\nilp}\right)
\end{equation}
is pure, of Tate type.  Furthermore, the monodromic mixed Hodge module
\begin{equation}
\label{stacMon}
\Ho\left(\Dim_* \left(\phim{\WWW^{\zeta_{\Ss}}_{[0,\theta_{\Ss}]}}\QQ_{\Mst(Q)^{\zeta_{\Ss}}_{[0,\theta_{\Ss}]}}\right)_{\nilp}\right)
\end{equation}
has trivial monodromy.
\end{proposition}
\begin{proof}
We prove the proposition by induction on the length of $\Ss$.  The result is clearly true for $\Ss=\emptyset$, since then $\mathcal{F}_{\Ss}$ contains only the zero module, and (\ref{stacPure}) and (\ref{stacMon}) are isomorphic to $\mathbb{Q}_{\{0\}}$, the constant pure mixed Hodge module supported at the origin $0\in\mathbb{N}^m$.  As in Proposition \ref{TPure} we set $\gamma=\Mu^{\zeta_{\Ss}}(\shave)$.  By the proof of Proposition \ref{TPure} (see isomorphism (\ref{Glout})), $\Ho\left(\Dim_*\ICS_{\Mst(Q,W)_{\gamma}^{\zeta_{\Ss}\sst,\nilp}}(\mathbb{Q})\right)$ is pure.  So if $\Ss$ is additive, purity follows from the isomorphism (\ref{indIso1}) and the inductive hypothesis, since $\boxtimes_{+}^{\tw}$ preserves purity, by Proposition \ref{weightPreserve} and Lemma \ref{finLem}.  On the other hand, if $\Ss$ is subtractive, then impurity of 
\begin{equation}
\label{noPrime}
\Ho\left(\Dim_* \left(\phim{\WWW^{\zeta_{\Ss'}}_{[0, \theta_{\Ss'}]}}\ICS_{\Mst(Q)^{\zeta_{\Ss'}}_{[0,\theta_{\Ss'}]}}(\mathbb{Q})\right)_{\nilp}\right)
\end{equation}
or its failure to be of Tate type, is implied by impurity of
\begin{equation}
\label{primeCase}
\Ho\left(\Dim_* \left(\phim{\WWW^{\zeta_{\Ss}}_{[0, \theta_{\Ss}]}}\ICS_{\Mst(Q)^{\zeta_{\Ss}}_{[0,\theta_{\Ss}]}}(\mathbb{Q})\right)_{\nilp}\right)
\end{equation}
or its failure to be of Tate type, since there is an inclusion of the monoidal unit 
\[
\mathbb{Q}_{\{0\}}\subset \Ho\left(\Dim_{\gamma,*}^{\zeta_{\Ss}\sst}\ICS_{\Mst(Q,W)_{\gamma}^{\zeta_{\Ss}\sst,\nilp}}(\mathbb{Q})\right)
\]
as a direct summand, and so (\ref{indIso2}) implies there is an inclusion $(\ref{primeCase})\subset(\ref{noPrime})$ as a direct summand.  So purity, of Tate type, follows again from the inductive hypothesis.  The monodromy statement is proved in exactly the same way, using the monodromy statement of Proposition \ref{TPure} in the inductive step, and the modified isomorphisms of Remark \ref{untwist} (in which no half Tate twists appear).
\end{proof}
The next theorem is a modification, in the sense elaborated upon in Remark \ref{differenceRem}, of a conjecture of Kontsevich and Efimov, stated as Conjecture 6.8 of \cite{Efi11}.  In addition to proving the conjecture, we prove that the relevant monodromic mixed Hodge module is of Tate type, and after a half Tate twist determined by $\dd$, it has trivial monodromy.
\begin{theorem}
\label{KConj}
Let $\ff\in\mathbb{N}^n$ be a framing vector, and let $\dd\in \mathbb{N}^m$ be a dimension vector satisfying $\Mu^{\zeta_{\Ss}}(\dd)\in[0,\theta_{\Ss}]$.  The monodromic mixed Hodge module 
\[
\mathcal{H}=\Ho\left(\dim_*\left(\phim{\WW^{\zeta_{\Ss},\theta_{\Ss}\sfr}_{\ff,\dd}}\ICS_{\Msp(Q)^{\zeta_{\Ss},\theta_{\Ss}\sfr}_{\ff,\dd}}(\mathbb{Q})\right)_{\nilp}\right)\
\]
is pure, of Tate type, and admits a Lefschetz operator $l\colon\mathcal{H}\rightarrow\mathcal{H}[2]$ such that for all $k\in\mathbb{N}$, $l^k\colon\mathcal{H}^{-k}\rightarrow\mathcal{H}^k$ is an isomorphism.  Furthermore, the monodromic mixed Hodge module
\[
\Ho\left(\dim_*\left(\phim{\WW^{\zeta_{\Ss},\theta_{\Ss}\sfr}_{\ff,\dd}}\QQ_{\Msp(Q)^{\zeta_{\Ss},\theta_{\Ss}\sfr}_{\ff,\dd}}\right)_{\nilp}\right)
\]
has trivial monodromy.
\end{theorem}
\begin{proof}
We write $\mathbb{N}^{(Q_{\ff,\princ})_0}=\mathbb{N}^{m+1}$, where the identification is via the ordering $(\infty,1,\ldots,m)$ of the principal vertices of $Q_{\ff}$.  Let $\xi=\zeta_{\Ss}^{(\theta_{\Ss})}$ be as in Definition \ref{ztDef}.  Set $\kappa=\Mu^{\xi}(1_\infty)$.  Recall that, by construction, $\kappa$ is slightly larger than $\theta_{\Ss}$.  By Theorem \ref{rays}, there is an isomorphism in $\Dulf(\MMHM(\mathbb{N}^{m+1}))$
\begin{align}
\label{WC}
&\Ho\left(\Dim_*\left(\phim{\WWW^{\xi}_{[0,\kappa]}}\ICS_{\Mst(Q_{\ff})_{[0,\kappa]}^{\xi}}(\mathbb{Q})\right)_{\nilp}\right)\cong\\&\Boxtimes_{+,[\kappa\xrightarrow{\gamma}0]}^{\tw}\Ho\left(\Dim_*\left(\phim{\WWW_{\gamma}^{\xi\sst}}\ICS_{\Mst(Q_{\ff})_{\gamma}^{\xi\sst}}(\mathbb{Q})\right)_{\nilp}\right),\nonumber
\end{align}
where we have again commuted the operations of passing to total cohomology and restricting to the nilpotent locus via Corollary \ref{allW}.  Since 
\[
\Ho\left(\Dim_*\phim{\WWW_{0}^{\xi\sst}}\ICS_{\Mst(Q_{\ff})_{0}^{\xi\sst}}(\mathbb{Q})\right)\cong\mathbb{Q}_{\{0\}},
\]
the constant pure mixed Hodge module on the point $0$, we deduce that for each $\gamma\in[0,\kappa]$,
\begin{equation}
\label{po}
\Ho\left(\Dim_*\left(\phim{\WWW_{\gamma}^{\xi\sst}}\ICS_{\Mst(Q_{\ff})_{\gamma}^{\xi\sst}}(\mathbb{Q})\right)_{\nilp}\right)
\end{equation} 
is a direct summand of the left hand side of (\ref{WC}), since there is an isomorphism
\begin{align*}
&\Ho\left(\Dim_*\left(\phim{\WWW_{\gamma}^{\xi\sst}}\ICS_{\Mst(Q_{\ff})_{\gamma}^{\xi\sst}}(\mathbb{Q})\right)_{\nilp}\right)\cong \\&\left(\Boxtimes^{\tw}_{+,[\kappa\xrightarrow{\gamma'}\gamma)}\mathbb{Q}_{\{0\}}\right)\boxtimes_+^{\tw}\Ho\left(\Dim_*\left(\phim{\WWW_{\gamma}^{\xi\sst}}\ICS_{\Mst(Q_{\ff})_{\gamma}^{\xi\sst}}(\mathbb{Q})\right)_{\nilp}\right)\boxtimes^{\tw}_+ \left(\Boxtimes^{\tw}_{+,(\gamma\xrightarrow{\gamma'}0]}\mathbb{Q}_{\{0\}}\right).
\end{align*}
So purity of (\ref{po}) is implied by the purity of the left hand side of (\ref{WC}), which we now demonstrate.

Define a new stability condition $\xi'\in\mathbb{H}_+^{Q_{\ff}}$ by setting $\xi'|_{Q_0}=\zeta_{\Ss}=\xi|_{Q_0}$, and $\xi'_{\infty}=1$.  Note that, by construction of $\zeta_{\Ss}$ (in particular (\ref{strictStab})), the imaginary part of $\zeta_{\Ss}\cdot\dd$ is greater than zero, for all $\dd\in\mathbb{N}^{Q_0}\setminus\{0\}$, and so with respect to the stability condition $\xi'$ any $\mathbb{C}Q_{\ff}$-module that is not entirely supported at the vertex $\infty$ has strictly greater slope than a $\mathbb{C}Q_{\ff}$-module supported entirely at $\infty$.  In particular, any $\mathbb{C}Q_{\ff}$-module which is supported both at the vertex $\infty$ and on the original quiver $Q$ is destabilised by its underlying $\mathbb{C}Q$-module.

There is an equality
\begin{align}\label{bigWall}
&\Ho\left(\Dim_*\left(\phim{\WWW^{\xi}_{[0,\kappa]}}\ICS_{\Mst(Q_{\ff})_{[0,\kappa]}^{\xi}}(\mathbb{Q})\right)_{\nilp}\right)=\\&\Ho\left(\Dim_{*}\left(\phim{\WWW^{\xi'}_{[0,\kappa]}}\ICS_{\Mst(Q_{\ff})_{[0,\kappa]}^{\xi'}}(\mathbb{Q})\right)_{\nilp}\right)\nonumber
\end{align}
since a $\mathbb{C}Q_\ff$-module belongs to $(\rmod{\mathbb{C}Q_\ff})^{\xi}_{[0,\kappa]}$, equivalently $(\rmod{\mathbb{C}Q_\ff})^{\xi'}_{[0,\kappa]}$, if and only if the underlying $\mathbb{C}Q$-module belongs to $(\rmod{\mathbb{C}Q})^{\zeta_{\Ss}}_{[0,\theta_{\Ss}]}$.  Applying Theorem \ref{rays} again, there are isomorphisms in the category $\Dulf(\MMHM(\mathbb{N}^{m+1}))$
\begin{align}\label{bigWCF}
&\Ho\left(\Dim_*\left(\phim{\WWW^{\xi'}_{[0,\kappa]}}\ICS_{\Mst(Q_{\ff})_{[0,\kappa]}^{\xi'}}(\mathbb{Q})\right)_{\nilp}\right)\\
\cong&\Boxtimes_{+,[\kappa\xrightarrow{\gamma}0]}^{\tw}\Ho\left(\Dim_*\left(\phim{\WWW_{\gamma}^{\xi'\sst}}\ICS_{\Mst(Q_{\ff})_{\gamma}^{\xi'\sst}}(\mathbb{Q})\right)_{\nilp}\right)\nonumber
\\
\cong & \Boxtimes_{+,[\theta_{\Ss}\xrightarrow{\gamma}0]}^{\tw}\Ho\left(\Dim_*\left(\phim{\WWW_{\gamma}^{\zeta_{\Ss}\sst}}\ICS_{\Mst(Q)_{\gamma}^{\zeta_{\Ss}\sst}}(\mathbb{Q})\right)_{\nilp}\right)\boxtimes_{+}^{\tw}\nonumber
\\
& \boxtimes_+^{\tw}\Ho\left(\Dim_*\ICS_{\Mst(Q_{\ff})_{(\mathbb{N},0,\ldots,0)}}(\mathbb{Q})\right)\nonumber
\\
\cong &\Ho\left(\Dim_*\left(\phim{\WWW_{[0,\theta_{\Ss}]}^{\zeta_{\Ss}}}\ICS_{\Mst(Q)_{[0,\theta_{\Ss}]}^{\zeta_{\Ss}}}(\QQ)\right)_{\nilp}\right)\boxtimes_{+}^{\tw}\nonumber
\\
&\boxtimes_+^{\tw}\Ho\left(\Dim_*\ICS_{\Mst(Q_{\ff})_{(\mathbb{N},0,\ldots,0)}}(\mathbb{Q})\right).\nonumber
\end{align}
Here the term $\Mst(Q_{\ff})_{(\mathbb{N},0,\ldots,0)}$ is the stack of finite-dimensional $\mathbb{C}Q_{\ff}$-modules having dimension vector $0$ when restricted to $\mathbb{C}Q$.  This, in turn, is the moduli stack 
\[
\coprod_{r\geq 0} (\pt/\Gl_r),
\]
which has pure cohomology, of Tate type.  As such, the term 
\[
\Ho\left(\Dim_*\ICS_{\Mst(Q_{\ff})_{(\mathbb{N},0\ldots,0)}}(\mathbb{Q})\right)
\]
is pure, of Tate type, as is 
\[
\Ho\left(\Dim_*\left(\phim{\WWW_{[0,\theta_{\Ss}]}^{\zeta_{\Ss}}}\ICS_{\Mst(Q)_{[0,\theta_{\Ss}]}^{\zeta_{\Ss}}}(\QQ)\right)_{\nilp}\right)
\]
by Proposition \ref{stackPure}.  Since $\boxtimes_{+}^{\tw}$ takes pairs of pure objects to pure objects by Proposition \ref{weightPreserve}, we finally deduce that (\ref{bigWCF}) is pure, of Tate type.  It then follows that the left hand side of (\ref{WC}) is pure, of Tate type, via the equality (\ref{bigWall}).  

Now fix $\dd$ satisfying $\Mu^{\zeta_{\Ss}}(\dd)\in[0,\theta_{\Ss}]$.  We have shown that 
\begin{align}\label{extraq}
&\Ho\left(\Dim_*\left(\phim{\WWW_{(1,\dd)}^{\xi\sst}}\ICS_{\Mst(Q_\ff)_{(1,\dd)}^{\xi\sst}}(\mathbb{Q})\right)_{\nilp}\right)\cong \\ \nonumber
&\Ho\left(\dim_*\left(\phim{\WW_{\ff,\dd}^{\zeta_{\Ss},\theta_{\Ss}\sfr}}\ICS_{\Msp(Q)_{\ff,\dd}^{\zeta_{\Ss},\theta_{\Ss}\sfr}}(\QQ)\right)_{\nilp}\right)\otimes\HO(\BC,\QQ)_{\vir}
\end{align}
is pure, of Tate type, which gives the purity (of Tate type) of
\begin{equation}
\label{des}
\Ho\left(\dim_*\left(\phim{\WW_{\ff,\dd}^{\zeta_{\Ss},\theta_{\Ss}\sfr}}\ICS_{\Msp(Q)_{\ff,\dd}^{\zeta_{\Ss},\theta_{\Ss}\sfr}}(\QQ)\right)_{\nilp}\right).
\end{equation}
The monodromy statement is proved in exactly the same way, using the monodromy statement of Proposition \ref{stackPure}.  

By Proposition \ref{eqSchemes}, the cohomology of (\ref{des}) is the cohomology of a restriction of 
\[
\phim{\WW_{\ff,\dd}^{\zeta_{\Ss},\theta_{\Ss}\sfr}}\ICS_{\Msp(Q)_{\ff,\dd}^{\zeta_{\Ss},\theta_{\Ss}\sfr}}(\QQ)
\]
to a proper union of connected components of its support.  As explained in \cite[Thm.2.3]{DMSS13}, using the machinery of \cite{Sa88} and the fact that this is an example of case (c) of \cite[Thm.2.1]{DMSS13}, we deduce that 
\[
\Ho\left(\dim_*\left(\phim{\WW_{\ff,\dd}^{\zeta_{\Ss},\theta_{\Ss}\sfr}}\ICS_{\Msp(Q)_{\ff,\dd}^{\zeta_{\Ss},\theta_{\Ss}\sfr}}(\QQ)\right)_{\nilp}\right)
\]
carries a Lefschetz operator, as required.\end{proof}
\section{Proof of the main theorem}
\label{mainProof}

\label{cluVia}
We recall from \cite{Efi11} the passage from Kontsevich's conjecture to the quantum cluster positivity conjecture.  From now on we leave out the symbols $\Ho(\ldots)$; we now have purity of all relevant mixed Hodge modules, and so all mixed Hodge modules that are well-defined before passing to total cohomology are nonetheless isomorphic to their total cohomology.  
\renewcommand*{\proofname}{Proof of Theorem \ref{mainThm}}
\begin{proof}
As in the proof of Theorem \ref{KConj}, we use Definition \ref{ztDef} and set $\xi=\zeta_{\Ss}^{(\theta_{\Ss})}$.  Define
\[
R:=\{\dd\in\mathbb{N}^m|\Mu^{\zeta_{\Ss}}(\dd)\in[0,\theta_{\Ss}]\}.  
\]
Combining isomorphisms (\ref{WC}) (\ref{bigWall}) and (\ref{bigWCF}) from the proof of Theorem \ref{KConj}, and restricting to $\{1\}\times R$, there is an isomorphism
\begin{align}\label{WCD}
&\Dim_*\left(\phim{\WWW^{\zeta_{\Ss}}_{[0,\theta_{\Ss}]}}\ICS_{\Mst(Q)_{[0, \theta_{\Ss}]}^{\zeta_{\Ss}}}(\mathbb{Q})\right)_{\nilp}\boxtimes_{+}^{\tw}\Dim_*\ICS_{\Mst(Q_{\ff})_{(1,0,\ldots,0)}}(\mathbb{Q})\cong\\
&\Dim_*\left(\phim{\WWW^{\xi\sst}_{\{1\}\times R}}\ICS_{\Mst(Q_{\ff})_{\{1\}\times R}^{\xi\sst}}(\mathbb{Q})\right)_{\nilp}\boxtimes^{\tw}_{+} \Dim_*\left(\phim{\WWW^{\zeta_{\Ss}}_{[0, \theta_{\Ss}]}}\ICS_{\Mst(Q)_{[0, \theta_{\Ss}]}^{\zeta_{\Ss}}}(\mathbb{Q})\right)_{\nilp}.\nonumber
\end{align}
There is an equality
\[
\Dim_*\ICS_{\Mst(Q_{\ff})_{(1,0,\ldots,0)}}(\mathbb{Q})=(q^{-1/2}-q^{1/2})^{-1}Y^{1_{\infty}}
\]
arising from the identity 
\[
\chi_q\left(\HO(\BC,\mathbb{Q})_{\vir},q^{1/2}\right)=-(q^{-1/2}-q^{1/2})^{-1}
\]
(recall from (\ref{IM}) the sign change in the definition of $\chi_{Q_{\ff}}$).
Applying $\chi_{Q_{\ff}}$ to the identity in $\KK(\Dulf(\MMHM(\mathbb{N}^{m+1})))$ resulting from (\ref{WCD}), and multiplying both sides on the right by 
\[
\left(q^{-1/2}-q^{1/2}\right)\chi_{Q_\ff}\left(  \Dim_*\left(\phim{\WWW^{\zeta_{\Ss}}_{[0, \theta_{\Ss}]}}\ICS_{\Mst(Q)_{[0, \theta_{\Ss}]}^{\zeta_{\Ss}}}(\mathbb{Q})\right)_{\nilp} \right)^{-1}, 
\]
we obtain the identity 
\begin{align}\label{finalId}
&\chi_{Q_{\ff}}\left(  \Dim_*\left(\phim{\WWW^{\zeta_{\Ss}}_{[0,\theta_{\Ss}]}}\ICS_{\Mst(Q)_{[0,\theta_{\Ss}]}^{\zeta_{\Ss}}}(\mathbb{Q})\right)_{\nilp} \right)Y^{1_\infty}\\ \nonumber&\chi_{Q_{\ff}}\left(  \Dim_*\left(\phim{\WWW^{\zeta_{\Ss}}_{[0,\theta_{\Ss}]}}\ICS_{\Mst(Q)_{[0,\theta_{\Ss}]}^{\zeta_{\Ss}}}(\mathbb{Q})\right)_{\nilp} \right)^{-1}=\\
&\sum_{\dd\in R}\chi_{Q_{\ff}}\left( \dim_*\left(\phim{\WW_{\ff,\dd}^{\zeta_{\Ss},\theta_{\Ss}\sfr}}\ICS_{\Msp(Q)_{\ff,\dd}^{\zeta_{\Ss},\theta_{\Ss}\sfr}}(\QQ)\right)_{\nilp} \right).\nonumber
\end{align}
Here we have used the identity
\begin{align*}
&\sum_{\dd\in R}\chi_{Q_{\ff}}\left( \dim_*\left(\phim{\WW_{\ff,\dd}^{\zeta_{\Ss},\theta_{\Ss}\sfr}}\ICS_{\Msp(Q)_{\ff,\dd}^{\zeta_{\Ss},\theta_{\Ss}\sfr}}(\QQ)\right)_{\nilp} \right)(q^{-1/2}-q^{1/2})^{-1}=\\&\chi_{Q_{\ff}}\left( \Ho\left(\Dim_*\left(\phim{\WWW_{(1,\dd)}^{\xi\sst}}\ICS_{\Mst(Q_\ff)_{(1,\dd)}^{\xi\sst}}(\mathbb{Q})\right)_{\nilp}\right)  \right)
\end{align*}
arising from (\ref{extraq}).

We extend the homomorphism $\iota\colon\hat{\Ror}_Q\rightarrow \hat{\Tor}_{\Lambda}$ to a homomorphism $\iota_{\ff}\colon\hat{\Ror}_{Q_{\ff}}\rightarrow\hat{\Tor}_{\Lambda}$ by sending $Y^{1_\infty}$ to $X^{\ff}$.  The map $\iota_{\ff}$ is indeed a ring homomorphism, as can be verified via the relations (\ref{differentStrokes}) and the calculation
\begin{align*}
\langle 1_{\infty},\dd\rangle_{Q_{\ff}}=&-\dd\cdot \ff\\
=&\Lambda(-\tilde{B}\dd,\ff),
\end{align*}
where the second equality follows from (\ref{compatibility}).  The following identity, expressing mutated cluster variables in terms of vanishing cycle cohomology, then follows from (\ref{finalId}) and Theorem \ref{sfrThm}:
\begin{equation}
\label{finsf}
\mu_{\Ss}(M)(\ff)=\sum_{\dd\in R}\iota_{\ff}\chi_{Q_{\ff}}\left( \dim_*\left(\phim{\WW_{\ff,\dd}^{\zeta_{\Ss},\theta_{\Ss}\sfr}}\ICS_{\Msp(Q)_{\ff,\dd}^{\zeta_{\Ss},\theta_{\Ss}\sfr}}(\QQ)\right)_{\nilp}\right).
\end{equation}
By Theorem \ref{KConj}, the mixed Hodge module on the right hand side of (\ref{finsf}) is pure and carries a Lefschetz operator.  Positivity, and the Lefschetz property then follow.  

Finally, note that each of the nonzero polynomials $a_\dd(q^{1/2})$ appearing in the theorem is given by the weight polynomial of the single monodromic mixed Hodge structure 
\begin{equation*}
\label{TF}
\dim_*\left(\phim{\WW_{\ff,\dd'}^{\zeta_{\Ss},\theta_{\Ss}\sfr}}\ICS_{\Msp(Q)_{\ff,\dd'}^{\zeta_{\Ss},\theta_{\Ss}\sfr}}(\QQ)\right)_{\nilp}
\end{equation*}
for $\dd'$ satisfying $\iota(\dd')+\ff=\dd$, since $\iota$ is injective by (\ref{compatibility}).  Since the monodromic mixed Hodge structure
\[
\dim_*\left(\phim{\WW_{\ff,\dd'}^{\zeta_{\Ss},\theta_{\Ss}\sfr}}\QQ_{\Msp(Q)_{\ff,\dd'}^{\zeta_{\Ss},\theta_{\Ss}\sfr}}\right)_{\nilp}
\]
is pure, of Tate type, with trivial monodromy by Theorem \ref{KConj}, we deduce that 
\begin{align*}
&\chi_q\left(\dim_*\left(\phim{\WW_{\ff,\dd'}^{\zeta_{\Ss},\theta_{\Ss}\sfr}}\ICS_{\Msp(Q)_{\ff,\dd'}^{\zeta_{\Ss},\theta_{\Ss}\sfr}}(\QQ)\right)_{\nilp},q^{1/2}\right)\\&=\chi_q\left(\LL^{-\dim(\Msp(Q)_{\ff,\dd'}^{\zeta_{\Ss},\theta_{\Ss}\sfr})},q^{1/2}\right)\chi_q\left(\dim_*\left(\phim{\WW_{\ff,\dd'}^{\zeta_{\Ss},\theta_{\Ss}\sfr}}\mathbb{Q}_{\Msp(Q)_{\ff,\dd'}^{\zeta_{\Ss},\theta_{\Ss}\sfr}}\right)_{\nilp},q^{1/2}\right)\\&=(-q^{1/2})^{-\dim(\Msp(Q)_{\ff,\dd'}^{\zeta_{\Ss},\theta_{\Ss}\sfr})}h(q)
\end{align*}
for $h(q)\in\mathbb{N}[q]$.  The sign before the half power of $q$ in the final line is as in Remark \ref{Lsign}, and is cancelled by definition of the map $\chi_{Q_{\ff}}$ (see (\ref{dulfjust})).  Combining this statement with the Lefschetz property finishes the proof of Theorem \ref{mainThm}.
\end{proof}
\renewcommand*{\proofname}{Proof}
By combining Theorem \ref{mainThm} and Remark \ref{quantization} we recover the following corollary, which is the classical positivity theorem, due to Lee and Schiffler.
\begin{theorem}\cite{LS15}
Let $Q$ be a quiver.  Then the classical positivity conjecture holds for the cluster algebra $\mathcal{A}_Q$.
\end{theorem}

\appendix
\section{No exotics property for motivic Donaldson--Thomas invariants}
In this appendix we prove a theorem related to the quantum cluster positivity theorem, regarding Donaldson--Thomas invariants for cluster collections.  Let $Q$ be a quiver without loops and 2-cycles, and let $W\in\mathbb{C}Q/[\mathbb{C}Q,\mathbb{C}Q]$ be an algebraic potential.  For the purposes of this section we assume that $Q$ has no frozen vertices, i.e. we identify the vertices of the quiver with the numbers $\{1,\ldots,m\}$, and place no restrictions apart from finiteness on the dimension vector $\dd$ of our $\mathbb{C}Q$-modules.  The bilinear form $\Lambda$ plays no role in this section, and so in particular we do not require that the coefficients of $\tilde{B}^{-1}$ are integral, or indeed that $\tilde{B}^{-1}$ exists.  Let $\zeta_{\Ss}\in\mathbb{H}_+^m$ be a stability condition, and let $\theta_{\Ss}$ be a slope, satisfying the conditions of Proposition \ref{zetaTheta}, though without the stipulation on the slope $\Mu^{\zeta_{\Ss}}(\shave)$ of $\shave$.

We assume that $\zeta_{\Ss}$ is generic, in the sense that if $\dd$ and $\dd'$ are two dimension vectors satisfying $\Mu^{\zeta_{\Ss}}(\dd)=\Mu^{\zeta_{\Ss}}(\dd')\leq \theta_{\Ss}$, then $\langle \dd,\dd'\rangle_Q=0$.  For an arbitrary stability condition satisfying the conditions of Proposition \ref{zetaTheta} this can be achieved, for example, by perturbing $\zeta_{\Ss}$ within the space of stability conditions satisfying the conditions of Proposition \ref{zetaTheta}, so that $\Mu^{\zeta_{\Ss}}(\dd)=\Mu^{\zeta_{\Ss}}(\dd')$ if and only if $\dd=r\dd'$ for some $r\in\mathbb{R}_{>0}$.  For $\gamma\in [0,\pi)$ we define
\[
\Lambda_{\gamma}^{\zeta_{\Ss}}=\{\dd\in \mathbb{N}^m\setminus\{0\}|\Mu^{\zeta_{\Ss}}(\dd)=\gamma\}\cup \{0\},
\]
i.e. $\Lambda_{\gamma}^{\zeta_{\Ss}}$ is the monoid of dimension vectors of slope $\gamma$ with respect to the stability condition $\zeta_{\Ss}$.  Then for $\gamma\in [0,\theta_{\Ss}]$, restricting the twisted product $\boxtimes_+^{\tw}$ to $\Dulf(\MMHM(\Lambda_{\gamma}^{\zeta_{\Ss}}))$, it becomes a symmetric monoidal product, as the Tate twist $\LL^{\langle \dd',\dd''\rangle_Q/2}$ is trivial for $\dd',\dd''\in\Lambda_{\gamma}^{\zeta_{\Ss}}$.  We define
\[
\Lambda^{\zeta_{\Ss},+}_{\gamma}:=\Lambda_{\gamma}^{\zeta_{\Ss}}\setminus \{0\}.
\]
Given $\mathcal{F}\in\Dulf(\MMHM(\Lambda^{\zeta_{\Ss},+}_{\gamma}))$, we define $\Sym_{\boxtimes_+}(\mathcal{F})\in\Dulf(\MMHM(\Lambda^{\zeta_{\Ss}}_{\gamma}))$ to be the free symmetric unital algebra generated by $\mathcal{F}$ in the category $\Dulf(\MMHM(\Lambda_{\gamma}^{\zeta_{\Ss}}))$.  We define the plethystic exponential
\begin{align*}
\EXP^{\Hodge}\colon\KK(\Dulf(\MMHM(\Lambda_{\gamma}^{\zeta_{\Ss},+})))\rightarrow &\KK(\Dulf(\MMHM(\Lambda_{\gamma}^{\zeta_{\Ss}})))\\
[\mathcal{F}]\mapsto &[\Sym_{\boxtimes_+}(\mathcal{F})].
\end{align*}
The map $\EXP^{\Hodge}$ is an isomorphism onto its image, which is $1+\KK(\Dulf(\MMHM(\Lambda_{\gamma}^{\zeta_{\Ss},+})))$.  It is also a lift of the map $\EXP$ of Definition \ref{plethDef}, in the sense that 
\[
\EXP\circ\chi_Q|_{\KK(\Dulf(\MMHM(\Lambda_{\gamma}^{\zeta_{\Ss},+})))}=\chi_Q\circ\EXP^{\Hodge}.  
\]
We define $\hat{\Ror}^{\Hodge}_Q$ to be the free $\KK(\Dulf(\MMHM(\pt)))$-module generated by symbols $Y^{\ee}$, with $\ee\in\mathbb{N}^m$, and with multiplication defined by 
\[
[\mathcal{G}]Y^{\ee}\cdot [\mathcal{G}']Y^{\dd}=[\LL^{\langle \dd,\ee\rangle_Q/2}\otimes \mathcal{G}\otimes\mathcal{G}']Y^{\ee+\dd},
\]
completed with respect to the ideal generated by $Y^{\ee}$ for $\ee\in\mathbb{N}^m\setminus \{0\}$.  We define the isomorphism
\begin{align*}
\chi^{\Hodge}_Q:\KK\left(\Dulf(\MMHM(\mathbb{N}^m))\right)\rightarrow&\hat{\Ror}^{\Hodge}_Q\\
[\mathcal{F}]\mapsto&\sum_{\dd\in\mathbb{N}^m}[\mathcal{F}_{\dd}]Y^{\dd}.
\end{align*}

Let $\gamma\in [0,\theta_{\Ss}]$.  The Hodge-theoretic Donaldson--Thomas invariants $\Omega_{\dd}^{\zeta_{\Ss}}\in\KK(\Dulf(\MMHM(\pt)))$ for the category of $\zeta_{\Ss}$-semistable $\Jac(Q,W)$-modules of slope $\gamma$ are defined to be the classes satisfying
\begin{equation}
\chi_Q^{\Hodge}\left(\left[\Ho\left(\Dim_{\gamma,*}^{\zeta_{\Ss}\sst}\left(\phim{\WWW^{\zeta_{\Ss}\sst}_{\gamma}}\ICS_{\Mst(Q)^{\zeta_{\Ss}\sst}_{\gamma}}(\mathbb{Q})\right)\right)\right]\right)= \EXP^{\Hodge}\left(\frac{\sum_{\dd\in\Lambda^{\zeta_{\Ss},+}_{\gamma}}\Omega^{\zeta_{\Ss}}_{\dd}Y^{\dd}}{[\mathbb{L}^{-1/2}]-[\mathbb{L}^{1/2}]}\right),
\end{equation}
where we define the right hand side via the identification 
\[
\KK(\Dulf(\MMHM(\Lambda_{\gamma}^{\zeta_{\Ss}})))=\KK(\Dulf(\MMHM(\pt)))[[Y^{\ee}|\ee\in \Lambda^{\zeta_{\Ss}}_{\gamma}]]
\]
induced by $\chi_Q^{\Hodge}$, and the expansion
\[
1/([\LL^{-1/2}]-[\LL^{1/2}])=[\LL^{1/2}]+[\LL^{3/2}]+\ldots.
\]

Similarly, the DT invariants $\Omega_{\dd}^{\zeta_{\Ss},\nilp}$ are defined by the equation
\begin{equation}
\label{niceNilp}
\chi_Q^{\Hodge}\left(\left[\Ho\left(\Dim_{\gamma,*}^{\zeta_{\Ss}\sst}\left(\phim{\WWW^{\zeta_{\Ss}\sst}_{\gamma}}\ICS_{\Mst(Q)^{\zeta_{\Ss}\sst}_{\gamma}}(\mathbb{Q})\right)_{\nilp}\right)\right]\right)= \EXP^{\Hodge}\left(\frac{\sum_{\dd\in\Lambda^{\zeta_{\Ss},+}_{\gamma}}\Omega^{\zeta_{\Ss},\nilp}_{\dd}Y^{\dd}}{[\mathbb{L}^{-1/2}]-[\mathbb{L}^{1/2}]}\right).
\end{equation}
\begin{remark}\label{VerSt}
Strictly speaking, the correct formulation of the second definition is
\begin{equation}
\label{badNilp}
\chi_Q^{\Hodge}\left(\left[\DD_{\Lambda_{\gamma}^{\zeta_{\Ss}}}^{\mon}\Ho\left(\Dim_{\gamma,!}^{\zeta_{\Ss}\sst}\left(\phim{\WWW^{\zeta_{\Ss}\sst}_{\gamma}}\ICS_{\Mst(Q)^{\zeta_{\Ss}\sst}_{\gamma}}(\mathbb{Q})\right)_{\nilp}\right)\right]\right)= \EXP^{\Hodge}\left(\frac{\sum_{\dd\in\Lambda^{\zeta_{\Ss},+}_{\gamma}}\Omega^{\zeta_{\Ss},\nilp}_{\dd}Y^{\dd}}{[\mathbb{L}^{-1/2}]-[\mathbb{L}^{1/2}]}\right)
\end{equation}
instead of (\ref{niceNilp}).  However, the left hand sides of (\ref{niceNilp}) and (\ref{badNilp}) are equal, by self-duality of the vanishing cycle complex and Corollary \ref{allW}.
\end{remark}

In the language of motivic Donaldson--Thomas theory \cite{KS}, the classes $\Omega_{\dd}^{\zeta_{\Ss}}$ and $\Omega_{\dd}^{\zeta_{\Ss},\nilp}$ are the Hodge--theoretic realisations of the respective motivic Donaldson--Thomas invariants, as explained in \cite{COHA}.  As in \cite{DaMe4} and \cite{DaMe15b}, we define
\[
\DT_{\gamma}^{\zeta_{\Ss}}\in\Db(\MMHM(\Lambda_{\gamma}^{\zeta_{\Ss},+}))
\]
by the condition that, for $\dd\in\Lambda_{\gamma}^{\zeta_{\Ss},+}$, 
\begin{equation*}
\DT_{\dd}^{\zeta_{\Ss}}=\begin{cases}\Ho\left(\tau_*\phim{\WW^{\zeta_{\Ss}\sst}_{\dd}}\ICS_{\Msp(Q)_{\dd}^{\zeta_{\Ss}\sst}}(\QQ)\right)& \textrm{if }\Msp(Q)_{\dd}^{\zeta_{\Ss}\st}\neq\emptyset\\ 0&\textrm{otherwise,}\end{cases}
\end{equation*}
where $\tau:\Msp(Q)^{\zeta_{\Ss}\sst}_{\dd}\rightarrow\pt$ is the map to a point.  Similarly, we define
\begin{equation*}
\DT_{\dd}^{\zeta_{\Ss},\nilp}=\begin{cases}\Ho\left(\tau_*\left(\phim{\WW^{\zeta_{\Ss}\sst}_{\dd}}\ICS_{\Msp(Q)_{\dd}^{\zeta_{\Ss}\sst}}(\QQ)\right)_{\nilp}\right)& \textrm{if }\Msp(Q)_{\dd}^{\zeta_{\Ss}\st}\neq\emptyset\\ 0&\textrm{otherwise.}\end{cases}
\end{equation*}

Then by \cite[Thm.A]{DaMe15b}, there are isomorphisms
\begin{align}
\Ho\left(\Dim_{\gamma,*}^{\zeta_{\Ss}\sst}\left(\phim{\WWW^{\zeta_{\Ss}\sst}_{\gamma}}\ICS_{\Mst(Q)^{\zeta_{\Ss}\sst}_{\gamma}}(\mathbb{Q})\right)\right)\cong &\Sym_{\boxtimes_+}\left(  \DT_{\gamma}^{\zeta_{\Ss}}\otimes\HO(\BC,\QQ)_{\vir} \right)\\
\Ho\left(\Dim_{\gamma,*}^{\zeta_{\Ss}\sst}\left(\phim{\WWW^{\zeta_{\Ss}\sst}_{\gamma}}\ICS_{\Mst(Q)^{\zeta_{\Ss}\sst}_{\gamma}}(\mathbb{Q})\right)_{\nilp}\right)\cong &\Sym_{\boxtimes_+}\left(  \DT_{\gamma}^{\zeta_{\Ss},\nilp}\otimes\HO(\BC,\QQ)_{\vir} \right),\label{nilpver}
\end{align}
from which we deduce that 
\begin{align*}
\Omega^{\zeta_{\Ss}}_{\dd}=&[\DT^{\zeta_{\Ss}}_{\dd}]\\
\Omega^{\zeta_{\Ss},\nilp}_{\dd}=&[\DT^{\zeta_{\Ss},\nilp}_{\dd}].
\end{align*}
We now state our main results regarding Donaldson--Thomas invariants.
\begin{theorem}
\label{KPos}
For $\zeta_{\Ss}$ and $\theta_{\Ss}$ as above, and $\dd\in\mathbb{N}^m$ of slope less than or equal to $\theta_{\Ss}$, the Hodge--theoretic Donaldson--Thomas invariants $\Omega^{\zeta_{\Ss},\nilp}_{\dd}$ can be written as $h_{\dd}(\LL^{1/2})$, for $h_{\dd}(q^{1/2})=h_{\dd}(q^{-1/2})$ equal to $b_{\dd}(q)q^{{}-\deg(b_{\dd}(q))/2}$, for some polynomial $b_{\dd}(q)\in\mathbb{N}[q]$ with unimodal coefficients.
\end{theorem}
This theorem is in turn a consequence of the following one.
\begin{theorem}
\label{cohPos}
For $\zeta_{\Ss},\theta_{\Ss},\dd$ as above, the monodromic mixed Hodge module $\mathcal{H}=\DT^{\zeta_{\Ss},\nilp}_{\dd}$ is pure, of Tate type, and carries a Lefschetz operator $l\colon \mathcal{H}^{\bullet}\rightarrow \mathcal{H}^{\bullet+2}$ such that $l^k\colon \mathcal{H}^{-k}\rightarrow\mathcal{H}^k$ is an isomorphism for all $k$.  Moreover, either $\mathcal{H}$ or $\mathcal{H}\otimes\LL^{1/2}$ has trivial monodromy.
\end{theorem}
\begin{proof}
By Theorem \ref{rays} we may write

\begin{align}\label{DTcalc}
&\Ho\left(\Dim_*\left(\phim{\WWW^{\zeta_{\Ss}}_{[0,\theta_{\Ss}]}}\ICS_{\Mst(Q)_{[0,\theta_{\Ss}]}^{\zeta_{\Ss}}}(\mathbb{Q})\right)_{\nilp}\right)\cong\\&\Boxtimes_{+,[\theta_{\Ss}\xrightarrow{\gamma}0]}^{\tw}\Ho\left(\Dim_*\left(\phim{\WWW^{\zeta_{\Ss}\sst}_{\gamma}}\ICS_{\Mst(Q)^{\zeta_{\Ss}\sst}_{\gamma}}(\mathbb{Q})\right)_{\nilp}\right)\label{rayCalc}
\end{align}
and by Proposition \ref{stackPure} we moreover deduce that (\ref{DTcalc}) is pure, of Tate type.  It follows as in the proof of Theorem \ref{KConj} that each of the terms in the product (\ref{rayCalc}) is pure, of Tate type.  As each $\DT_{\dd}^{\zeta_{\Ss},\nilp}$ is a summand of (\ref{rayCalc}), we deduce that all of the monodromic mixed Hodge structures $\DT_{\dd}^{\zeta_{\Ss},\nilp}$ are pure, of Tate type.  

In addition, from Corollary \ref{allW} it follows that 
\begin{align*}
\left(\phim{\WW^{\zeta_{\Ss}\sst}_{\dd}}\ICS_{\Msp(Q)_{\dd}^{\zeta_{\Ss}\sst}}(\QQ)\right)_{\nilp}
\end{align*}
is the restriction of 
\begin{align*}
\phim{\WW^{\zeta_{\Ss}\sst}_{\dd}}\ICS_{\Msp(Q)_{\dd}^{\zeta_{\Ss}\sst}}(\QQ)
\end{align*}
to a proper union of components of its support, i.e. the preimage of the origin under the proper map $q^{\zeta_{\Ss}}_\dd\colon  \Msp(Q)^{\zeta_{\Ss}\sst}_{\dd}\rightarrow \Msp(Q)_{\dd}$, and so its cohomology carries a Lefschetz operator, as in the proof of \cite[Thm.2.3]{DMSS13}.  Moreover, since by Proposition \ref{stackPure}, (\ref{DTcalc}) has trivial monodromy (possibly after tensoring by a half Tate twist, depending on $\dd$), we deduce that the same is true of each $\DT_{\dd}^{\zeta_{\Ss},\nilp}$.
\end{proof}
\begin{remark}
The above is a kind of categorified ``no exotics'' statement for the BPS/DT invariants associated to cluster collections --- compare with \cite{GeomEng}, where in the physics context the no exotics property of refined DT invariants is explained by the principal that the cohomology of the spaces of BPS states that they derive from carry a Lefschetz operator, as representations of $\mathfrak{sl}_2$.  We use the word ``categorified'' here to mean that for cohomological DT invariants coming from cluster collections, we can construct the Lefschetz action itself, in addition to deducing the no exotics property on the underlying refined DT invariants, which in our context is the Tate type property.
\end{remark}
\begin{corollary}\label{firstCor}
Let $(Q,W)$ be an algebraic QP, such that there is a sequence of vertices $\Ss$, for which $W$ is nondegenerate, and for which $\mathcal{F}_{\Ss}\cap \rmod{\HJac(Q,W)}=\rmod{\HJac(Q,W)}$.  Then for a generic stability condition $\zeta$, the Hodge--theoretic Donaldson--Thomas invariants $\Omega^{\zeta_{\Ss},\nilp}_{\dd}$ can be written as $h_{\dd}(\LL^{1/2})$, for $h_{\dd}(q^{1/2})=h_{\dd}(q^{-1/2})$ equal to $b_{\dd}(q)q^{{}-\deg(b_{\dd}(q))/2}$, for $b_{\dd}(q)\in\mathbb{N}[q]$ with unimodal coefficients.
\end{corollary}
As a special case, we recover the following result from \cite{MeRe14}; note however that this is not a new proof, as the proof of \cite[Thm.A]{DaMe15b} uses the results of \cite{MeRe14} in an essential way.  For an example of a QP satisfying the conditions of Corollary \ref{firstCor}, and for which $Q$ is not acyclic, see the example worked out after Conjecture 6.8 of \cite{Efi11}.
\begin{corollary}\cite[Cor.1.2]{MeRe14}
Let $Q$ be acyclic.  Then for a generic stability condition $\zeta$ the Hodge--theoretic Donaldson--Thomas invariants $\Omega^{\zeta_{\Ss},\nilp}_{\dd}$ can be written as $h_{\dd}([\LL^{1/2}])$, for $h_{\dd}(q^{1/2})=h_{\dd}(q^{-1/2})$ equal to $b_{\dd}(q)q^{{}-\deg(b_{\dd}(q))/2}$, for some polynomial $b_{\dd}(q)\in\mathbb{N}[q]$ with unimodal coefficients.
\end{corollary}
\bibliographystyle{amsalpha}
\bibliography{SPP}

\vfill

\textsc{\small B. Davison: School of Mathematics and Statistics\\University of Glasgow, University Place\\Glasgow G12 8SQ}\\
\textit{\small E-mail address:} \texttt{\small ben.davison@glasgow.ac.uk}\\
\\

\end{document}